\documentclass[11pt]{amsart}

\usepackage{upgreek}

\usepackage{ytableau}

\usepackage{amsmath, amsxtra, amsthm, amssymb,mathtools,bm,amsfonts}
\usepackage{mathabx}
\usepackage{mathrsfs}
\usepackage[normalem]{ulem}
\usepackage[mathscr]{euscript}
\usepackage{graphicx}
\usepackage{url}
\usepackage{color}
\usepackage{bbm}
\newcommand{\stkout}[1]{\ifmmode\text{\sout{\ensuremath{#1}}}\else\sout{#1}\fi}
\usepackage{tikz-cd}
\usepackage{tikz}
\newcommand{\circled}[1]{
    \tikz[baseline=(char.base)]{
        \node[shape=circle, 
              draw,           
              inner sep=2pt,  
              fill=white      
             ] (char) {$#1$};
    }
}

\usepackage[T1]{fontenc}
\graphicspath{/Figures/}

\usepackage{enumitem}
\usepackage{comment}
\usepackage[pdftex,hidelinks,backref=page]{hyperref}
\hypersetup{
    colorlinks,
    citecolor=magenta,
    filecolor=magenta,
    linkcolor=blue,
    urlcolor=black
}
\usepackage{cleveref}

\usepackage{xcolor}
\usepackage[colorinlistoftodos]{todonotes}
\renewcommand{\emptyset}{\varnothing}

\newcommand{\QQ}{\mathbb Q}
\newcommand{\RR}{\mathbb R}
\newcommand{\CC}{\mathbb C}

\newcommand{\ZZ}{\mathbb Z}

\newcommand{\id}{\mathrm{id}}

\theoremstyle{definition}
\newtheorem{thm}{Theorem}[section]
\newtheorem{cor}[thm]{Corollary}
\newtheorem{lem}[thm]{Lemma}

\newtheorem{prop}[thm]{Proposition}
\newtheorem{defn}[thm]{Definition}

\newtheorem{eg}[thm]{Example}
\newtheorem{rem}[thm]{Remark}
\newtheorem{obs}[thm]{Observation}

\newtheorem{question}[thm]{Question}
\newtheorem{maintheorem}{Theorem}	\newtheorem{fact}[thm]{Fact}

\numberwithin{equation}{section}
\allowdisplaybreaks

\newcommand{\ar}[1]
{{\xrightarrow{#1}}}

\newcommand{\qsym}[2][]{
{\ifx&#1&%
  {\operatorname{QSym}_{#2}}
\else
  {{}^{#1}\!\operatorname{QSym}_{#2}}
\fi}
} 

\newcommand{\eqsym}[2][]{
{\ifx&#1&%
  {\operatorname{EQSym}_{#2}}
\else
  {{}^{#1}\!\operatorname{EQSym}_{#2}}
\fi}
} 

\newcommand{\qseq}[2][]{
{\ifx&#1&%
  {\operatorname{QSeq}_{#2}}
\else
  {{}^{#1}\!\operatorname{QSeq}_{#2}}
\fi}
}

\newcommand{\qsymide}[2][]{
{\ifx&#1&%
  {\operatorname{QSym}_{#2}^+}
\else
  {{}^{#1}\!\operatorname{QSym}_{#2}^+}
\fi}
} 

\newcommand{\eqsymide}[2][]{
{\ifx&#1&%
  {\operatorname{EQSym}_{#2}^+}
\else
  {{}^{#1}\!\operatorname{EQSym}_{#2}^+}
\fi}
} 

\newcommand{\sym}[1]{\operatorname{Sym}_{#1}} 
\newcommand{\compatible}[2][]{
{\ifx&#1&%
  {\mathcal{C}(#2)}
\else
  {\mathcal{C}^{m}(#2)}
\fi}
} 
\newcommand{\suchthat}{\;|\;}

\newcommand{\schub}[1]{\mathfrak{S}_{#1}} 
\newcommand{\des}[1]{\operatorname{Des}(#1)} 

\date{}
\newcommand{\cat}[1]{\operatorname{Cat}_{#1}} 

\newcommand{\idem}{\operatorname{id}} 
\newcommand{\Perm}{\operatorname{Perm}} 
\newcommand{\rope}[1]{\mathsf{R}_{#1}} 

\newcommand{\fl}[1]{\mathrm{Fl}_{#1}}
\newcommand{\GL}{\operatorname{GL}}
\newcommand{\Pl}{\operatorname{Pl}}

\definecolor{ao}{rgb}{0.0, 0.5, 0.0}

\newcommand{\hhmp}{\mathrm{QFl}}

\newcommand{\wt}[1]{\widetilde{#1}} 
\newcommand{\wh}[1]{\widehat{#1}} 

\newcommand{\tl}{\textbf{t}}
\newcommand{\xl}{\textbf{x}}

\newcommand{\NC}{\operatorname{NC}}

\newcommand{\mc}[1]{\mathcal{#1}}
\newcommand{\inv}[1]{\operatorname{Inv}(#1)}
\newcommand{\invnc}[1]{\operatorname{Inv}_{\NC}(#1)} 
\newcommand{\invncR}[1]{\operatorname{Inv}_{\NC}^R(#1)} 

\renewcommand\emph[1]{\textcolor{blue}{\textit{#1}}} 

\newcommand{\nc}{\operatorname{nc}} 
\newcommand{\Sort}{\operatorname{Sort}} 

\newcommand{\movetoid}[2]{\overleftarrow{[#1,#2]}}

\newcommand{\skips}[1]{\operatorname{Skips}(#1)}

\newcommand{\fskips}[1]{\operatorname{FSkips}(#1)}

\newcommand{\uskips}[1]{\operatorname{USkips}(#1)}

\newcommand{\cfl}{\operatorname{CFl}}

\newcommand{\cov}{\operatorname{Covers}}

\newcommand{\fs}{\operatorname{fs}}

\newcommand{\ufs}{\operatorname{ufs}}

\newcommand{\Complex}
{\operatorname{Complex}}

\newcommand{\GKM}
{\operatorname{GKM}}

\newcommand{\Cay}
{\operatorname{Cayley}}

\newcommand{\Simple}{\mathsf{S}} 
\newcommand{\Reflections}{\mathsf{T}} 
\newcommand{\pluckervanishing}[1]{\operatorname{PV}_{#1}}

\newcommand{\Edges}[1]{\mc{E}(#1)} 
\newcommand{\Maxchains}[1]{\mc{M}(#1)} 

\newcommand{\sort}{\pi_\downarrow}

\newcommand{\upsort}{\pi_\uparrow}

\newcommand{\Xc}{\mathring{X}}

\newcommand{\Q}[1]{\mathrm{Char}(#1)}
\newcommand{\Qvee}[1]{\mathrm{Char}^{\vee}(#1)}

\newcommand{\nctocluster}{{\operatorname{Clust}^+}}
\newcommand{\regweight}{{\lambda_{\mathrm{reg}}}}
\newcommand{\wnaught}{w_{\circ}}
\newcommand{\graph}{\mathcal{G}} 
\newcommand{\Cone}[1]{\operatorname{Cone}(#1)}

\definecolor{col1}{RGB}{210,90,90}
\definecolor{col2}{RGB}{90,190,90}
\definecolor{col3}{RGB}{90,110,210}

\title{The Coxeter Flag Variety}

\author{Nantel Bergeron}
\address{Dept. of Math. and Stat., York University, Toronto, ON M3J 1P3, Canada}
\email{\href{mailto:bergeron@yorku.ca}{bergeron@yorku.ca}}

\author{Lucas Gagnon}
\address{Department of Mathematics,  University of Southern California, Los Angeles, CA 90089, USA}
\email{\href{mailto:lgagnon@usc.edu}{lgagnon@usc.edu}}

\author{Hunter Spink}
\address{Department of Mathematics,
University of Toronto, Toronto, ON M5S 2E4, Canada}
\email{\href{mailto:hunter.spink@utoronto.ca}{hunter.spink@utoronto.ca}}

\author{Vasu Tewari}
\address{Department of Mathematical and Computational Sciences, University of Toronto Mississauga, Mississauga, ON L5L 1C6, Canada}
\email{\href{mailto:vasu.tewari@utoronto.ca}{vasu.tewari@utoronto.ca}}

\begin{document}

\begin{abstract}
For a Coxeter element $c$ in a Weyl group $W$, we define the $c$-Coxeter flag variety $\cfl_c\subset G/B$ as the union of left-translated Richardson varieties $w^{-1}X^{wc}_w$.  
This is a complex of toric varieties whose geometry is governed by the lattice $\NC(W,c)$ of $c$-noncrossing partitions. We show that $\cfl_c$ is the common vanishing locus of the generalized Pl\"ucker coordinates indexed by $W\setminus\NC(W,c)$. 
We also construct an explicit affine paving of $\cfl_c$ and identify the $T$-weights of each cell in terms of $c$-clusters.  This paving gives a GKM description of $H^\bullet(\cfl_c)$ and $H^\bullet_{T_{ad}}(\cfl_c)$ in terms of the induced Cayley subgraph on $\NC(W,c)$, and we show these rings are naturally isomorphic for different choices of $c$.  
In type~$\mathrm{A}$, this recovers the quasisymmetric flag variety for a special $c$, and for general $c$ we show the cohomology ring has a presentation as permuted quasisymmetric coinvariants.
\end{abstract}

\maketitle
\setcounter{tocdepth}{1}
\tableofcontents

\section{Introduction}

Let $G$ be a reductive group with opposite Borel subgroups $B, B^{-} \subseteq G$ and maximal torus $T = B \cap B^{-}$.  
The  \emph{generalized flag variety} is the coset space $G/B$, which is a projective algebraic variety.  
The Weyl group $W = N_{G}(T)/T$ permutes the $T$-stable subvarieties in $G/B$ and indexes the cells in the Bruhat decomposition
\[
G/B = \bigsqcup_{w \in W} \Xc^{w} 
\qquad\text{with}\qquad
\Xc^{w} \coloneq BwB/B \cong \mathbb{A}^{\ell(w)}.
\]
One of the main ways in which the geometry of $G/B$ connects with the combinatorics of $W$ and the associated Bruhat order $\le_B$ is via Schubert, opposite Schubert, and Richardson varieties
\[
X^{w}\coloneqq \overline{BwB/B},
\qquad
X_{w}\coloneqq \overline{B^{-}wB/B},
\qquad\text{and}\qquad
X^{v}_{u} \coloneqq X^{v} \cap X_{u}\text{ for }u\le_B v.
\]

Let $c \in W$ be a \emph{Coxeter element}, by which we mean the product, in any order, of the simple reflections of $W$ determined by $B$. We will show later that $w\le_B wc$ is equivalent to $\ell(w)+\ell(c)=\ell(wc)$ and we will write $w\cdot c$ to represent the product $wc$ and the assertion that the product is length additive.
We define a \emph{$c$-Coxeter Richardson variety} to be a Richardson variety of the form
\[
X_{w}^{w\cdot c} \subseteq G/B,
\]
which we show is a toric variety of dimension $\ell(c)$.
In this article we introduce an equidimensional complex of toric varieties that we call the \emph{$c$-Coxeter flag variety}  
\[
\cfl_{c} \coloneq \bigcup w^{-1} X^{w\cdot c}_{w} \subseteq G/B.
\]
Surprisingly, many $w\in W$ produce exactly the same $w^{-1}X^{w\cdot c}_w$, so we can view this translation process as removing certain redundancies from the set of Coxeter Richardson varieties.
We will show that $\cfl_c$ is a geometric realization of Coxeter--Catalan phenomena in algebraic combinatorics.  

Beginning with Reiner's 1997 paper~\cite{Re97}, developments in algebraic combinatorics have demonstrated that the Catalan numbers $\frac{1}{n+1}\binom{2n}{n}$ are the ``Type $A$'' member in the family of \emph{$W$-Catalan numbers $\cat{W}$}; see e.g.~\cite{STW25}.  
These numbers enumerate interesting combinatorial structures for each $W$, many of which vary nontrivially over the choice of a Coxeter element $c \in W$. 
We will need three examples of this ``Coxeter--Catalan'' phenomenon.  
\begin{enumerate}[parsep=1.5ex,partopsep=1ex, label = (\Alph*)]
\item The lattice of \emph{$c$-noncrossing partitions $\NC(W, c) \subseteq W$} comprises all prefixes $\tau_1\cdots \tau_i$ of minimal length reflection factorizations $c=\tau_1\cdots \tau_n$; this type-independent definition comes from the theory of Artin groups and the $K(\pi, 1)$ conjecture, see~\cite{Be03, BW08} and~\cite[Chapter 1]{Arm06}.

\item The \emph{(combinatorial) $c$-clusters $\mathrm{Cl}_{c}$} are subsets of the almost positive roots corresponding to $g$-vectors for a cluster algebra; Fomin--Zelevinsky~\cite{FZ03, FZ03b} introduced combinatorial clusters to record cluster algebraic data, and Reading~\cite{Reading2007} generalized this notion to arbitrary $c$.

\item The \emph{$c$-Cambrian congruence} $\equiv_{c}$ on $W$ is a lattice quotient of the right weak order $\le_{R}$; introduced by Reading~\cite{Reading2006} as a generalization of the Tamari lattice and the associahedron, it provides an interface between noncrossing partitions and clusters~\cite{Reading2007, RS11}.  

\end{enumerate}
For a fixed $W$ the size of each set above is independent of $c$, and we have
\begin{align*}
\cat{W} \coloneqq |\NC(W, c)| = |\mathrm{Cl}_{c}| = |W \big/ \equiv_{c}|,
\end{align*}
which we take as our definition of the $W$-Catalan number. There is also a well-known formula for $\cat{W}$ using the fundamental degrees of $W$-invariants ~\cite{CelliniPapi}.

Our results show that the families in (A)--(C) anticipate the structure of $\cfl_{c}$.  
We begin with a description of $\cfl_{c}$ in terms of the generalized Pl\"{u}cker coordinates on $G/B$, among which we denote by $\Pl_{w}$ the extremal coordinate indexed by $w \in W$.

\begin{maintheorem}[{\Cref{thm:everying_about_cfl}}]
\label{maintheorem:PlVanish}
We have $\displaystyle \cfl_{c} = \bigcap_{w \in W \setminus \NC(W, c)} \{\Pl_{w} = 0\}$.
\end{maintheorem}
The ``thick matroid strata'' where extremal P\"{u}cker coordinates vanish was first studied in \cite{GGMS87, GelSer87} and plays an important role in combinatorial algebraic geometry, but these strata exhibit notorious combinatorial complexity. 
In fact, Mn\"ev universality~\cite{Mn88} shows that even in type $A$ Grassmannians the vanishing loci for extremal P\"{u}cker coordinates can realize arbitrary singularities, so the tractability of $\cfl_{c}$ is noteworthy.

Theorem~\ref{maintheorem:PlVanish} implies that the $T$-fixed points $(\cfl_{c})^{T}$ are the $c$-noncrossing partitions $\NC(W, c)$, and $\cfl_c$ is the union of all $T$-invariant subvarieties $X$ with $X^T\subset \NC(W,c)$.  In particular, the intersection of $\cfl_{c}$ with a Schubert cell $\Xc^{w}$ is nonempty if and only if $w \in \NC(W, c)$.  We call these \emph{Coxeter Schubert cells}
\[
\Xc^u_{\NC}\coloneqq \cfl_c\cap \Xc^{u} \qquad\text{for $u \in \NC(W, c)$}.
\]
In the following result, we use a bijection $\mathrm{Clust}: \NC(W, c) \to \mathrm{Cl}_{c}$ of Biane--Josuat-Verg\`es~\cite{BJV19}, and write $\nctocluster(u)$ for the subset of positive roots in $\mathrm{Clust}(u)$; see Section~\ref{sec:clusternoncrossing}.

\begin{maintheorem}[{\Cref{thm:everying_about_cfl}}]\label{maintheorem:Paving}
The Coxeter Schubert cells give a $T$-stable affine paving of $\cfl_{c}$, 
\[
\cfl_{c} = \bigsqcup_{u \in \NC(W, c)} \Xc^u_{\NC},
\qquad\text{and as a $T$-representation}\qquad
\Xc^u_{\NC}\cong \bigoplus_{\alpha\in \nctocluster(u)} \mathbb{C}_{-u\cdot \alpha}.
\]
\end{maintheorem}

An immediate consequence of Theorem~\ref{maintheorem:Paving} is that $H_{\bullet}(\cfl_{c})$ has a free basis indexed by the closure of each Coxeter Schubert cell, which we call a \emph{Coxeter Schubert variety}
\[
X^u_{\NC} = \overline{\Xc^u_{\NC}} \qquad\text{for $u \in \NC(W, c)$}.
\]
Therefore we have
$$\dim H_{2i}(\cfl_c)=\#\{ C \in \mathrm{Cl}_{c} \suchthat \text{$C$ contains $i$ positive roots}\},$$
see also \Cref{rem:ChapotonFTriangle}. 
The irreducible components are those $X^u_{\NC}$ indexed by $u\in \NC(W,c)$ where $\mathrm{Clust}(u)=\nctocluster(u)$. These are the \emph{fully supported $c$-noncrossing partitions} $\NC(W,c)^+\subset \NC(W,c)$, those which do not lie in any subgroup generated by a strict subset of the simple reflections, enumerated by the lesser-known \emph{positive $W$-Catalan numbers} $\cat{W}^+\coloneqq |\NC(W,c)^+|$. 

The Coxeter Schubert varieties are $W$-translates of toric Richardson varieties; to specify which ones, we use Reading's $c$-sortable combinatorics~\cite{Reading2007}; see Section~\ref{sec:csortability}.  
In particular there is a canonical way to assign each $x \in W$ a noncrossing partition $\nc_c(\sort(x))\in \NC(W,c)$ which is bijectively determined by the Cambrian class of $x$.

\begin{maintheorem}[\Cref{sec:Bruhatmax}]
    \label{maintheorem:BruhatMax}
    For each $u \in \NC(W, c)$ there is a unique $c'\le_B c$ (the subproduct of simple reflections in the minimal standard parabolic subgroup containing $u$) such that
    \[
    X^u_{\NC} = w^{-1}X^{w\cdot c'}_w
    \qquad\text{for every $w$ with $\nc_c(\sort(w^{-1}\wnaught))= u$}.
    \]
    In particular, $w,w'\in W$ correspond to the same $X^u_{\NC}$ if and only if $ w^{-1}\wnaught \equiv_c (w')^{-1}\wnaught$.
\end{maintheorem}
Further aspects of the combinatorics are elaborated in Section~\ref{sec:Results}, where we describe the complex of moment polytopes associated to $\cfl_{c}$.

\medskip

We conclude this portion of the introduction by sketching some results and questions about the cohomology ring of $\cfl_{c}$.  Each of the following generalizes to---and depends upon---similar statements about torus-equivariant cohomology, where in all cases the torus in question is the adjoint torus $T_{ad} \coloneq T/Z(G)$.  
We omit these results for brevity.

First recall that in addition to the classical $W$-coinvariant presentation of $H^{\bullet}(G/B)$ due to Borel, there is a presentation of $H^{\bullet}(G/B)$ as a GKM-type graph cohomology ring for the Cayley graph of $W$.
Moreover, Billey's formula~\cite{Bil99} defines a distinguished basis of Schubert classes 
\[
\schub{w} \in H^{\bullet}(G/B)
\qquad\text{for $w \in W$}
\]
which are dual to the homology classes of the Schubert varieties. 

We find similar results for $\cfl_{c}$, presenting $H^{\bullet}(\cfl_{c})$ as the GKM ring for the Hasse diagram of $\NC(W, c)$, which is a subgraph of the Cayley graph.  This allows us to prove the following result, which makes use of the fact~\cite[Proposition 3.16]{Hum90} that all Coxeter elements are conjugate.  

\begin{maintheorem}[\Cref{cor:coho_iso}]
\label{maintheorem:indepofc}
The cohomology ring of the Coxeter flag variety is independent of the choice of Coxeter element: if $c$ and $c' = wcw^{-1}$ are both Coxeter elements then there is an associated isomorphism
\[
\Psi_{c, w}: H^\bullet(\cfl_c)\cong H^\bullet(\cfl_{c'}).
\]
Furthermore $\Psi_{c', v} \circ \Psi_{c, w} = \Psi_{c, vw}$ when these maps are well-defined.
\end{maintheorem}

We also show, using a combinatorial argument, that $H^{\bullet}(\cfl_{c})$ has a distinguished basis---dual to a homology basis of translated Richardson varieties described in Theorem~\ref{maintheorem:BruhatMax}---which we call \emph{Coxeter Schubert classes} and denote by
\[
\schub{u}^{\NC}\in H^\bullet(\cfl_c) \qquad\text{for $u \in \NC(W, c)$}.
\]
Surprisingly, the maps $\Psi_{c, w}$ in Theorem~\ref{maintheorem:indepofc} do not map Coxeter Schubert classes to Coxeter Schubert classes, making this basis heavily dependent on $c$.

For a specific $c$ in type $\mathrm{A}$ we have a complete combinatorial understanding of the $\schub{u}^{\NC}$ \cite{BGNST1,BGNST2} (see the discussion after the statement of \Cref{maintheorem:TypeA}). We understand very little about the combinatorics of $\schub{u}^{\NC}$ outside of this case, and leave the reader with a list of questions to generalize results in \cite{BGNST1,BGNST2} for what we might aspirationally call ``Coxeter Schubert calculus''. 

\begin{question}
\label{ques:Billey}
Is there an analogue of Billey's formula for the $\schub{u}^{\NC}$ in the graph cohomology ring?
\end{question}

\begin{question}
\label{ques:ShubToFor}
For $\iota:\cfl_c\hookrightarrow G/B$ the inclusion map, is there a combinatorially positive interpretation of the coefficients $a^u_v$ in
\[
\iota^{\ast} \schub{v} = \sum_{u\in \NC(W,c)}a^u_v\schub{u}^{\NC}\in H^\bullet(\cfl_c)?
\]
These numbers equivalently decompose the classes $[X^{u}_{\NC}]=\sum a^u_v[X^v]$ in the Schubert homology basis $[X^{v}]$ of $H_{\bullet}(G/B)$, and are therefore positive for geometric reasons.
\end{question}

\begin{question}
\label{ques:products}
How do Coxeter Schubert polynomials multiply?  Are the coefficients $c^{w,\NC}_{u,v}$ in
\[
\schub{u}^{\NC}\schub{v}^{\NC}=\sum_{w\in \NC(W,c)}c^{w,\NC}_{u,v}\schub{w}^{\NC}\in H^\bullet(\cfl_c)
\]
positive, and is there a combinatorial witness to this fact? Because it is a reducible variety, $\cfl_c$ lacks Poincar\'e duality and to our knowledge does not have a geometric reason to be true.
\end{question}

As remarked above, there are $T_{ad}$-equivariant versions of these latter questions involving the notion of Graham-positivity \cite{Gra01}.

\subsection{Type $\mathrm{A}$}  

We now highlight some special features that appear when $W = S_{n+1}$, e.g. when $G = \GL_{n+1}$.  
In type $\mathrm{A}$, $G/B$ is isomorphic to the variety of complete flags in $n$-space. Denoting $\sym{n+1}=\mathbb{Z}[x_1,\ldots,x_{n+1}]^{S_{n+1}}$, Borel's theorem~\cite{Bor53} gives an isomorphism
\begin{equation}
\label{eqn:flagcohomology}
H^{\bullet}(G/B) \cong \mathbb{Z}[x_1,\ldots,x_{n+1}]/\langle f(x_1,\ldots,x_{n+1})-f(0,\ldots,0)\suchthat f\in \sym{n+1}\rangle,
\end{equation}
by identifying $x_{1}, \ldots, x_{n+1}$ with the negative Chern roots of the tautological flag.

We first recall our prior work with P.~Nadeau in~\cite{BGNST1, BGNST2}.  Here, the Coxeter flag variety for $c = (n+1\,\cdots\, 2\,1)$ was studied under the name \emph{quasisymmetric flag variety}, due to the following result. 

\begin{thm}[{\cite[Theorem A]{BGNST2}}]
\label{thm:quasi}
For $G = \GL_{n+1}$, $W = S_{n+1}$, and $c = (n+1\,\cdots\,2\,1)$, we have
\[
H^{\bullet}(\cfl_{c}) \cong \mathbb{Z}[x_1,\ldots,x_{n+1}]/\langle f(x_1,\ldots,x_{n+1})-f(0,\ldots,0)\suchthat f\in \qsym{n+1}\rangle
\]
where $\qsym{n}$ is the ring of quasisymmetric polynomials of Gessel \cite{Ges84} and Stanley \cite{StThesis}.
\end{thm} 

Combining Theorem~\ref{thm:quasi} with our results, particularly Theorem~\ref{maintheorem:indepofc}, we see that each Coxeter flag variety in type $\mathrm{A}$ corresponds to a nonstandard definition of quasisymmetry.  

\begin{maintheorem}[\Cref{cor:permutedqsym}]
\label{maintheorem:TypeA}
    For $G/B=\fl{n+1}$ and any Coxeter element $c\in S_{n+1}$, the restriction map
$H^\bullet(\fl{n+1})\to H^\bullet(\cfl_c)$ is surjective.  Moreover, if
$c=(w(n+1)\,w(n)\,\cdots\,w(1))$ as a cycle, then this map realizes $H^\bullet(\cfl_c)$ as the further quotient of \eqref{eqn:flagcohomology} as
\[
H^\bullet(\cfl_c)\cong
\mathbb{Z}[x_1,\ldots,x_{n+1}]\Big/
\Big\langle
f(x_{w(1)},\ldots,x_{w(n+1)})-f(0,\ldots,0)\ \Big|\ f\in \qsym{n+1}
\Big\rangle,
\]
the ring of \emph{permuted quasisymmetric coinvariants}.
\end{maintheorem}

For $c = (n+1\,\cdots\,2\,1)$, we were able to answer Question~\ref{ques:Billey} very concretely and to give a (highly recursive) combinatorial solution to Questions~\ref{ques:ShubToFor} and~\ref{ques:products} that confirmed positivity.  
Moreover, it was possible to realize each Coxeter Schubert class in the Borel presentation by defining a family of \emph{forest polynomials}, paralleling the Lascoux--Schutzenberger model of Schubert polynomials for the  Schubert classes \cite{LS82}.  

\begin{question}
For $G = \GL_{n+1}$, do the Coxeter Schubert classes $\schub{u}^{\NC}$ have nice polynomial representatives in $x_1,\ldots,x_{n+1}$?
\end{question}

In other types, we can be certain the analogous question has a negative answer: the $2$nd Betti numbers of $\cfl_{c}$ exceed those of $G/B$ in types $B$ and $C$, so the canonical map $\iota^{\ast}$ from Borel's model to $H^{\bullet}(\cfl_{c})$ cannot be surjective; see Remark~\ref{rem:ChapotonFTriangle}.

\subsection{Toric complexes}

The class of $c$-Coxeter Richardson varieties has been considered elsewhere owing to its connection to the permutohedral toric variety, which is the closure $\overline{T \cdot x}$ of a generic $T$-orbit in $G/B$.  

For $W = S_{n+1}$ and $c = (n+1\,\cdots\,2\,1)$, this connection was first observed in purely cohomological terms by~\cite{And10}, and then later as a regular toric degeneration of $\overline{T \cdot x}$ into $c$-Coxeter Richardson varieties by the authors of~\cite{HHMP,Li24}.  
This result was greatly generalized in~\cite{KSSB}, where it was shown that in any type and for any Coxeter element $c$ there is a degeneration of $\overline{T \cdot x}$ into the union 
\[
\operatorname{HHMP}_{c}
\coloneqq 
\bigcup X^{w\cdot c}_w.
\]

Each of these results relies on the fact that the moment polytopes for the $c$-Coxeter Richardson varieties form a regular subdivision of the generalized flag permutahedron 
\[
\mathrm{Perm}_{W} \coloneqq \operatorname{conv}\{w \cdot \rho \suchthat w \in W\},
\]
where $\rho \in \Q{T}$ is the Weyl vector.   Figure~\ref{fig:hhmp_more_general} illustrates this in type~$\mathrm{A}_3$: for $c=s_3s_2s_1$, the two polytopes of the same color are translates of one another.
Defant--Sherman-Bennett--Williams~\cite{DSBW25} studied a triangulation refining this subdivision which is determined by the combinatorics of noncrossing partitions and $c$-sortability.

\begin{figure}[!ht]
  \centering
  \includegraphics[
  width=0.48\textwidth]{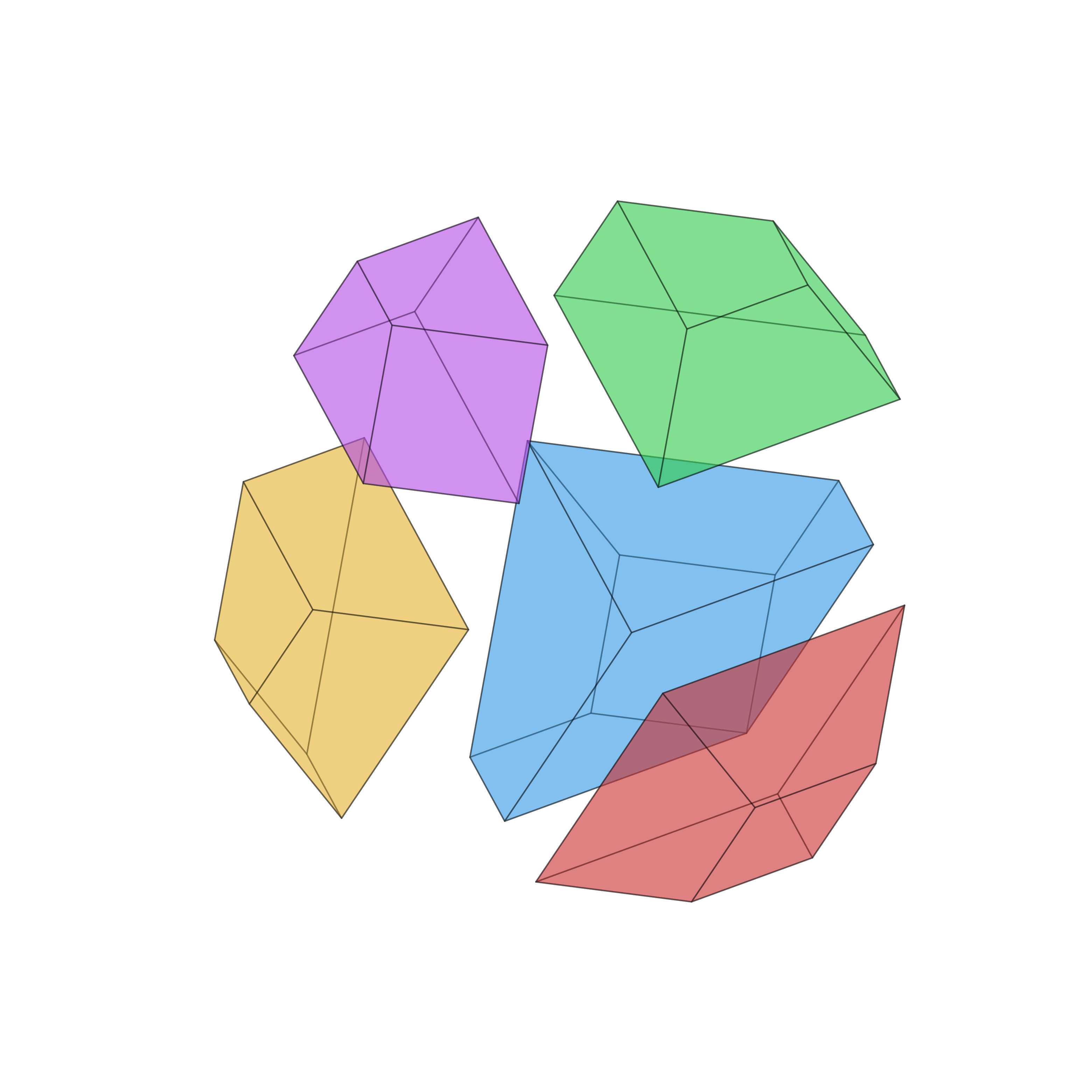}\hfill
  \includegraphics[width=0.48\textwidth]{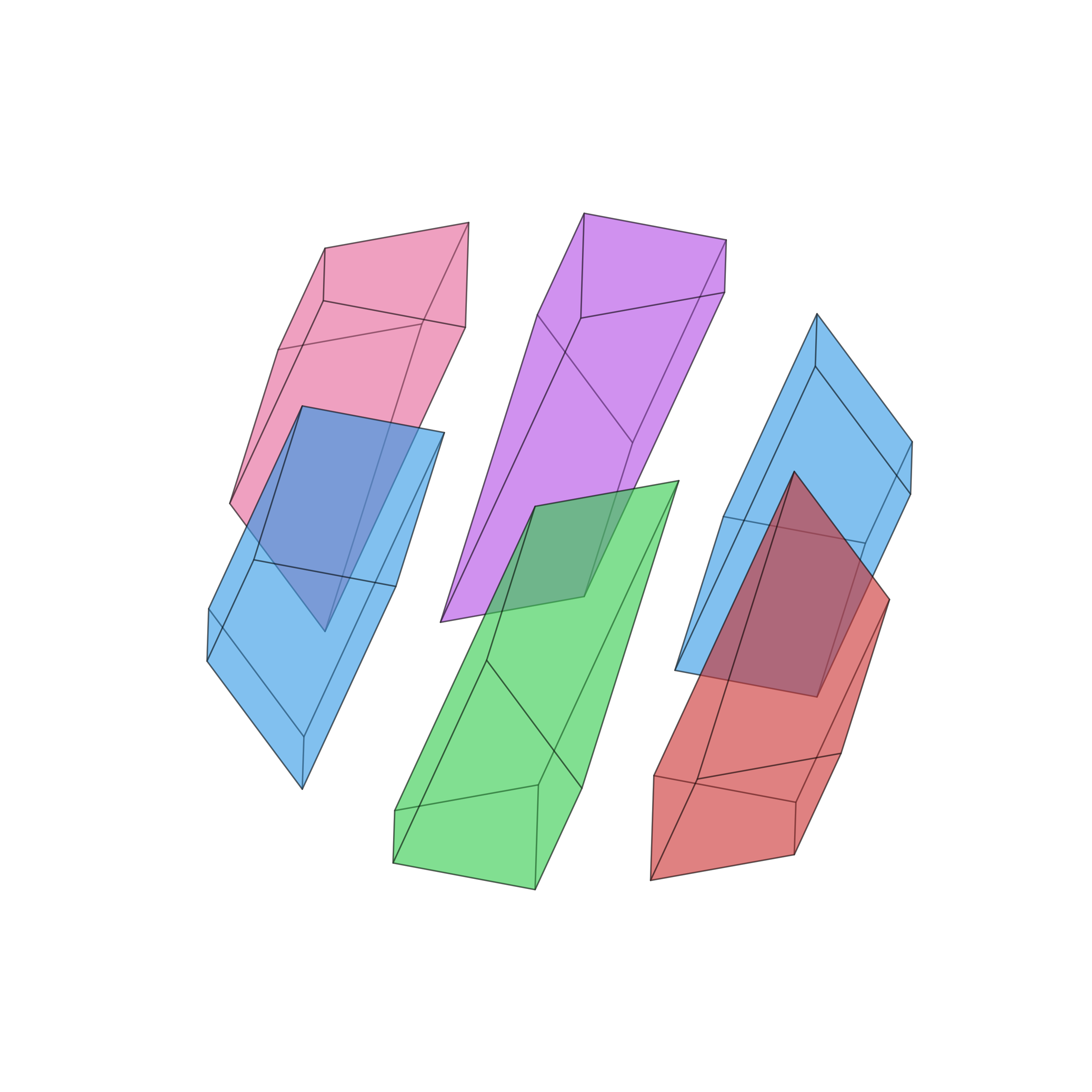}
  \caption{$\operatorname{HHMP}$ decomposition of the type $\mathrm{A}_3 $ permutahedron corresponding to $c=s_1s_3s_2$ (left) and $c=s_3s_2s_1$ (right). 
  The corresponding Coxeter flag varieties are each unions of five $3$-dimensional toric varieties, whose moment polytopes are translates of the depicted polytopes (the two blue polytopes on the right are identified  under this translation). See {\Cref{fig:two-models-panels}} for how the polytopes on the left assemble to a polytopal complex whose face lattice is $\cfl_{s_1s_3s_2}$.}
  \label{fig:hhmp_more_general}
\end{figure}

Our approach also makes use of moment polytopes, and in particular the ``moment complex'' comprising all moment polytopes for translated $c$-Coxeter Richardson varieties.  
This necessitates some combinatorial overlap between our results and~\cite{DSBW25}, and in particular~\Cref{thm:Equivwwc} is proved independently in ~\cite[Theorem~1.9 and Corollary~1.10]{DSBW25}.

\subsection{Paper structure}

After Section~\ref{sec:Results}, which reviews the moment complex of $\cfl_{c}$, we recall preliminary material on posets and Coxeter groups (Section~\ref{section:comb_prelims}), reductive groups (Section~\ref{section:coxlie_prelims}) and the geometry of $G/B$ (Section~\ref{section:Goe_prelim}).  
Section~\ref{section:Plucker} concerns Pl\"{u}cker coordinates, including a mix of new and folklore results that we present in a single exposition.  
The study of $\cfl_{c}$ begins in earnest in Section~\ref{section:noncrossingandBruhet}, where we relate the $T$-fixed points of $c$-Coxeter Richardson varieties to noncrossing partitions, and in Section~\ref{section:CFLc} where we study $\cfl_{c}$ as a toric complex using these combinatorics.  
In Section~\ref{sec:ClusterCharts} we introduce cluster charts and prove both Theorems~\ref{maintheorem:PlVanish} and~\ref{maintheorem:Paving}.  
Sections~\ref{section:cohomology},~\ref{section:DualityBases}, and~\ref{sec:typeAmain} contain a full account of our cohomological results, including Theorems~\ref{maintheorem:indepofc} and~\ref{maintheorem:TypeA}.  
Finally, Sections~\ref{sec:csortability},~\ref{section:CharacterizeBruhatIntervals}, and~\ref{sec:Bruhatmax} resolve the problem of determining when two $c$-Coxeter Richardson varieties are equal after translation, proving Theorem~\ref{maintheorem:BruhatMax}.  
In Appendices~\ref{appendix:GL},~\ref{appendix:SO},~\ref{appendix:C}, and~\ref{appendix:D}, we apply Sections~1--9 to the classical groups of the corresponding type and construct explicit examples of a $c$-Coxeter flag variety in each case.  

\subsection*{Acknowledgements} 
We thank Melissa Sherman-Bennett for a stimulating talk at the IAS workshop ``Combinatorics of Enumerative Geometry'' highlighting the results of \cite{KSSB}.
We are grateful to Allen Knutson for several enlightening conversations, and L.G.~is grateful to Arun Ram for help with Chevalley groups. 
We also thank Nathan Williams for sharing results from \cite{DSBW25} ahead of their appearance on the arXiv, and for answering our questions.

\section{Detailed structure of the moment complex}
\label{sec:Results}

In this section we describe the moment complex of $\cfl_c$, which underpins the combinatorics. We will consider $\cfl_{c}$ as a \emph{toric complex}, which we take to mean a union of projective toric varieties $\bigcup Y_i$ in a common projectivized $T$-representation, each of which determines a moment polytope $\mu(Y_i)$ as the convex hull of the characters associated to the $T$-fixed points $Y_i^T$.

\begin{rem}
Similar terms and ideas are present in the literature; see for example~\cite{BrunRomer, CasbiMasoomiYakimov, Sundbo}.  
As no source gives both a clear definition and a satisfactory name, we take the above out of expediency.
\end{rem}
The generalized Pl\"ucker embedding $\Pl$ realizes $\cfl_c$ as a toric complex. The $T$-fixed points of a $c$-Coxeter Richardson variety $X^{w\cdot c}_{w}$ come from the Bruhat interval $[w, w\cdot c]$. The moment polytope is the $\ell(c)$-dimensional twisted Bruhat interval polytope 
\[
P_{[w, w\cdot c]} = \operatorname{conv}\big( [w, w\cdot c] \cdot \rho \big)
\qquad\text{where $\rho$ is the Weyl vector},
\]
see Section~\ref{subsec:CoxMat}.  
Translation by $w^{-1}$ in $G/B$ also translates the moment polytope by $w^{-1}$, so the moment polytopes determined by $\cfl_c$ are all $w^{-1} P_{[w, w\cdot c]}$ and their lower-dimensional faces. These translated twisted Bruhat interval polytopes and their faces have remarkably special combinatorics.
\begin{thm}[\Cref{thm:polyposSect8}]\label{thm:polyposSect2}
    Each $w^{-1}P_{[w,w\cdot c]}$ (and its faces) are $c$-polypositroids in the sense of Lam--Postnikov \cite[\S 13]{LP20}.
\end{thm}
The remainder of this section will consider how we can glue these polypositroids together to create a polytopal complex that encodes the structure of $\cfl_c$.

Our results show that the set of vertices and edges in the totality of the moment polytopes of $\cfl_c$ are the Hasse diagram of $\NC(W,c)$ under the Kreweras order, corresponding to a decomposition of $\NC(W,c)$ (\Cref{thm:HasseUnion}) into translates of certain Bruhat intervals:
\begin{align*}
\NC(W,c)=\bigcup w^{-1}[w,w\cdot c].
\end{align*}

We would like to upgrade this statement about the $1$-skeleton of the ``union'' of the moment polytopes to describe a higher dimensional polytopal ``moment complex'' that faithfully encodes the combinatorics of how the $T$-orbit closures fit together in $\cfl_c$. 
We cannot do this by simply taking the union of the polytopes directly as there are spurious overlaps. In parallel with the face--orbit correspondence for toric varieties, we would like
\begin{enumerate}
    \item The dimension $k$ faces of the polytopal complex biject with $k$-dimensional $T$-orbits.
    \item Face containment corresponds to $T$-orbit closure containment.
\end{enumerate}
One issue that could arise in creating such a complex is that there could be two toric varieties in the complex with the exact same moment polytope, such as if the toric complex was the union of two general $T$-orbit closures in a projectivized $T$-representation. 
$\cfl_c$ has a remarkable property that allows us to carefully glue the overlapping polytopes together in a way that achieves the intended goal. In \Cref{thm:FirstCFlFacts} we show if $Y,Z\subset \cfl_c$ then
$$Y^T\subset Z^T \Longleftrightarrow Y\subset Z.$$
In particular, the facial structure we are attempting to achieve is detected entirely at the level of vertex sets. We can therefore define an abstract polytopal complex whose vertex set is $\NC(W,c)$ and edge set is $\Edges{\NC(W,c)}$ which encodes the combinatorial structure of $\cfl_c$:
\begin{align*}
\Complex(\cfl_c)=\left(\bigsqcup w^{-1}P_{[w,w\cdot c]}\right)/\sim,
\end{align*}
where $\sim$ glues two faces when they share the same vertex set. To clarify our intentions, consider the right panel of \Cref{fig:hhmp_original_and_translated.pdf}, consisting of the three polytopes which glue to form $\Complex(\cfl_c)$. Even though the polytopes overlap 
in a diamond shape, the above gluing process says we should only identify the two bolded edges. 
See Figure~\ref{fig:two-models-panels} for another example.

\begin{figure}[!htb]
    \centering
    \includegraphics[width=0.8\textwidth]{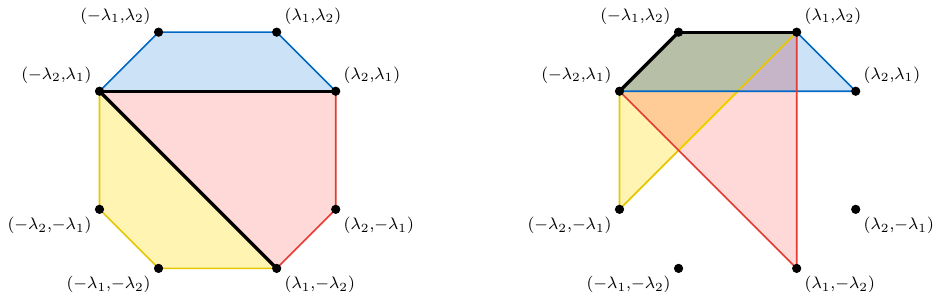}
    \caption{The blue, red, and yellow trapezoids $P_{[\idem, c]}$, $P_{[s_1, s_{0} s_{1} s_{0}]}$, and $P_{[c, \wnaught]}$ comprising the HHMP subdivision of the type $\mathrm{B}_{2}$-permutahedron for $c = s_{0}s_{1}$ (left) and the intersecting polytopes $P_{[\idem, c]}$, $s_{1}P_{[s_1, s_{0} s_{1} s_{0}]}$, and $c^{-1} P_{[c, \wnaught]}$ (right), where we only glue the bold edges.  
The elements $s_{1} s_{0}$ and $\wnaught$ outside of $\NC(\mathrm{B}_{2}, c)$ correspond to the vertices $s_{1}s_{0} \cdot (\lambda_{1}, \lambda_{2}) = (\lambda_2, -\lambda_1)$ and $\wnaught \cdot (\lambda_{1}, \lambda_{2}) = (-\lambda_{1}, -\lambda_{2})$}
    \label{fig:hhmp_original_and_translated.pdf}
\end{figure}

\begin{figure}[!htb]
\centering
\setlength{\tabcolsep}{1pt} 

\begin{tabular}{ccccc}
\includegraphics[width=0.19\textwidth]{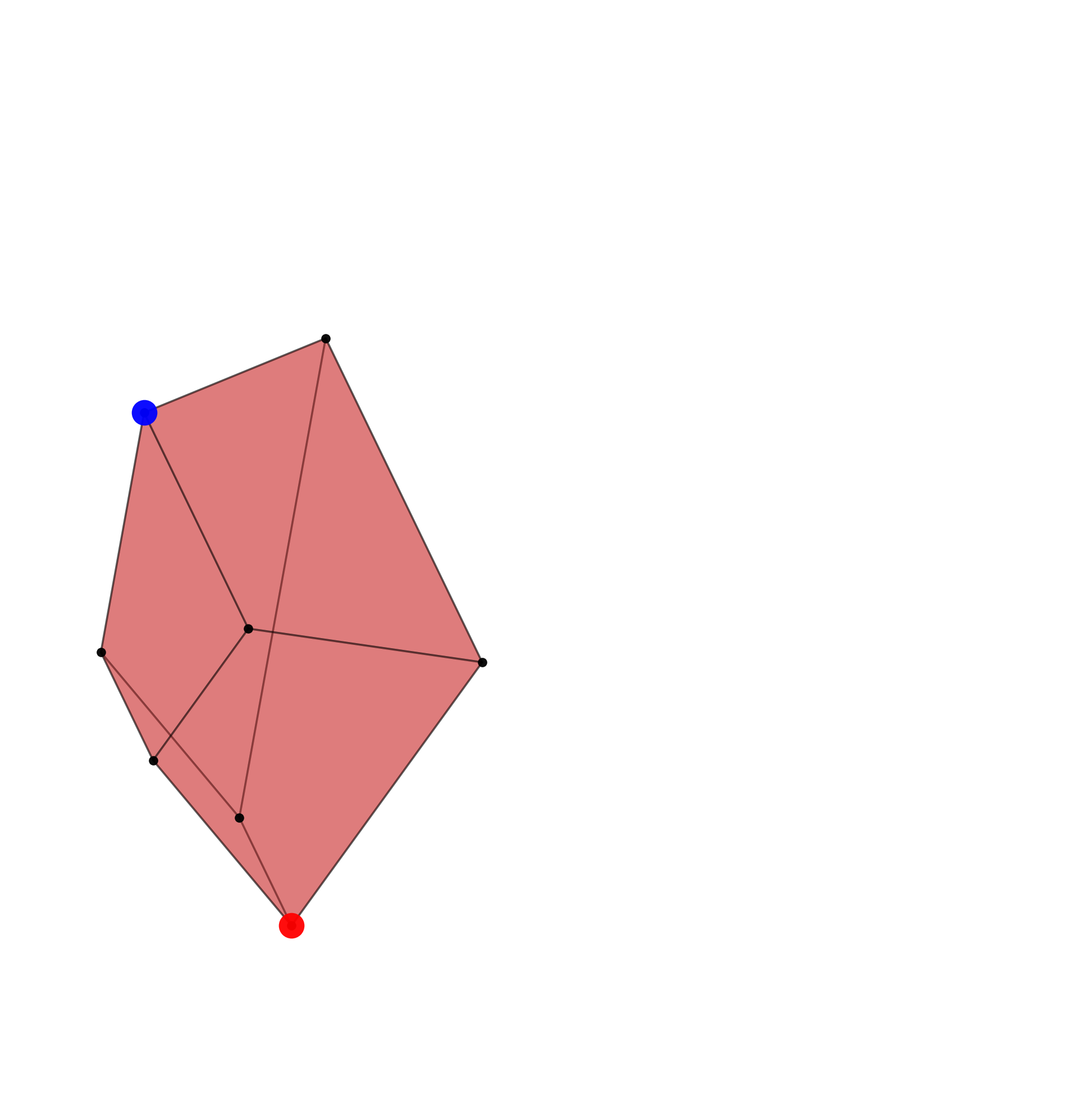} &
\includegraphics[width=0.19\textwidth]{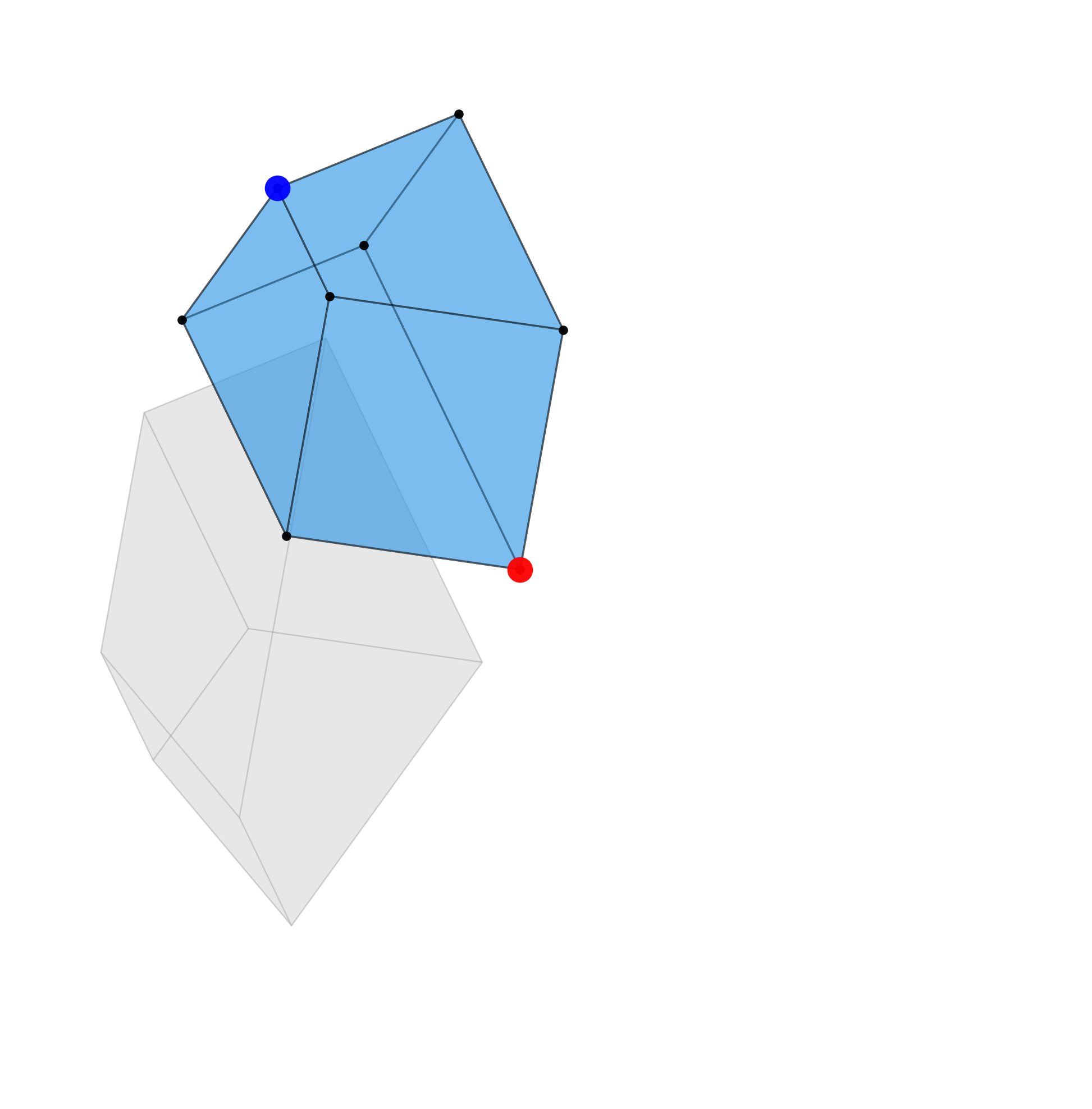} &
\includegraphics[width=0.19\textwidth]{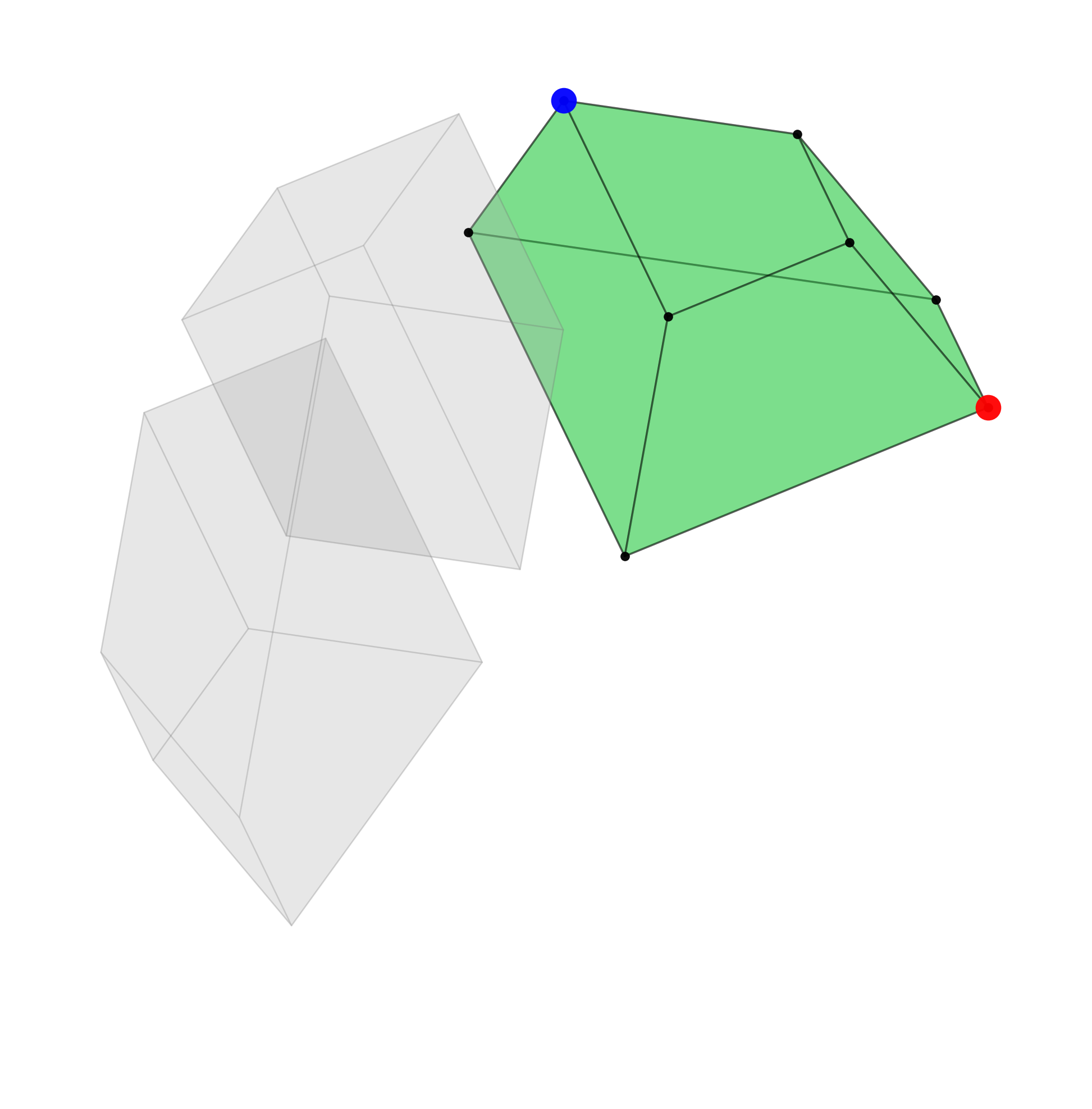} &
\includegraphics[width=0.19\textwidth]{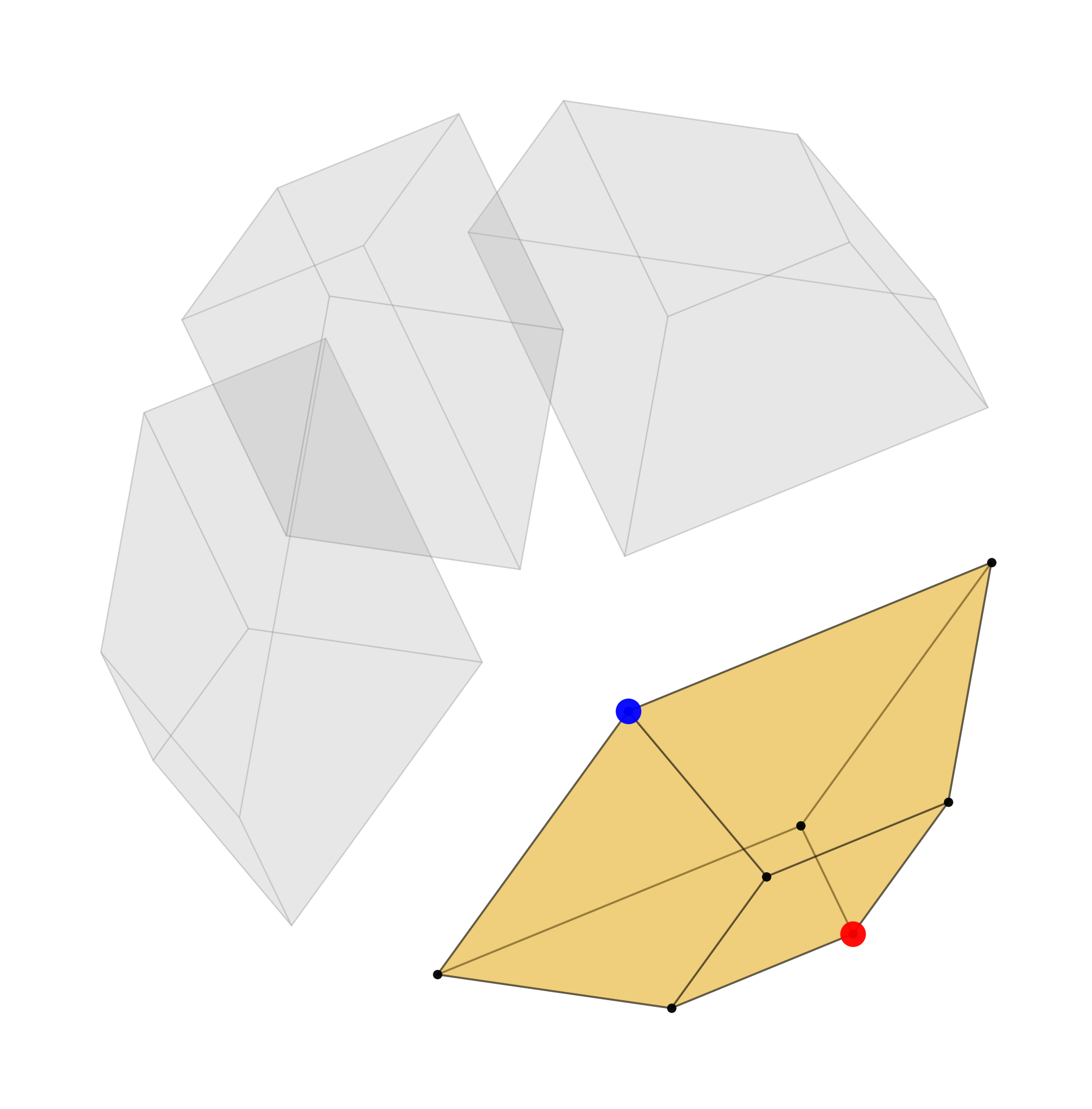} &
\includegraphics[width=0.19\textwidth]{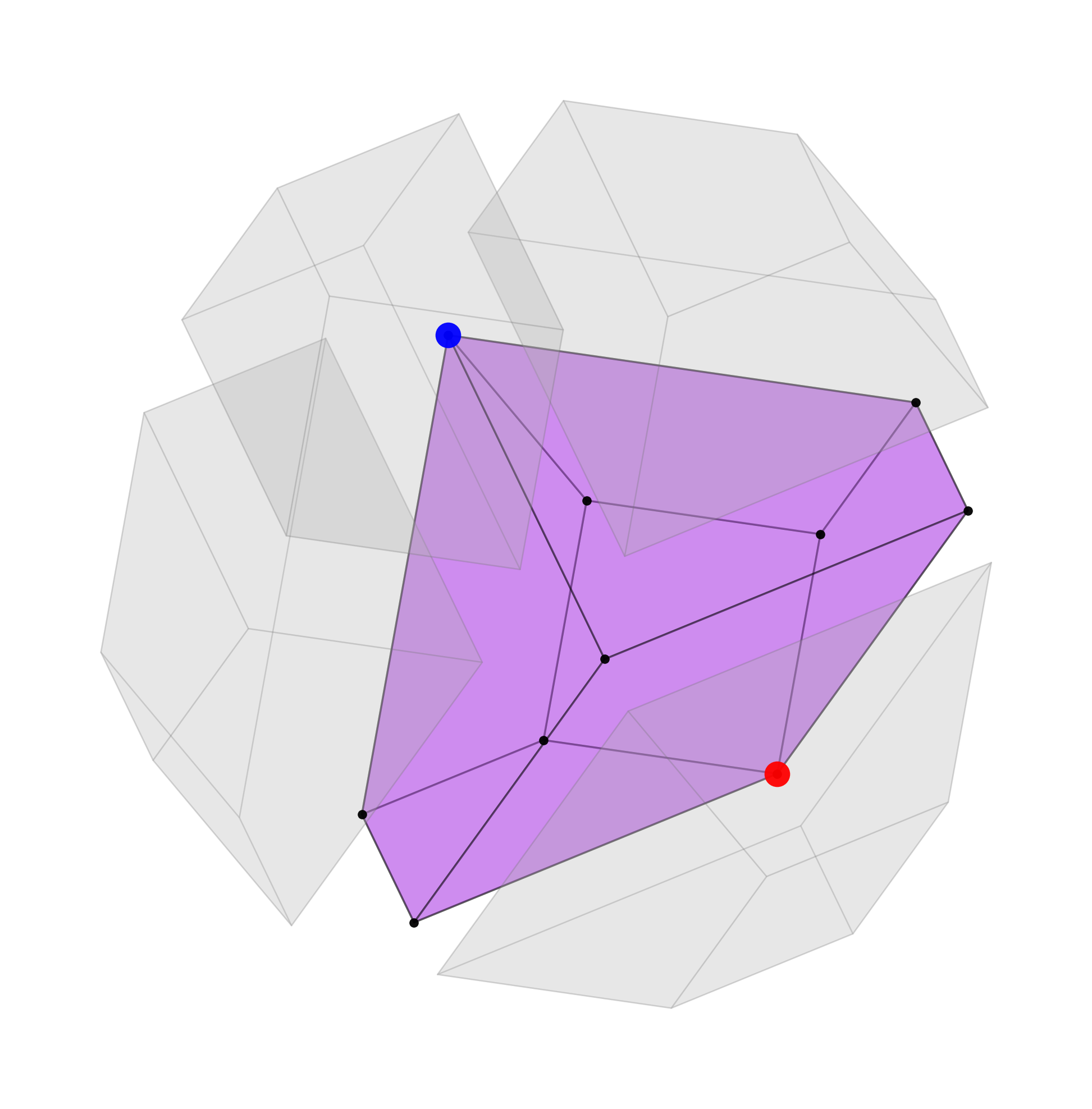} \\
\includegraphics[width=0.19\textwidth]{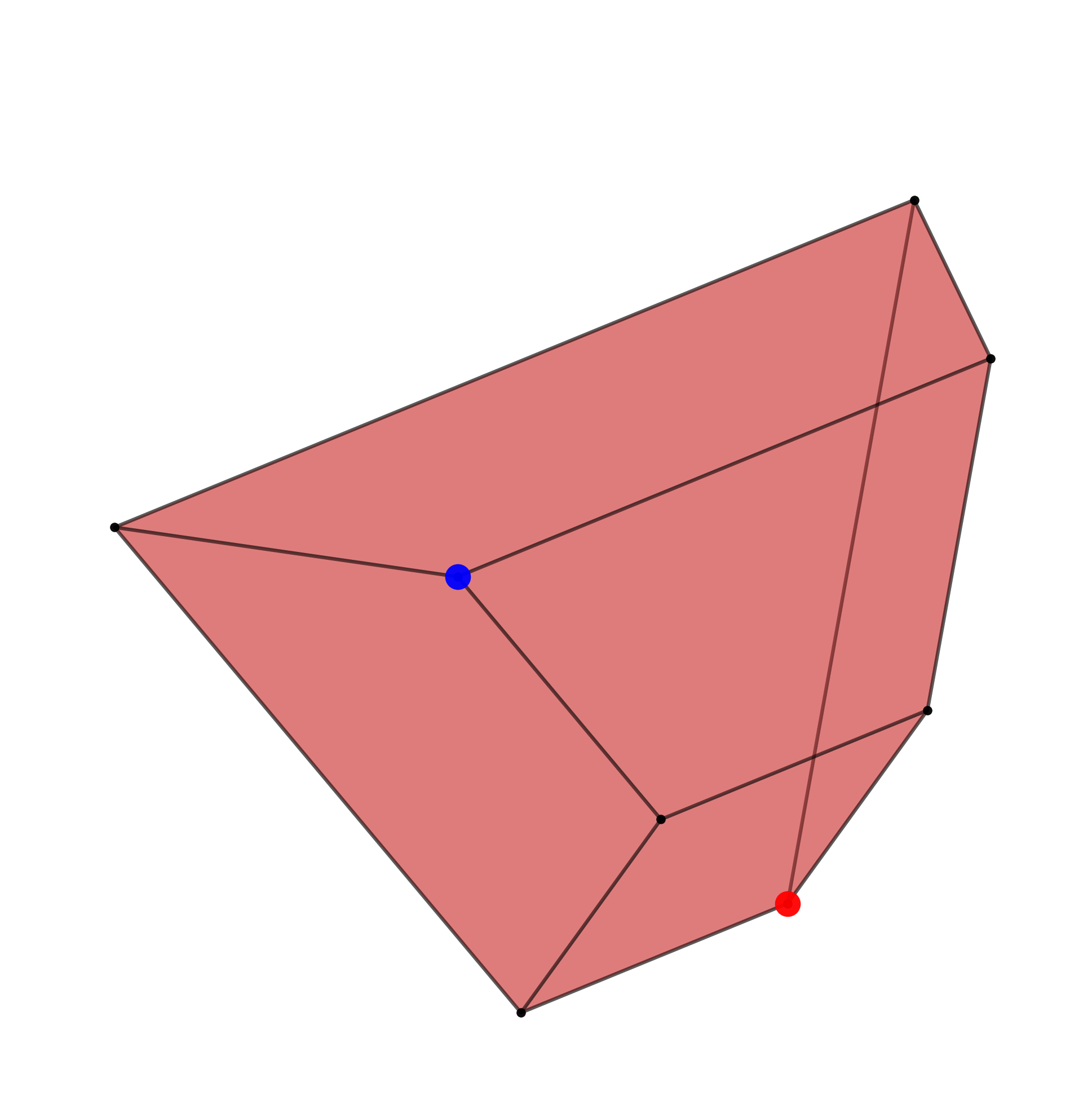} &
\includegraphics[width=0.19\textwidth]{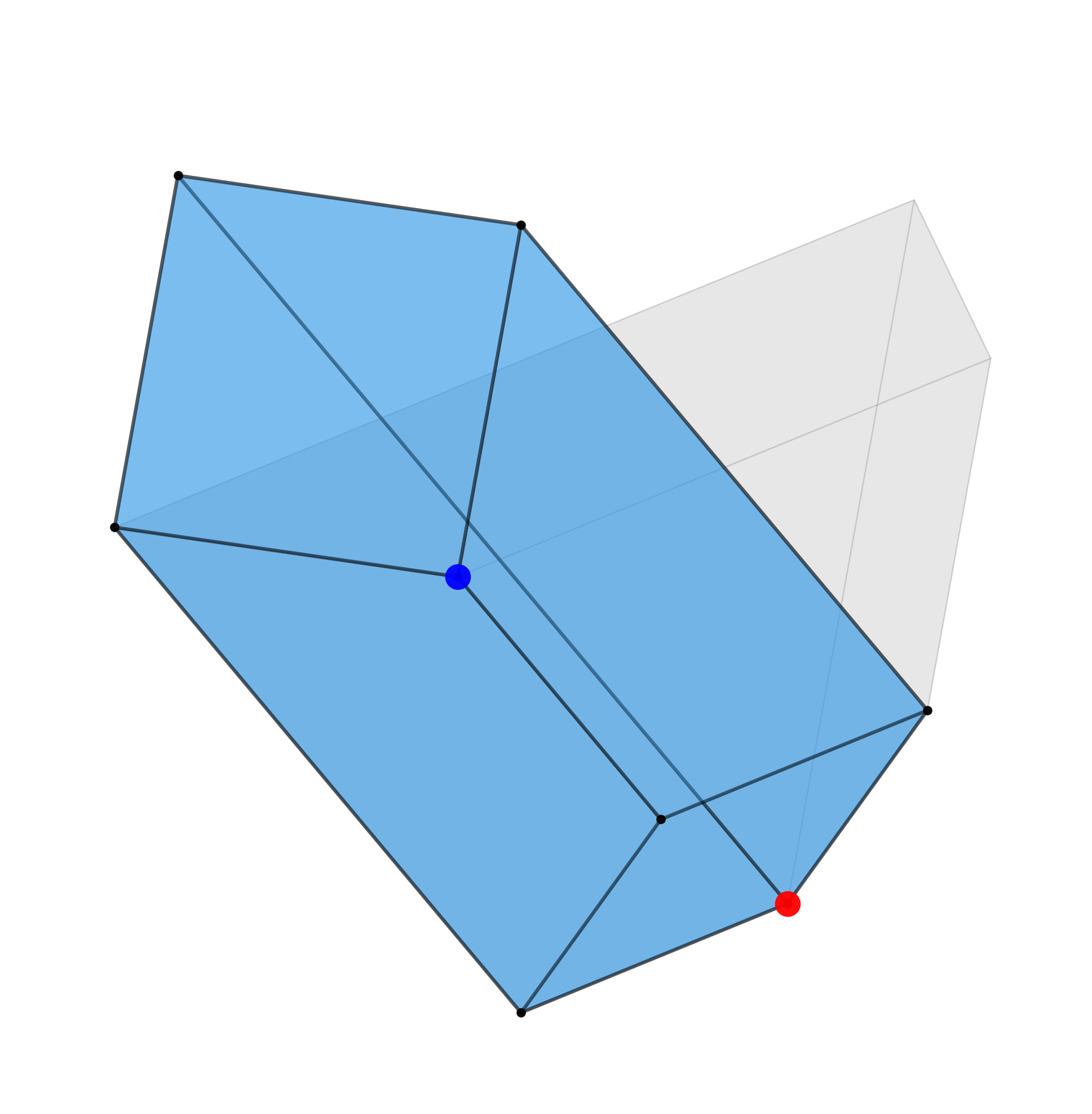} &
\includegraphics[width=0.19\textwidth]{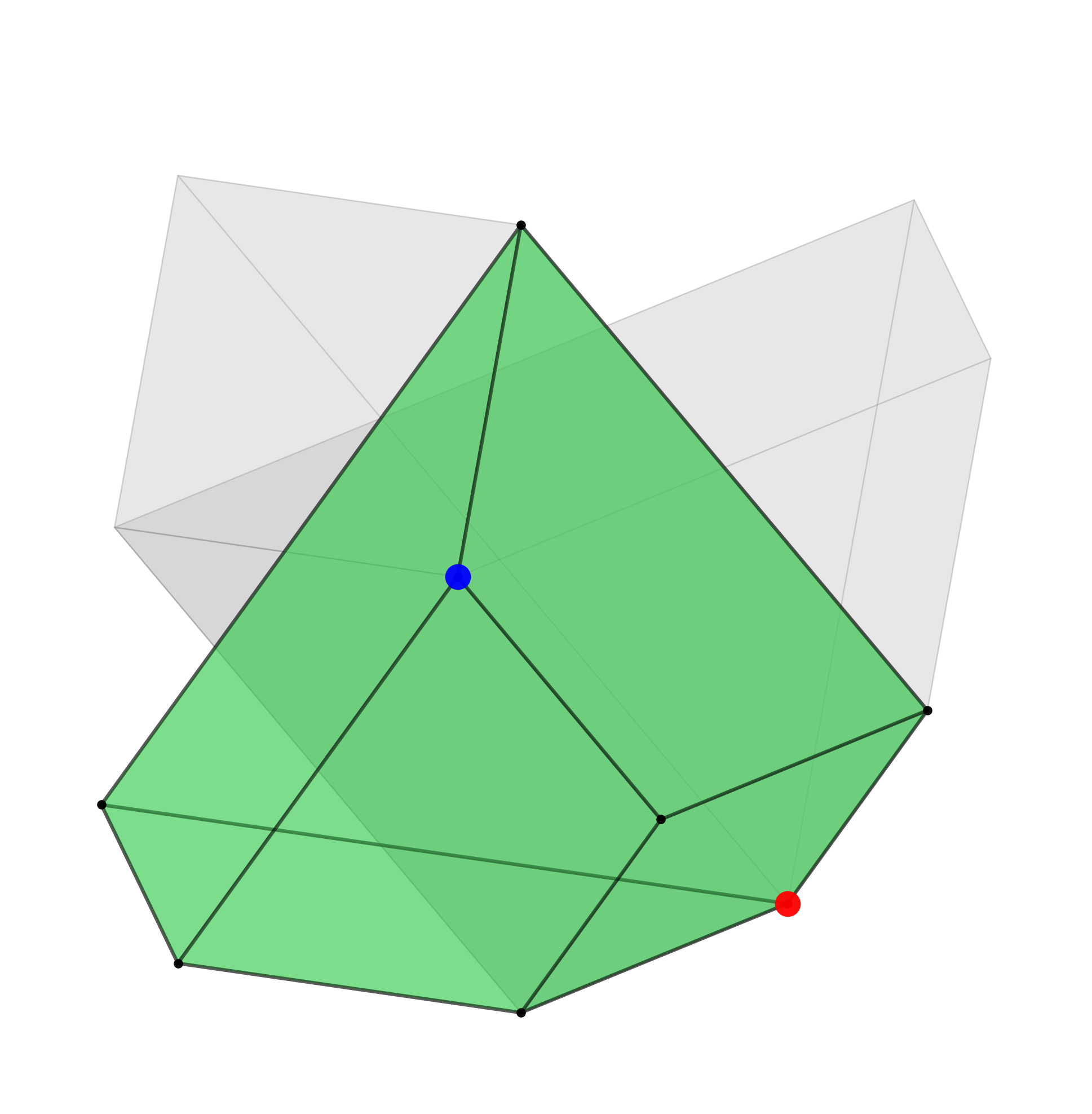} &
\includegraphics[width=0.19\textwidth]{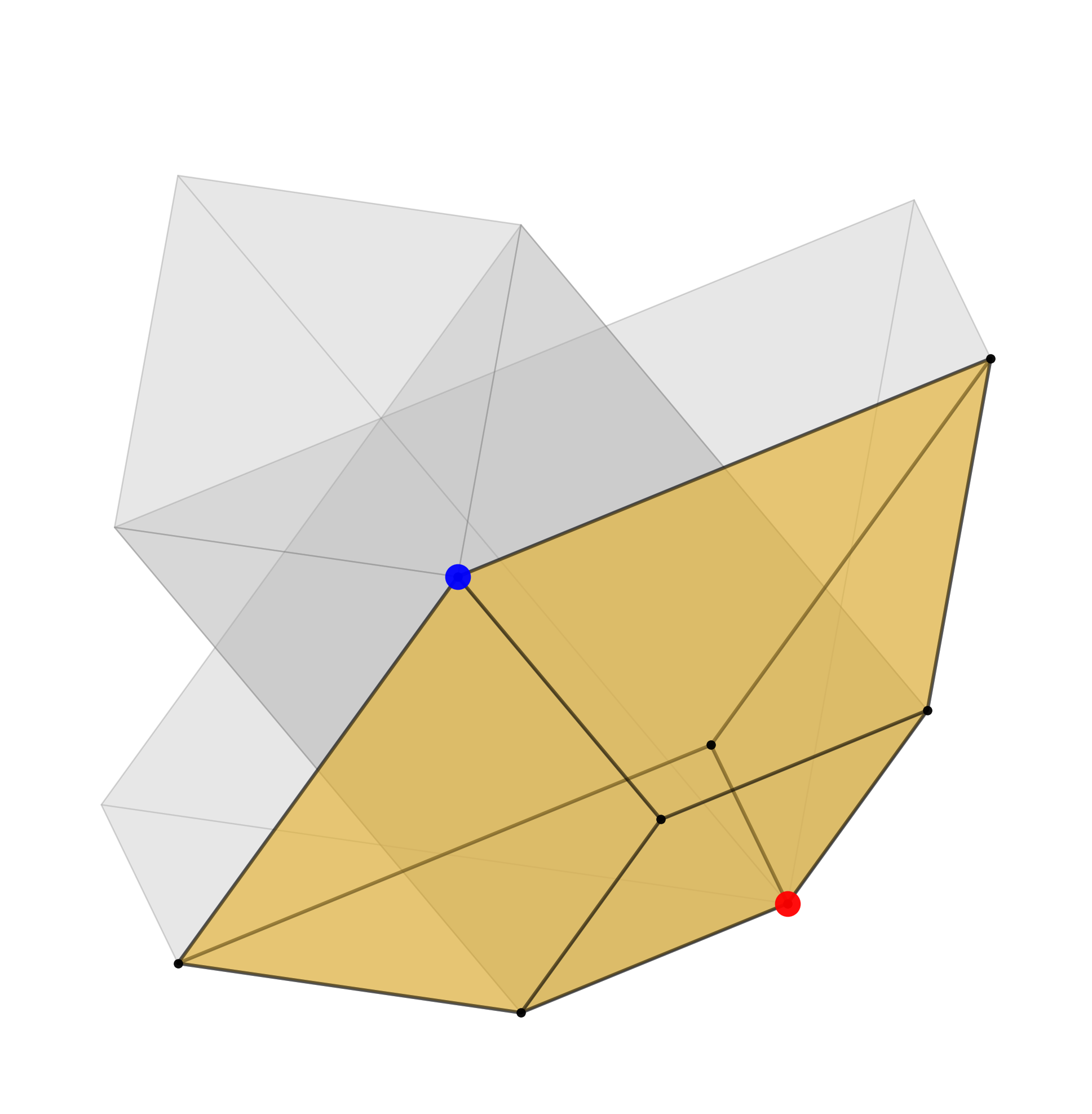} &
\includegraphics[width=0.19\textwidth]{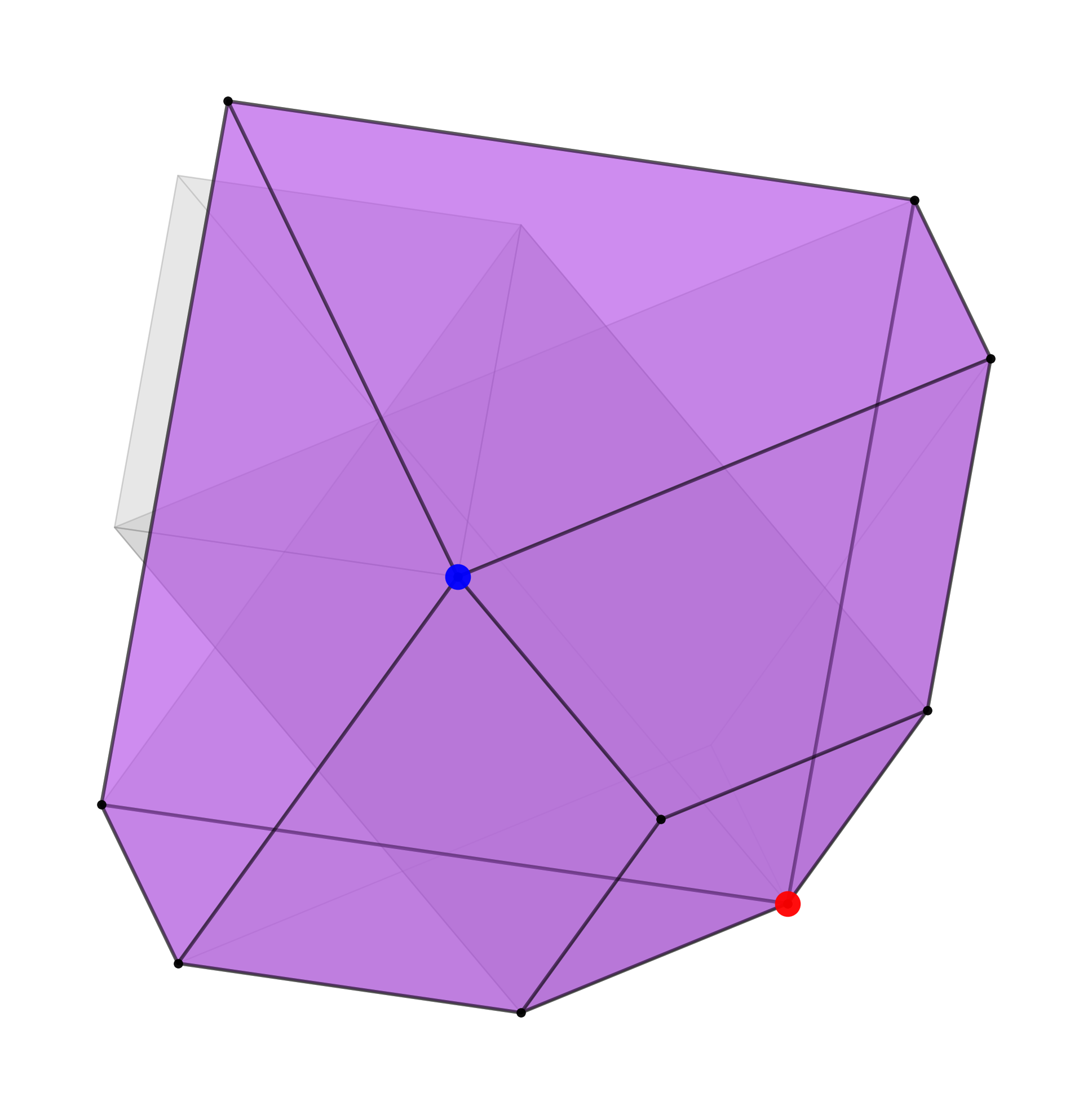}
\end{tabular}

\caption{Cumulative overlay of the moment polytopes for $c=s_1s_3s_2$; The top row is an exploded view of the ``HHMP'' decomposition  and  the bottom is the polytopal complex. The moment polytopes $P_{[w,wc]}$ from left to right come from $w$ equal to $2314$, $2413$, $1423$, $1234$, and $1324$. The first four have the same face lattice as a cube, and the last has the face lattice of the $10$ vertex tetragonal trapezohedron. The red and blue dots indicate respectively the images of $w$ and $w\cdot c$ either before or after translation by $w^{-1}$.}
\label{fig:two-models-panels}
\end{figure}

The top-dimensional faces of $\Complex(\cfl_{c})$ correspond to the $\cat{W}^+$-many irreducible components of $\cfl_c$, which does not depend on the choice of $c$. 
This is in contrast to the HHMP subdivisions of $\operatorname{Perm}_{W}$ as illustrated in \Cref{fig:hhmp_more_general}.

Finally, we note that the affine paving of $\cfl_{c}$ in Theorem~\ref{maintheorem:Paving} induces a decomposition of the moment complex $\Complex(\cfl_c)$ into ``open half-cubes'' around simple vertices. In fact, the more careful statement \Cref{thm:everying_about_cfl} shows that we can build $\Complex(\cfl_c)$ one half-cube at a time, in such a way that each intermediate complex is closed. By collapsing these half-cubes in reverse order we show that $\Complex(\cfl_c)$ is contractible.

\section{Combinatorial Preliminaries}
\label{section:comb_prelims}

We recall standard Coxeter group notation and refer to \cite{BjBr05,Bou98,Hum90} for background. 
Let $\mathbb{E}$ be a Euclidean space with inner product $(\cdot, \cdot)$, and $\Phi \subseteq \mathbb{E}$ a finite, crystallographic root system with a fixed choice of positive roots $\Phi^{+}$ and simple base $\Delta \subseteq \Phi^{+}$.  
Let $W$ be the group generated by the set of reflections $\Reflections = \{s_{\beta} \suchthat \beta \in \Phi^{+}\}$, which is a finite, crystallographic Coxeter group.

The association between positive roots and reflections is a bijection, and for $\tau\in \Reflections$ we denote the corresponding positive root by $r(\tau)\in \Phi^+$. Let $\Simple = \{s_{\alpha} \:|\; \alpha \in \Delta\}$ be the \emph{simple reflections} in $W$ associated to $\Delta$, so that $\Reflections = \{wsw^{-1}\suchthat s\in \Simple, w \in W\}$.  Denote the rank of $W$ by
\[
n = |\Simple| = |\Delta|.
\]
Note that we do not require $n$ to be the dimension of the ambient space $\mathbb{E}$.  

A \emph{reflection subgroup $W' \subseteq W$} is a subgroup generated by a subset of $\Reflections$.  
Each reflection subgroup is the Coxeter group for a sub-root system $\Phi' \subseteq \Phi$.  
A reflection subgroup $W' \subseteq W$ is a \emph{standard parabolic subgroup} if it can be generated by a subset $\Simple' \subseteq \Simple$; in this case each Coxeter element $c \in W$ determines a Coxeter element $c' \in W'$ by multiplying the elements of $\Simple'$ in the same order as in $c$.

All posets $P$ we consider in this article are finite and graded, with unique minimal and maximal elements. 
We denote the partial order on $P$ by $<_P$.
 We refer the reader to standard combinatorial texts (cf.~\cite{St12}) for any undefined terminology in the context of posets.
As is standard in combinatorics, we will identify a poset with its Hasse diagram whenever necessary. The underlying partial order will be clear from context.

\subsection{Bruhat order}

A \emph{reduced word} $\bm w=(s_1^{\bm w},\ldots, s_\ell^{\bm w})$ for $w\in W$ is any minimal-length factorization
$$w=s^{\bm w}_1\cdots s^{\bm w}_{\ell}\text{ with }s^{\bm w}_{1},\ldots,s^{\bm w}_{\ell}\in \Simple.
$$
The \emph{length} of $w$ is $\ell=\ell(w)$, the number of simple reflections in any such factorization. 
We shall call $t\in \Reflections$ an \emph{inversion} if $\ell(tw)<\ell(w)$, and denote the set of inversions of $w$ by $\inv{w}$.
We have $|\inv{w}|=\ell(w)$, and for any reduced word $\bm w$ a complete set of inversions for $w$ may be produced by $\tau^{\bm w}_j=s_{1}^{\bm w}\cdots s_{j}^{\bm w}\cdots s_{1}^{\bm w}$ for $1 \le j \le \ell(w)$, so that
\begin{align}
\label{eq:reflection_factorization}
    \tau_j^{\bm w}w=s_{1}^{\bm w}\cdots \wh{s_{j}^{\bm w}}\cdots s_{\ell}^{\bm w},
\end{align}
where $\hat{\cdot}$ denotes omission from the product.  
Accordingly, $w$ has an \emph{inversion factorization} 
\[
w=\tau^{\bm w}_\ell\cdots \tau^{\bm w}_1.
\]
The positive roots associated to the inversions of $w$ can be described as
$r(\inv{w})=\Phi^+\cap w(\Phi^-).$

The \emph{Bruhat order $\leq_B$} on $W$ is the transitive closure of the cover relations $$tw\lessdot_B w\text{ for }t\in \inv{w}.$$
Equivalently $v\leq_B w$ provided that a reduced word for $v$ appears as a subsequence inside any (equivalently, all) reduced words for $w$.

The Bruhat order has a unique maximum element $\wnaught \in W$, and $\inv{\wnaught}=\Reflections$. 
We write $[u,v]$ for the Bruhat interval between $u$ and $v$, considered as a subposet, and note that $[u,v]$ is isomorphic to $[u^{-1},v^{-1}]$ and anti-isomorphic to both $[\wnaught v,\wnaught u]$ and $[v\wnaught,u\wnaught]$  under the obvious identifications.

\subsection{Right weak order}
\label{sec:rightorder}

We call $s\in \Simple$ a \emph{descent} of $w\in W$ if $\ell(ws)<\ell(w)$. 
Equivalently, $s$ is a descent of $w$ if there exists a reduced word for $w$ whose last letter is $s$.
We denote the set of descents of $w$ by $\des{w}$. 

The \emph{(right) weak order $\leq_R$} on $W$ is the transitive closure of the cover relations $ws\lessdot_R w$ for $s\in \des{w}$.
In terms of reduced words, we have $u\leq_R v$ in weak order if some reduced word of $u$ is a prefix of some reduced word for $v$, or $\ell(v)=\ell(u)+\ell(u^{-1}v)$.  
Note that $u \le_{R} v$ implies $u \le_{B} v$. If $a\le_R ab$ then $\ell(ab)=\ell(a)+\ell(b)$ -- we write $a\cdot b$ both for the product $ab$ and the assertion that the product is length additive.

\subsection{Absolute order, Coxeter elements, and noncrossing partitions}

A \emph{minimal reflection factorization} of $w\in W$ is any factorization
$$w=\tau_{i_1}\cdots \tau_{i_k}\text{ with }\tau_{i_1},\ldots,\tau_{i_k}\in \Reflections$$
with as few terms as possible.  
The \emph{absolute length} of $w\in W$ is $\ell_{\Reflections}(w)=k$,
the number of reflections in any such factorization, which is also, 
$$
\ell_{\Reflections}(w)= n - \operatorname{dim}(\mathrm{Fix}(w)),
$$
where $\mathrm{Fix}(w)$ is the subspace in the defining representation for $W$ fixed by $w$.  
The \emph{absolute order $\leq_{\Reflections}$} on $W$ is defined by:
\[
    v\leq_{\Reflections} w \Longleftrightarrow \ell_{\Reflections}(w)=\ell_{\Reflections}(v)+\ell_{\Reflections}(v^{-1}w).
\]

A \emph{Coxeter element} $c$ of $W$ is any product of all the simple reflections taken in any order
\[
c=s_{1}^{\bm c}\cdots s_{n}^{\bm c}\text{ with }\Simple=\{s_1^{\bm c},\ldots,s_n^{\bm c}\}.
\]
Then $\ell_{\Reflections}(c)=n$ and the \emph{$c$-noncrossing partitions} are defined to be the elements
$$\NC(W,c)=\{u\suchthat u\leq_T c\}.$$
Two equivalent definitions of $\NC(W, c)$ that we will use later are:
\begin{itemize}
    \item the set of prefixes $\tau_{1}\cdots \tau_{\ell}$ of minimal length factorizations $\tau_1\cdots \tau_n$ of $c$, or
    \item the set of subproducts $\tau_{i_1}\cdots \tau_{i_\ell}$ of minimal length factorizations $\tau_1\cdots \tau_n$ of $c$,
\end{itemize}
with the equivalence following from applying Hurwitz moves $ab=b(b^{-1}ab)=(aba^{-1})a$.

The absolute order gives $\NC(W,c)$ the lattice structure known as the \emph{Kreweras lattice} \cite{Bi97, Kre72}.
This lattice is, up to isomorphism, independent of the choice of Coxeter element.   
Indeed, one need only observe that all Coxeter elements are conjugate \cite[Proposition 3.16]{Hum90} and that absolute order is invariant under conjugation.

The subset $\NC(W,c)^+\subset \NC(W,c)$ of \emph{fully supported $c$-noncrossing partitions} are those which do not lie in the subgroup generated by a strict subset of the simple reflections $\Simple$. 
The number $|\NC(W,c)^+|=\cat{W}^+$ is called the positive $W$-Catalan number, and is again independent of $c$. 

\subsection{EL-labelings and Reflection orders}
\label{sec:el_labeling}

We let $\Edges{P}$ denote the set of edges in the Hasse diagram of $P$ and $\Maxchains{P}$ denote the set of maximally refined chains.
An \emph{edge-labeling} of $P$ is a labeling of $\Edges{P}$ by a totally ordered set.  
The following foundational notion is due to Bj\"orner \cite{Bjo80}.
\begin{defn}
    An \emph{EL-labeling} is an edge-labeling of $P$ such that there is a unique increasing chain in $\Maxchains{[s,t]}$ for each interval $[s,t]$ in $P$ and, furthermore, this chain is lexicographically smaller than all other chains in the interval.
\end{defn}
The following result shows that EL-labeling carries topological information about $P$\footnote{In fact, this result was first discovered by Richard Stanley under weaker hypotheses.}.
\begin{fact}[{\cite[Theorem 2.7]{Bjo80}}]
\label{fact:mobius}
    For $P$ a finite graded EL-labeled poset with $\hat{0}$ and $\hat{1}$, the number of decreasing maximal chains is the absolute value $|\mu_P(\hat{0},\hat{1})|$ of the M\"obius function.
\end{fact}
We will be interested in the EL-labelings of posets where the edge labels are reflections in $\Reflections$.
\begin{defn}[{\cite[\S 2.2]{Dy93}}]
    A total ordering $\prec$ of $\Reflections$ is a \emph{reflection ordering} if whenever $\alpha,\beta,\gamma\in \Phi^+$ are such that $\alpha $ is a positive linear combination of $\beta $ and $\gamma$ then either $t_{\beta} \prec t_{\alpha} \prec t_{\gamma}$ or $t_{\gamma} \prec t_{\alpha} \prec t_{\beta}$.    
\end{defn}

In Dyer's study of Kazhdan--Lusztig polynomials, he showed~\cite[Proposition 2.13]{Dy93} that every reflection order corresponds to a reduced word $\mathbf{\wnaught}$ for the longest element $\wnaught\in W$ by setting $\tau_{1}^{\mathbf{\wnaught}} \prec \tau_{2}^{\mathbf{\wnaught}} \prec \cdots \prec \tau_{N}^{\mathbf{\wnaught}}$, where $\tau_{i}^{\mathbf{\wnaught}}$ is the reflection defined in~\eqref{eq:reflection_factorization}.
 Furthermore, he produced the following natural EL-labeling of Bruhat intervals\footnote{The existence of EL-labelings for Bruhat intervals is originally due to Bj\"orner and Wachs \cite[Theorem 4.2]{BW82}, but they did not employ reflection orders.}.
 
\begin{thm}[{\cite[Proposition 4.3]{Dy93}}]
\label{thm:Dyer}
The edge labeling of a Bruhat interval $[u,v]$ where we label  $w\lessdot_B wt$ with the reflection $t$ is an EL-labeling for any reflection order $\prec$.
Consequently, there is a unique increasing (and decreasing) chain in $[u,v]$ with respect to any reflection ordering.
\end{thm}

One distinguished reflection order which we will need is the \emph{$c$-reflection order}.  
The choice of Coxeter element $c$ determines a unique word for $\wnaught$, the \emph{$c$-sorting word}, which we recall in \Cref{sec:csortability}.  
Athanasiadis--Brady--Watt \cite{ABW07} show that the reflection order determined by the $c$-sorting word for $\wnaught$ gives an EL-labeling of $\NC(W,c)$.\footnote{The paper~\cite{ABW07} only establishes this labeling for bipartite Coxeter elements, but the more general claim holds. See, for instance, \cite[Appendix A]{NJV23}.}
\begin{thm}[{\cite{ABW07}}]
    The edge labeling of $\NC(W,c)$ where we label  $w\lessdot_{\Reflections} wt$ with the reflection $t$ is an EL-labeling for the $c$-reflection order.
\end{thm}

The M\"obius value of $\NC(W,c)$ is given by $\mu_{\NC(W,c)}(\idem,c)= \pm \cat{W}^+$; see \cite[Proposition 9]{Cha04}. 
We say that a maximal chain in $\NC(W,c)$ is \emph{$c$-decreasing} if it is decreasing under the edge labeling determined by the $c$-reflection order.
\Cref{fact:mobius} then implies the following.
\begin{cor}
\label{cor:numbdecreasing}
    In $\NC(W,c)$ there are $\cat{W}^+$-many $c$-decreasing maximal chains.
\end{cor}
See Figure~\ref{fig:nc4_chains} which shows the $c$-decreasing maximal chains in the Hasse diagram of $\NC(S_4,s_1s_3s_2)$.

\begin{figure}[!ht]
    \centering
    \includegraphics[scale=1.1]{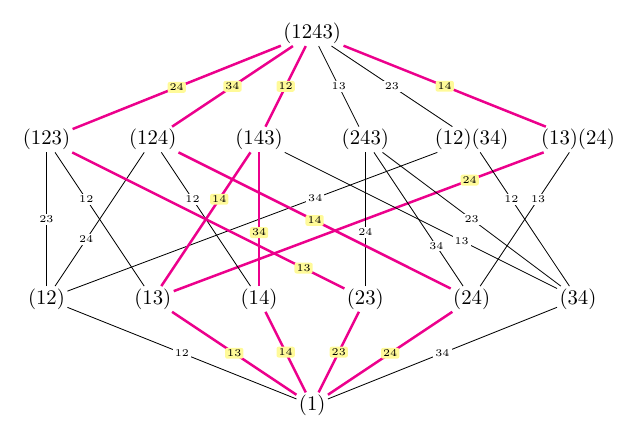}
    \caption{The Hasse diagram of $\NC(S_4,s_1s_3s_2)$ with the five decreasing chains in the $c$-reflection order $(1\, 2) \prec (3 \, 4)\prec (1\, 4) \prec (2\, 4) \prec (1\, 3) \prec (2\, 3)$ bolded.}
    \label{fig:nc4_chains}
\end{figure}

\subsection{Clusters and noncrossing inversions}
\label{sec:clusternoncrossing}

In~\cite[Theorems 1.9 and 1.13]{FZ03}, Fomin and Zelevinsky show that the cluster variables of a finite-type cluster algebra are in bijection with the almost-positive roots $\Phi_{\ge -1} = \Phi^{+} \cup \{-\alpha \suchthat \alpha \in \Delta\}$ for the corresponding root system.  
In~\cite[Section 7]{Reading2007}, Reading extends Fomin and Zelevinsky's result into a family of bijections parameterized by any pair $(W, c)$ for $W$ a finite Coxeter group and $c \in W$ a Coxeter element.  
Let
\[
\mathrm{Cl}_{c} = \{ \mathcal{F} \subseteq \Phi_{\ge -1} \suchthat \text{$\mathcal{F}$ is the image of a cluster under Reading's $c$-bijection}\}.
\]
We refer to the elements of $\mathrm{Cl}_c$ as \emph{$c$-clusters}, and define $\Cone{\mc{F}}$ to be the cone determined by the positive real span of the roots in $\mc{F}\in \mathrm{Cl}_c$.
The \emph{$c$-cluster fan} is the simplicial fan whose cones are $\Cone{\mc{F}}$ for $\mc{F} \in \mathrm{Cl}_{c}$.  
This generalizes, and is combinatorially isomorphic to, the cluster fan constructed by Fomin--Zelevinsky \cite[Theorem 1.10]{FZ03b}, but has additional orientation data $c$.

In what follows, we restrict our attention to the positive roots in each $c$-cluster, as in
\[
\mathrm{Cl}_c^+=\{\mathcal{F}\cap \Phi^{+} \suchthat \mathcal{F}\in \mathrm{Cl}_c\}.
\]
Each $c$-cluster $\mathcal{F}$ is determined by the positive roots that it contains, so this change sacrifices no combinatorial data.  
In fact, unlike $\mathrm{Cl}_c$, the sizes of the sets in $\mathrm{Cl}_c^+$ vary from $0$ to $n$. 
The cones generated by size $n$ elements of $\mathrm{Cl}_{c}^{+}$, together with all of their faces, give a simplicial fan, the \emph{positive $c$-cluster fan}, which subdivides the positive root cone $\operatorname{Cone}(\Delta)$ so that every positive root is the ray generator of a cone.

\begin{rem}
    Not every face of the positive $c$-cluster fan corresponds to an element of $\mathrm{Cl}_c^{+}$.
    Equivalently, $\mathrm{Cl}_c^{+}$ is not closed under taking subsets. 
    See \Cref{eg:B2examplewithdiagram}.
\end{rem}

\begin{defn}
    Given $u \in \NC(W, c)$, define the \emph{noncrossing inversion set} and \emph{right noncrossing inversion set} to be 
\[
\invnc{u}\coloneqq \{\tau\in \inv{u}\suchthat \tau u\in \NC(W,c)\}\text{ and }\invncR{u}\coloneqq \{u^{-1} \tau u \suchthat \tau\in \invnc{u}\}.
\]
\end{defn}
We can equivalently define the right noncrossing inversions as $\invncR{u}=\{\tau\in \Reflections\suchthat u\tau\lessdot_B u\text{ and }u\tau\in \NC(W,c)\}.$ In~\cite[Theorem 8.2, Lemma 8.3]{BJV19}, Biane and Josuat-Verg\`es define a bijection
\[
\begin{array}{rcl}
\nctocluster: \NC(W,c) &\to& \mathrm{Cl}_{c}^+\\
u &\mapsto& \{r(\tau) \suchthat \tau\in \invncR{u}\}.
\end{array}
\]
Clearly, restricting to $\NC(W,c)^+$ gives a bijection to the maximal cones in $\mathrm{Cl}_{c}^+$.

\begin{eg}
\label{eg:B2examplewithdiagram}
    Let us consider $W=\mathrm{B}_{2}$ with $c=s_0s_1$. Choose $\Delta=\{\alpha_0=\epsilon_1,\alpha_1=\epsilon_2-\epsilon_1\}$, so that the subset $\Phi^+$ of positive roots given by $\{\alpha_0,\alpha_1, \alpha_0+\alpha_1,2\alpha_0+\alpha_1\}$ corresponds respectively to the set $\Reflections$ of  reflections $\{s_0,s_1,s_1s_0s_1,s_0s_1s_0\}$. 
    The table below lists the image of the map $\nctocluster$ in the rightmost column.
See Figure~\ref{fig:b2_positive_cluster} for the root system of $\mathrm{B}_{2}$ with the maximal cones in the positive $c$-cluster fan highlighted. These correspond to the last three rows in the preceding table. Note that the rays $\alpha_0+\alpha_1$ and $2\alpha_0+\alpha_1$ do not correspond to positive $c$-clusters.
\begin{table}[ht]
\centering
\renewcommand{\arraystretch}{1.1}
\begin{tabular}{|c|c|c|c|}
\hline
$\operatorname{Sort}(W,c)$ & $\operatorname{NC}(W,c)$ & $\operatorname{Inv}_{\operatorname{NC}}$ & $\operatorname{Clust}^+$\\
\hline
$\operatorname{id}$ & $\operatorname{id}$ & $\varnothing$ & $\varnothing$\\
\hline
$s_0$ & $s_0$ &
$\{\,s_0\,\}$ &
$\{\,\alpha_0\,\}$\\
\hline
$s_1$ & $s_1$ &
$\{\,s_1\,\}$ &
$\{\,\alpha_1\,\}$\\
\hline
$s_0s_1$ & $s_0s_1s_0$ &
$\{\,s_0s_1s_0,\ s_1s_0s_1\,\}$ &
\textcolor{col1}{$\{\,2\alpha_0+\alpha_1,\ \alpha_0\,\}$}\\
\hline
$s_0s_1s_0$ & $s_1s_0s_1$ &
$\{\,s_1,\ s_1s_0s_1\,\}$ &
\textcolor{col2}{$\{\,2\alpha_0+\alpha_1,\ \alpha_0+\alpha_1\,\}$}\\
\hline
$s_0s_1s_0s_1$ & $s_0s_1$ &
$\{\,s_0,\ s_0s_1s_0\,\}$ &\textcolor{col3}{$\{\,\alpha_0+\alpha_1,\ \alpha_1\,\}$}
\\
\hline
\end{tabular}
\end{table}
\end{eg}
\begin{figure}
    \centering
    \includegraphics[width=0.35\linewidth]{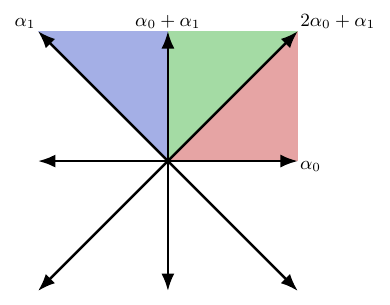}
    \caption{The $\mathrm{B}_{2}$ root system with its positive $c$-cluster fan highlighted, for $c=s_0s_1$. Note that the rays $\alpha_0+\alpha_1$ and $2\alpha_0+\alpha_1$ do not correspond to positive $c$-clusters.}
    \label{fig:b2_positive_cluster}
\end{figure}


\section{Representations of reductive groups}
\label{section:coxlie_prelims}

We now review facts about reductive groups and their representations; further details and definitions can be found in standard texts such as~\cite{Bo91, Hum75, Springer98}.
Let $G$ be a complex reductive group, so that $G$ is an algebraic group with a faithful completely reducible representation $V_{def}$.  
Let $T \subseteq G$ be a maximal algebraic torus and $B \subseteq G$ a Borel subgroup. 
The Weyl group is $W = N_{G}(T)/T$, where $N_{G}(T)$ denotes the normalizer of $T$ in $G$, and the opposite Borel $B^{-}$ is $\wnaught B \wnaught$.  

\subsection{Cartan decomposition}

Let $\Q{T}$ be the group of algebraic characters of $T$. 
The Weyl group $W$ acts on $\mathbb{E} = \QQ \otimes_{\ZZ} \Q{T}$, and we fix once and for all a $W$-invariant inner product $(\cdot, \cdot)$, making $\mathbb{E}$ a Euclidean space. 

Every finite-dimensional representation $V$ of $G$  determines a weight space decomposition
\[
V = \bigoplus_{\lambda \in \Q{T}} V_{\lambda}
\qquad\text{where}\qquad
V_{\lambda} = \{v \in V \suchthat \text{$h.v = \lambda(h)v$ for all $h \in T$}\}.
\]
Let $\mathfrak{g}$ be the Lie algebra of $G$.  The weight space decomposition of the adjoint representation on $\mathfrak{g}$ is the Cartan decomposition
\[
\mathfrak{g} = \mathfrak{g}^{T} \oplus \bigoplus_{\alpha \in \Phi} \mathfrak{g}_{\alpha}.
\]
The set $\Phi$ in the Cartan decomposition is a root system under an appropriate generalization of the Killing form on $\Q{T}$, see~\cite[\S 7.4]{Springer98}, and the corresponding Coxeter group is isomorphic to $W$. The Lie algebras $\mathfrak{b},\mathfrak{b}^-\subset \mathfrak{g}$ of $B$ and $B^-$ are
\[
\mathfrak{b}=\mathfrak{g}^T\oplus \bigoplus_{\alpha\in \Phi^+}\mathfrak{g}_{\alpha}
\qquad\text{ and }\qquad
\mathfrak{b}^-=\mathfrak{g}^T\oplus \bigoplus_{\alpha\in \Phi^-}\mathfrak{g}_{\alpha},
\]
which induces a decomposition $\Phi=\Phi^+\sqcup \Phi^-$ into the \emph{positive roots} and the \emph{negative roots}.

 Going forward, we fix a \emph{Chevalley basis} $\{H_{\lambda_{1}}, \ldots, H_{\lambda_{n}}\} \cup \{E_{\alpha} \suchthat \alpha \in \Phi\}\subset \mathfrak{g}$ with respect to $V_{\text{def}}$ with each $H_{\lambda_{i}} \in \mathfrak{g}^{T}$ and $E_{\alpha} \in \mathfrak{g}_{\alpha}$. 

\subsection{Chevalley presentation}
\label{sec:Greps}  
We now describe some key generators and relations of $G$.

We first identify the set $\Qvee{T} = \{\mu^{\vee} \in \mathbb{E} \suchthat \text{$(\lambda, \mu^{\vee}) \in \ZZ$ for all $\lambda \in \Q{T}$}\}$ with the cocharacters of $T$ by associating to each $\mu^{\vee} \in \Qvee{T}$ a \emph{one-parameter cocharacter subgroup}
\begin{equation}
\label{eq:one_param_subgroup}
h_{\mu^{\vee}}: \CC^{\times} \to T 
\qquad\text{such that for all $\lambda \in \Q{T}$,}\qquad
\lambda(h_{\mu^{\vee}}(x)) = x^{(\lambda, \mu^{\vee})}.
\end{equation}
In particular, we have one-parameter subgroups $h_{\alpha^{\vee}}$ for each \emph{coroot} $\alpha^{\vee} = 2 \alpha / (\alpha, \alpha)$, $\alpha \in \Phi$.

In order to define our second family of one-parameter subgroups, we note that the derivative of a rational representation of $G$ gives a representation of $\mathfrak{g}$ on the same space. 
In particular, we can identify each Chevalley generator $E_{\alpha} \in \mathfrak{g}_{\alpha}$ with its image in $\mathrm{End}_{\CC}(V_{\text{def}})$ under this representation.  
Define now the \emph{one-parameter root subgroup} for $\alpha \in \Phi$ by
\begin{equation}
\label{eq:one_param_root}
e_{\alpha}: \CC \to G \qquad e_{\alpha}(x) = \mathrm{exp}(xE_{\alpha}) = 1 + xE_{\alpha} + \frac{x^{2}}{2} E_{\alpha}^{2} + \cdots.
\end{equation}

The group generated by the $e_{\alpha}(\CC)$ is a semisimple subgroup of $G$, and together with $T$ these elements generate $G$. Before giving a presentation on generators and relations, we define one new family of elements: for $\alpha \in \Phi$, let 
\begin{align}
\label{equation:sdefn}
s_{\alpha}(x) &= e_{\alpha}(x)e_{-\alpha}(-x^{-1})e_{\alpha}(x), &&\text{for $x \in \CC^{\times}$.} 
\end{align}
Then each $s_{\alpha}(x)$ is a representative for $s_{\alpha}$ in $N_{G}(T)$.  Moreover, with the one-parameter subgroups defined in~\eqref{eq:one_param_subgroup}, the following relations hold for all $\alpha, \beta \in \Phi$ and $x, y \in \CC$:
\begin{align}
\label{equation:ee_rel1}
e_{\alpha}(x)e_{\alpha}(y) &= e_{\alpha}(x+y), \\
\label{equation:ee_rel2}
e_{\alpha}(x)^{-1}e_{\beta}(y)^{-1}e_{\alpha}(x)e_{\beta}(y) &=  \prod_{i\alpha + j\beta \in \Phi^{+}} e_{i\alpha + j\beta}(c_{i,j}x^{i}y^{j}), &&  \text{for $c_{i, j} \in \ZZ$ independent of $x, y$},\\
\label{equation:h_rel}
s_{\alpha}(x)s_{\alpha}(-1) &= h_{\alpha^{\vee}}(x), && \text{for $x \in \CC^{\times}$},\\
\label{equation:sh_rel}
s_{\alpha}(x)h_{\lambda^{\vee}}(y)s_{\alpha}(-x) &= h_{s_{\alpha}\lambda^{\vee}}(y), && \text{for all $\lambda^{\vee} \in \Qvee{T}$,}\\
\label{equation:se_rel}
s_{\alpha}(x)e_{\beta}(y)s_{\alpha}(-x) &= e_{s_{\alpha}\beta}( \pm x^{-( \alpha^{\vee}, \beta ) } y), && \text{for $\pm$ independent of $x, y$},\\
\label{equation:he_rel}
h e_{\alpha}(x) h^{-1} &= e_{\alpha}(\alpha(h)\;x), && \text{for all $h \in T$}.
\end{align}
We omit relations between different one-parameter subgroups of $T$, so the above is not a complete presentation for $G$.

\subsection{Weights and representations}

A \emph{dominant weight} for $\mathfrak{g}$ is an element of
\[
\Lambda^{+} = \{\lambda \in \mathbb{E} \suchthat \text{$(\lambda, \alpha^{\vee}) \in \ZZ_{\ge 0}$ for all $\alpha \in \Phi^{+}$} \}.
\]
Say that $\lambda \in \Lambda^+$ is \emph{regular} if $(\lambda, \alpha^{\vee}) > 0$ for each $\alpha \in \Phi^{+}$.  

Every dominant weight $\lambda \in \Lambda^{+}$ determines a finite-dimensional, irreducible $\mathfrak{g}$-module.  
When $\lambda \in \Lambda^{+} \cap \Q{T}$, this $\mathfrak{g}$-module integrates to a simple $G$-module $V^{\lambda}$.  The following facts, paired with the Chevalley presentation, will allow us to compute in these representations.

\begin{fact}
\label{fact:repfacts}
For any dominant integral weight $\lambda$ and $\mu\in \Q{T}$ we have:
\begin{enumerate}[label=(\arabic*)]
\item \label{item:repfacts1} if $\dot{w} \in N_{G}(T)$ represents $w \in W$, $v \mapsto \dot{w}v$ is a vector space isomorphism from $V_{\mu}^{\lambda}$ to $V_{w\cdot \mu}^{\lambda}$,  

\item for $\alpha \in \Phi$, $E_{\alpha}$ maps $V_{\mu}^{\lambda}$ to $V_{\mu + \alpha}^{\lambda}$, and
\item if $\mu \in \lambda + \ZZ_{\ge 0}\Phi^{+}$ then $\dim(V^{\lambda}_{\mu}) = \delta_{\lambda, \mu}$.
\end{enumerate}
\end{fact}

The \emph{extremal weights} of $V^{\lambda}$ are those of the form $w \lambda$ for $w \in W$.  The above properties imply that all nonzero weight spaces come from weights $\mu$ contained in the convex hull of the extremal weights, and the dimension of each extremal weight space is $1$.

\section{Geometric preliminaries}
\label{section:Goe_prelim}

We now review important facts about the \emph{generalized flag variety} $G/B$.  
First recall that the group $W$ permutes the $T$-orbits in $G/B$ by having $w \in W$ act as left multiplication by any representative in $N_{G}(T)$, and that the $T$-fixed point set $(G/B)^{T}$ is the $W$-orbit of $B$.

For $u\in W$, we define the \emph{Schubert cell} $\Xc^u\coloneqq BuB\subset G/B$ and \emph{Schubert variety} $X^u\coloneqq \overline{BuB}\subset G/B$. 
Similarly we define the \emph{opposite Schubert cell} $\Xc_u\coloneqq B^-uB\subset G/B$ and \emph{opposite Schubert variety} $X_u\coloneqq \overline{B^-uB}\subset G/B$. These varieties are related by the identity $\wnaught X^u=X_{\wnaught u}$. For $u\le_B v$, we then define the \emph{Richardson variety} $X^v_u\coloneqq X^v\cap X_u$. The $T$-fixed points of $X^v_u$ are given by the Bruhat interval $[u,v]$.

\subsection{The adjoint group}
\label{subsec:adjointgroup}

In this section, we explain why the space $G/B$ depends only on the root system $\Phi$ for $G$.  
It is nonetheless useful to allow varying choices of $G$, for example in the appendices.  
However, a few results in \Cref{subsec:RelGrass} rely on using the \emph{adjoint group} $G_{ad} \coloneqq G / Z(G)$, where $Z(G)$ denotes the center.  
This is a centerless semisimple reductive group with the same root system and Weyl group as $G$.  As $Z(G) \subseteq T$, $G_{ad}$ has Borel subgroup $B_{ad} = B/Z(G)$ and torus $T_{ad} = T/Z(G)$.

The \emph{adjoint torus} $T_{ad}$ has a rank $n$ character lattice
\[
\Q{T_{ad}} = \ZZ\Phi \subseteq \mathbb{E}.
\]
As $Z(G)$ acts trivially on $G/B$, the adjoint torus acts on $G/B$, and 
\[
G/B=G_{ad}/B_{ad}\qquad\text{as $T_{ad}$-varieties}.
\]
Hence every root system $\Phi$ has a unique generalized flag variety $G/B$ associated to it.

\subsection{Closed subsets of roots and the Bruhat decomposition}
\label{sec:closedsets}\label{sec:BruhatDecomp}

For any $u\in W$ we have a decomposition
$X^u=\bigsqcup_{w\le u}\Xc^w,$ and applying this for $u=\wnaught$ we obtain the \emph{Bruhat decomposition}
$$G/B=\bigsqcup_{w\in W} \Xc^w.$$
This decomposes $G/B$ into Schubert cells, with each $\Xc^w$ $T$-equivariantly isomorphic to the $T$-representation $ \bigoplus_{\tau\in \inv{w}}\mathfrak{g}_{r(\tau)}$.
This isomorphism is non-canonical and we recall how it arises. 

Recall that in~\eqref{eq:one_param_root} we defined a family of one-parameter subgroups $e_{\alpha}(\CC)$, $\alpha \in \Phi^{+}$, by
\[
e_{\alpha}(x) = \mathrm{exp}(xE_{\alpha}).
\]

Say that a subset $C \subseteq \Phi^{+}$ is \emph{closed} if $\alpha, \beta \in C$ and $\alpha + \beta \in \Phi$ implies that $\alpha + \beta \in C$. 
For a closed set $C$, choose an order $C = \{\beta_{1}, \beta_{2}, \cdots, \beta_{|C|}\}$ and define
\[
U_{C} = \left\{ e_{\beta_{1}}(a_{1})\cdots e_{\beta_{|C|}}(a_{|C|}) \suchthat a_{i} \in \CC \right\}
\qquad\text{and}\qquad
N_{C} = \bigoplus_{i = 1}^{|C|} \mathfrak{g}_{\beta_{i}}.
\]
We will apply this construction with
\[
C \subseteq r(\inv{w})=\Phi^+\cap w(\Phi^-),
\]
the set of positive roots corresponding to inversions of $w$.
By~\cite[Lemma 17, Lemma 34]{SteinbergBook}, $U_{C}$ is a unipotent subgroup of $B$ which is independent of the chosen order on $C$.  Moreover there is an (order-dependent) isomorphism of varieties
\begin{equation}
\label{eq:closediso}
\begin{array}{rcl}
N_{C} &\to& U_{C}wB/B \\
a_{1} E_{\beta_{1}} + \cdots + a_{|C|} E_{\beta_{|C|}} & \mapsto & e_{\beta_{1}}(a_{1})\cdots e_{\beta_{|C|}}(a_{|C|})wB/B
\end{array}
\end{equation} 
By relation~\eqref{equation:he_rel}, this isomorphism is $T$-equivariant with respect to the adjoint action on $N_{C}$ and left multiplication by $T$ on $U_{C}wB/B$.

In the extreme case that $C = r(\inv{w})$, we write $U_{w} = U_{r(\inv{w})}$ and $N_{w} = N_{r(\inv{w})}$, so that~\eqref{eq:closediso} gives a $T$-equivariant isomorphism
\[
N_{w} \cong U_wwB=\Xc^w.
\]

\begin{eg}
    We take notation as in \Cref{appendix:GL}. 
    For $\GL_3/B$ with $\alpha=\epsilon_1-\epsilon_2, \beta=\epsilon_2-\epsilon_3, \gamma=\epsilon_1-\epsilon_3$ we take bases of the corresponding weight spaces of $\mathfrak{g}$ to be the elementary matrices $E_{1,2}$, $E_{2,3}$, $E_{1,3}$ where $(E_{i,j})_{k,l}=\delta_{i,k}\delta_{j,l}$. Then
    $$e_\alpha(a)=\begin{bmatrix}1&a&0\\0&1&0\\0&0&1\end{bmatrix},\quad e_\beta(b)=\begin{bmatrix}1&0&0\\0&1&b\\0&0&1\end{bmatrix},\quad e_\gamma(c)=\begin{bmatrix}1&0&c\\0&1&0\\0&0&1\end{bmatrix}.$$
    Then two different parametrizations of $U_{\wnaught}$ are
$$e_\beta(b)e_\alpha(a)e_\gamma(c)=\begin{bmatrix}1&a&c\\0&1&b\\0&0&1\end{bmatrix},\text{ and } e_\alpha(a)e_\beta(b)e_\gamma(c)=\begin{bmatrix}1&a&c+ab\\0&1&b\\0&0&1\end{bmatrix}.$$
We caution the reader that the $T$-equivariant automorphism $(a,b,c)\mapsto (a,b,c+ab)$ between these two parametrizations does not preserve the linear structure on the $T$-representations.
\end{eg}

\subsection{$T$-invariant curves in $G/B$}

As a consequence of the Bruhat decomposition, the $T$-invariant curves in $G/B$ are given by the closures of the $T$-invariant lines $\{e_{\beta}(t)wB\suchthat t\in \mathbb{C}\}$ for $s_\beta\in \inv{w}$. We will need the following fact.

\begin{fact}
    We have $\{e_{\beta}(t)wB\suchthat t\in \mathbb{C}\}\cong \mathbb{A}^1$, its closure contains the unique additional point $s_\beta wB$, and  $\{e_{\beta}(t)wB\suchthat t\in \mathbb{C}\} \cup \{s_\beta wB\} \cong \mathbb{P}^1$.
\end{fact}

Therefore, the $T$-invariant curves are the unique $T$-invariant $\mathbb{P}^1$'s which ``connect'' $T$-fixed points $u,v$ with $v=s_{\beta} u$ for some $\beta\in \inv{v}$. 
We present this more symmetrically as follows.

\begin{defn}
\label{defn:Puv}
    When $v=\tau u$ for some reflection $\tau$, we denote by $\mathbb{P}_{u,v}\subset G/B$ the unique $T$-invariant $\mathbb{P}^1$ which has $T$-fixed points $u$ and $v$.
\end{defn}
Thus if we draw the Cayley graph $\Cay(W,\Reflections)$ generated by reflections, every vertex corresponds to a $T$-fixed point and every edge $uv$ corresponds to $\mathbb{P}_{u,v}$. See for example \Cref{fig:GKM_S3}.
\begin{figure}[!ht]
    \centering
    \includegraphics[scale=1.1]{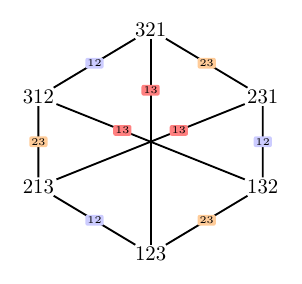}
    \caption{The Cayley graph of $\Cay(S_3,\{(12),(23),(13)\})$ generated by reflections. The vertices correspond to the six $T$-fixed points in $\GL_3/B$ and the edges correspond to the nine $T$-invariant $\mathbb{P}_{u,v}$'s in $\GL_3/B$.}
    \label{fig:GKM_S3}
\end{figure}

\section{Pl\"ucker functions and Pl\"ucker vanishing varieties}
\label{section:Plucker}

In this section we record facts about Pl\"ucker functions that will be used for the remainder of the paper. 
Most of the relevant literature focuses on the special case of type $\mathrm{A}$, or the related case of Grassmannians associated to minuscule roots -- to avoid any confusion we gather what is true in the fullest generality for $G/B$.

Fix, now and for the remainder of the paper, a regular dominant character 
\[
\regweight \in \Lambda^{+} \cap \Q{T}.
\]

\subsection{The Pl\"ucker embedding}

For $\lambda \in \Lambda^{+}\cap \Q{T}$, recall the simple $G$-module $V^{\lambda}$ from Section~\ref{sec:Greps}.  
The \emph{$\lambda$-Pl\"ucker map} is the map 
\[
\begin{array}{rcl}
\Pl^{\lambda}:G/B &\hookrightarrow &\mathbb{P}(V^{\lambda}) \\
gB & \mapsto & \langle g v_{\lambda} \rangle.
\end{array}
\]
When $\lambda = \regweight$, our fixed regular dominant character, the Pl\"{u}cker map is an embedding, which we call the \emph{Pl\"{u}cker embedding}, and we write
\[
\Pl \coloneq \Pl^{\regweight}.
\]
This embedding depends on the choice of $\regweight$, but our results are independent of this choice unless otherwise stated, so we will suppress it from the notation.

We now fix generators for each $1$-dimensional extremal weight space $V^\lambda_{w\lambda}$: for $w \in W$, let
\begin{equation}
\label{eq:weightbasis}
0\ne v_{w\lambda} \in V^{\lambda}_{w \lambda}.
\end{equation}
For any representative $\dot{w} \in N_{G}(T)$ of $w$, we have $\dot{w} v_{\lambda} \in \langle v_{w \lambda}\rangle$ by~\Cref{fact:repfacts}.

\begin{rem}
If the Weyl group $W$ embeds into $G$, as is the case with $G = \GL_{n}$, we can take the more straightforward definition $v_{w\lambda}=wv_\lambda$ without ambiguity. 
\end{rem}

Now define the \emph{$w$-Pl\"{u}cker function} $\Pl_{w}^{\lambda}$ as the coordinate function for $v_{w \lambda}$ with respect to the weight space decomposition of $V^{\lambda}$.  We will abuse notation and also write $\Pl_{w}^{\lambda}$ for the projective coordinate of $G/B$ under $\Pl^{\lambda}$; as we will only consider the vanishing of this coordinate, the arbitrary choices made in Equation~\eqref{eq:weightbasis} do not cause any issues.
For $\lambda = \regweight$, we write  $\Pl_{w}$.  

We now summarize some basic properties of Pl\"ucker functions for the Pl\"ucker embedding and their connection to $T$-fixed points.
\begin{fact}
\label{fact:BasicPlucker}
\leavevmode
\begin{enumerate}[label=(\arabic*)]
\item\label{5.2.1} For $t\in T$ we have $\Pl_w(tgB)= (w\regweight)(t)\Pl_w(gB)$.
 \item \label{5.2.2} For a $T$-invariant variety $X\subseteq G/B$, we have $X^T=\{wB \suchthat w\in W,\,\Pl_w|_X\not\equiv 0\}$.
    \item \label{5.2.3} We have $\Pl_w(gB)\ne 0\Longleftrightarrow wB\in \overline{T\cdot gB}$. 
    In particular the vanishing set $\{\Pl_w=0\}$ is independent of the choice of $\regweight$.
\end{enumerate}
\end{fact}
\begin{proof}
    \ref{5.2.1} follows as $v_{w\regweight}$ belongs to the weight space for $w\regweight$. For~\ref{5.2.2} note that if $\Pl_w(x)\ne 0$, then because the extremal weights $w'\regweight$ are in convex position we are able to find a cocharacter $\psi$ of $T$ which limits $\lim_{t\to 0}\psi(t)x$ to a point where all Pl\"ucker functions vanish except for $\Pl_w$. 
    We claim that only $w\in W$ has this Pl\"ucker vanishing property -- indeed, any $T$-orbit closure in this vanishing set could only have $wB$ as a fixed point. But a projective toric variety with a single fixed point is a singleton, so this $T$-orbit closure must be $\{wB\}$ itself. \ref{5.2.3} follows from \ref{5.2.2} once we note that because $\Pl_w$ is  $T$-equivariant we have $\Pl_w(gB)=0\Longleftrightarrow \Pl_w(\overline{T\cdot gB})=0$. 
\end{proof}
\begin{rem}
    Even though $\{\Pl_w\}_{w\in W}$ is basepoint-free, and hence induces a map $G/B\to \mathbb{P}^{|W|-1}$, this map is not necessarily injective even for a well-chosen $\regweight$, see Appendix~\ref{appendix:GL}.
\end{rem}

\subsection{Relation to Grassmannian Pl\"ucker coordinates}
\label{subsec:RelGrass}
This section is logically independent of the remainder of the paper. 
Here we relate $\Pl$ to the Grassmannian Pl\"ucker coordinates that appear more commonly in the literature, assuming for ease of exposition in this section only that $G$ is semisimple so that $\Q{T}=\Lambda$ and $\mathbb{E}$ is spanned by the simple roots (see the discussion in \Cref{subsec:adjointgroup}).

Enumerate the simple roots in $\Phi$ as $\Delta = \{\alpha_{1}, \ldots, \alpha_{n}\}$, and let the \emph{fundamental weights} $\omega_{1}, \ldots, \omega_{n} \in \Lambda^{+}$ be defined by 
$
(\omega_{i}, \alpha_{j}^{\vee}) = \delta_{i, j}$
where $\delta_{i, j}$ is the Kronecker delta. Then $\Lambda^+$ is the nonnegative integral span of $\omega_1,\ldots,\omega_n$ so we can write
\[
\regweight = \sum k_{i} \omega_{i} \text{ with }k_{i} \in \ZZ_{>0}.
\]

The \emph{Grassmannian Pl\"ucker map} $\Pl^{\omega_i}:G/B\to \mathbb{P}(V^{\omega_i})$ is associated to $\omega_i$, and factors as
$$G/B\to G/P_i\hookrightarrow \mathbb{P}(V^{\omega_i})$$ where $P_i$ is associated to the maximal parabolic subgroup associated to $\omega_i$. 
We now observe that the Pl\"{u}cker coordinates associated to any weight, in particular $\regweight$, decompose as a product of Grassmannian Pl\"ucker coordinates.

\begin{obs}
    If $a,b\in \Lambda^+$ then
    $$\Pl^a_w\Pl^b_w\text{ is a constant nonzero multiple of }\Pl^{a+b}_w.$$
\end{obs}
\begin{proof}
    The map
    $G/B\to \mathbb{P}(V^a)\times \mathbb{P}(V^b)\to \mathbb{P}(V^a\otimes V^b)$
    obtained by composing $(\Pl^a,\Pl^b)$ with the Segre embedding can also be written as $G/B\to \mathbb{P}(V^{a+b})\hookrightarrow \mathbb{P}(V^a\otimes V^b)$
    where the first map is $\Pl^{a+b}$ and the second map is the inclusion of $V^{a+b}\hookrightarrow V^a\otimes V^b$ with the one-dimensional weight space $V^{a+b}_{a+b}$ mapping isomorphically onto $(V^a\otimes V^b)_{a+b}$.
\end{proof}
\begin{cor}
    We have $\Pl^{\regweight}_w=\prod (\Pl^{\omega_i}_w)^{k_i}$.
\end{cor}
In particular by \Cref{fact:BasicPlucker}\ref{5.2.3}, the vanishing locus $\{\Pl_w=0\}\subset G/B$ decomposes as
\begin{equation}\label{eqn:npieces}\{\Pl_w=0\}=\bigcup_{i=1}^n \{\Pl_w^{\omega_i}=0\}\subset G/B.\end{equation}

\subsection{Pl\"ucker vanishing subvarieties}
\label{subsection:vanishingsubvariety}

For any $\mathcal{A}\subset W$, we define the \emph{Pl\"ucker vanishing subvariety}
$$\pluckervanishing{\mathcal{A}}\coloneqq \bigcap_{w\in W\setminus\mathcal A}\{\Pl_w=0\} \subset G/B.$$
We note that this variety is not necessarily irreducible, even when $|W\setminus \mathcal{A}|=1$ (see e.g. \eqref{eqn:npieces}).  

In the following result, say that a $T$-invariant subvariety $X\subset G/B$ is \emph{rigid} if
$$
Y^T\subset X^T \implies Y\subset X\text{ for any $T$-invariant subvariety }Y\subset G/B.
$$ 

\begin{thm}
\label{thm:PluVan}
Let $\mathcal{A} \subseteq W$.  
\begin{enumerate}[label=(\arabic*)]
\item \label{item1} $\pluckervanishing{\mathcal{A}}$ is the union of all $T$-orbit closures $Y\subset G/B$ with $Y^T\subset \mathcal{A}$.
\item \label{item2} The variety $\pluckervanishing{\mathcal{A}}$
is the unique rigid variety with $T$-fixed points given by $\mathcal{A}\subset W$.
\end{enumerate}
\end{thm}
\begin{proof}
For~\ref{item1}, every $T$-invariant variety is the union of the $T$-orbit closures contained within, so it suffices to show that if $\overline{T\cdot y}^T\subset \mathcal{A}$ then $\Pl_w (y)=0$ for all $w\not\in \mathcal{A}$. 
But 
\[\Pl_w (y)=0\Longleftrightarrow \Pl_w|_{\overline{T\cdot y}}\equiv 0 \Longleftrightarrow wB\not\in \overline{T\cdot y}.
\]

For~\ref{item2} we show that $\pluckervanishing{\mathcal{A}}$ is rigid; uniqueness then follows from rigidity.  
Indeed, if $Y^T\subset X^T$, then for $y\in Y$ we have $\overline{T\cdot y}^T\subset X^T$ and hence by~\ref{item1} we have $y\in \overline{T\cdot y}\subset X$.
\end{proof}

The next fact will be essential in our study of Pl\"{u}cker vanishing varieties in later sections.
\begin{fact}
\label{fact:rigidRich}
    For $u, v \in W$, we have $X^v_u=\pluckervanishing{[u,v]}$. More generally, for any $z\in W$ the translated Richardson variety $$zX^v_u=\pluckervanishing{z[u,v]}$$ is the Pl\"ucker vanishing variety associated to the translated Bruhat interval $\mathcal{A}=z[u,v]$, and is therefore rigid.
\end{fact}
\begin{proof}
    We first show that our claim about translated Richardson varieties $zX^{v}_{u}$ reduces to the claim about un-translated $X^{v}_{u}$.  By~\Cref{fact:repfacts}\ref{item:repfacts1}, left translating any Pl\"{u}cker vanishing variety $\pluckervanishing{\mathcal{A}}$ by any representative of $w \in W$ gives the Pl\"{u}cker vanishing variety $\pluckervanishing{w\mathcal{A}}$.

    Next we note that our claim about un-translated Richardson varieties $X^{v}_{u}$ need only be verified for each Schubert variety $X^{v} = X^{v}_{\idem}$. 
    Indeed, this is an immediate consequence of the equalities $X^{v}_{u} = X^{v} \cap X_{u}$,  $X_{u} = \wnaught X^{\wnaught u}$, and $[u,v]=[\idem,v]\cap \wnaught[\idem, \wnaught u]$, once we note that for subsets $\mc{A},\mc{B}\subseteq W$ we have $\pluckervanishing{\mc{A}}\cap\pluckervanishing{\mc{B}} =\pluckervanishing{\mc{A}\cap \mc{B}}$.
    
    Finally, we prove the claim for un-translated Schubert varieties.  
    As the $T$-fixed points of $X^{v}$ are by definition elements of $[\idem, v]$, we need only prove that $X^{v}$ is rigid and apply \Cref{thm:PluVan}\ref{item2} in order to complete the proof.  
    To this end suppose that $Y^{T} \subseteq [\idem, v]$ for some $T$-invariant subvariety $Y$.  
    Each $x \in Y$ belongs to a Schubert cell $\Xc^w$, and because $\Xc^w\cong N_{w}$ any strictly antidominant cocharacter $\psi:\mathbb{C}^*\to T$ (i.e. one which pairs negatively with the roots in $\Phi^-$) has the property that $\lim_{t\to 0}\psi(t)\cdot x = wB$.
    It follows that $wB \in \overline{T\cdot x} \subseteq Y$.  
    By assumption on $Y$ we have $w \in [\idem, v]$, and therefore $x \in X^{v}$.
\end{proof}

Going forward, say that two Bruhat intervals $[u,v]$ and $[u',v']$ are \emph{shape equivalent} if $u'u^{-1}[u,v]=[u',v']$. 
An immediate application of \Cref{fact:rigidRich} is to geometrically interpret shape equivalence.
\begin{defn}\label{defn:movetoid}
Given a Bruhat interval $[u,v]$, we denote $\movetoid{u}{v}\coloneqq u^{-1}[u,v].$
\end{defn}

\begin{cor}\label{cor:shapeequiv} For $u, v, u',v' \in W$, the following are equivalent.
\begin{enumerate}
    \item $[u,v]$ and $[u',v']$ are shape equivalent.
    \item $\movetoid{u}{v}=\movetoid{u'}{v'}$
    \item $u^{-1}X^v_u=(u')^{-1}X^{v'}_{u'}$
\end{enumerate}
\end{cor}
\begin{proof}
The first two are clearly equivalent, and the equivalence of (2) and (3) follows immediately from \Cref{fact:rigidRich}.
\end{proof}

\subsection{Coxeter matroids and moment polytopes}
\label{subsec:CoxMat}

Because the $T$-fixed points $wB\in G/B$ map to the coordinate lines $\langle v_{w\cdot \regweight}\rangle$ in $\mathbb{P}(V^\regweight)$, for an irreducible $T$-invariant subvariety $X\subset G/B$ the \emph{moment polytope} \cite{At82,GS82} is by definition
$$\mu(X)=\operatorname{conv}(\{w\cdot \regweight\suchthat w\in X^T\}).$$
In particular $\mu(G/B)$ is the \emph{$W$-permutahedron}
\[
\operatorname{Perm}_W\coloneqq \operatorname{conv}(\{w\cdot \regweight\suchthat w\in W\}).\]
Since $(X^v_u)^T=[u,v]$, the moment polytope of a Richardson variety is the \emph{twisted Bruhat interval polytope} \cite{TW15}
$$\mu(X^v_u)=P_{[u,v]}\coloneqq \operatorname{conv}(\{w\cdot \regweight \suchthat w\in [u,v]\})\subset \operatorname{Perm}_W.$$

A \emph{Coxeter matroid} $\mathcal{M}\subset W$ is a subset of $W$ such that $w\mathcal{M}$ has a unique Bruhat-maximum element for every $w\in W$, and a \emph{Coxeter matroid polytope} is a polytope $P\subset \Perm_W$ whose vertices are a subset of the vertices of $\Perm_W$ and whose edges are all parallel to roots of $W$. For example, the set $W$ is a Coxeter matroid, and the corresponding Coxeter matroid polytope is $\Perm_W$.

The Gelfand--Serganova theorem (see ~\cite[Theorem~6.3.1]{BGW} for a textbook treatment) gives a bijection between Coxeter matroids and Coxeter matroid polytopes given by
$$\mathcal{M}\mapsto P_{\mathcal{M}}\coloneqq \operatorname{conv}(\{w\cdot \regweight\suchthat w\in \mathcal{M}\})\subset \operatorname{Perm}_W.$$
\begin{fact}[{\cite[\S7 Theorem 1]{GelSer87}}]
The moment polytope $\mu(X)$ of an irreducible $T$-invariant variety $X\subset G/B$ is a Coxeter matroid polytope.
\end{fact}
Since this result is not stated in this precise way in loc.~cit.~we include a sketch of the proof.
\begin{proof}[Proof Sketch]
    $\mu(X)=\mu(\overline{T\cdot x})$ for generic $x\in X$. Each edge of the moment polytope corresponds to a $T$-invariant curve in $\overline{T\cdot x}$, which is of the form $\mathbb{P}_{u,\tau u}$ for a reflection $\tau$, and the moment polytope of such a curve is a line segment in the direction $r(\tau)\in \Phi^+$.
\end{proof}
Because $(zX^v_u)^T=z[u,v]$, we deduce the following combinatorial corollary.
\begin{cor}
    For $u\le_B v$, the translated Bruhat interval $z[u,v]$ is a Coxeter matroid.
\end{cor}

\section{Noncrossing partitions via translated Bruhat intervals}
\label{section:noncrossingandBruhet}
In this section we introduce the central combinatorial observation that will connect noncrossing partitions and Bruhat combinatorics.
Given a Bruhat interval $[u,v]$ recall that $\movetoid{u}{v}\coloneqq u^{-1}[u,v]$. 
We show the following.

\begin{thm}
\label{thm:HasseUnion}
Let $w\in W$ be such that $w \le_{B} wc$. 
Then $\movetoid{w}{wc}\subseteq \NC(W,c)$, and considered as an induced subposet of $(\NC(W,c),\le_{\Reflections})$, it is poset-isomorphic to the Bruhat interval $([w,wc],\le_B)$ via left multiplication by $w^{-1}$:
\[
[w,wc]\ \xrightarrow{\ \sim\ }\ \movetoid{w}{wc}
\qquad x\longmapsto w^{-1}x.
\]
Both intervals have rank $n$, and in particular $\ell(w) + \ell(c) = \ell(wc)$.

Moreover, every maximal chain in $\NC(W, c)$ lies in some translated Bruhat interval: for each
\[
\idem \lessdot_{\Reflections} \tau_{1} \lessdot_{\Reflections} \tau_{1}\tau_{2} \lessdot_{\Reflections} \cdots \lessdot_{\Reflections} \tau_{1}\tau_{2}\cdots\tau_{n} = c,
\]
there exists a $w \in W$ such that $w \lessdot_{B} w\tau_{1} \lessdot_{B} w\tau_{1}\tau_{2} \lessdot_{B} \cdots \lessdot_{B} w\tau_{1}\tau_{2}\cdots\tau_{n}=wc$.
\end{thm}
\begin{cor}
    $w\le_B wc$ is equivalent to $\ell(wc)=\ell(w)+\ell(c)$. In particular, in either of these two equivalent situations we may write the product $wc$ in the length additive  notation $w\cdot c$.
\end{cor}
\begin{proof}[Proof of \Cref{thm:HasseUnion}]
We deal with the two halves of the statement separately. 
For the first half, fix $w\in W$ with $w \le_{B} wc$.

\smallskip
\noindent\textbf{Step 1: containment and induced-subposet identification.}
Since Bruhat order is graded by $\ell$, every $x\in[w,wc]$ lies on some maximal chain from $w$ to $wc$.  
Moreover, the product of the labels of each chain give a factorization of $c$ into reflections.  
Thus the length of this chain must be exactly $n$: it cannot be larger as $\ell(wc) \le \ell(w) + n$, and it cannot be smaller as $\ell_{\Reflections}(c) = n$.  
Thus translating such a chain by $w^{-1}$ yields a maximal chain in absolute order from $\idem$ to $c$ passing through
$w^{-1}x$. Hence $\movetoid{w}{wc}\subseteq \NC(W,c)$, and the length of this chain must be $\ell_{\Reflections}(c) = n$.

Moreover, along any maximal Bruhat chain
$w=x_0\lessdot_B x_1\lessdot_B\cdots\lessdot_B x_{\ell(c)}=wc$,
$w^{-1}$-translation gives a maximal absolute-order chain
$\idem \lessdot_{\Reflections} w^{-1}x_1\lessdot_{\Reflections}\cdots\lessdot_{\Reflections} c$.
In particular, for $x\in[w,wc]$:
\begin{equation}\label{eq:rankmatch}
\ell_{\Reflections}(w^{-1}x)=\ell(x)-\ell(w).
\end{equation}
Since both posets are graded (by $\ell$ on $[w,wc]$ and by $\ell_{\Reflections}$ on $\NC(W,c)$), the equality in 
\eqref{eq:rankmatch} forces the translation map $x\mapsto w^{-1}x$ to have the property that
for $x,y\in[w,wc]$,
\[
x\lessdot_B y \iff w^{-1}x \lessdot_{\Reflections} w^{-1}y.
\]
This is exactly the statement that $\movetoid{w}{wc}$ with the induced order from $\NC(W,c)$, is poset isomorphic to $[w,wc]$
via $x\mapsto w^{-1}x$.

\smallskip
\noindent\textbf{Step 2: every absolute chain comes from a translated Bruhat chain.}
Let $\idem \lessdot_{\Reflections} \tau_{1} \lessdot_{\Reflections} \tau_{1}\tau_{2} \lessdot_{\Reflections} \cdots \lessdot_{\Reflections} \tau_{1}\tau_{2}\cdots\tau_{n} = c$ be a maximal chain in $\NC(W, c)$. Define roots
\[
\beta_i \coloneqq  \tau_1\cdots\tau_{i-1}(r(\tau_i))\qquad(1\le i\le n).
\]
Since $s_{\beta_1}\cdots s_{\beta_{n}}=c$ is a minimal reflection factorization we have that $\beta_1,\dots,\beta_n$ are linearly independent.
In particular there exists $w\in W$ so that $w(\beta_i)\in \Phi^+ $ for $1\leq i\leq n$.
We claim that for this choice of $w$ we get a maximal chain in Bruhat order that we seek.
Setting $y_i\coloneqq w\tau_1\cdots\tau_i$ we have
\[
y_{i-1}(r(\tau_i)) = w(\beta_i)\in\Phi^+,
\]
so $y_{i-1}<_B y_i$ for all $i$. As there are $n=\ell(c)$ steps, the chain $w<_B y_1<_B\cdots<_B wc$ is maximal, and every consecutive relation corresponds to a cover.
In particular $wu\lessdot_B wv$. Also, the same length count gives $\ell(wc)=\ell(w)+\ell(c)$.

This proves the $(\Rightarrow)$ direction in the final equivalence of the theorem. The converse $(\Leftarrow)$ direction
is exactly the cover correspondence established in the first half.
\end{proof}

Recall that there is an EL-labeling of $\NC(W,c)$ where the edge $xy$ is labeled with the reflection $x^{-1}y=y^{-1}x$. Decreasing chains under the $c$-reflection order are called $c$-decreasing chains, and there are $\cat{W}^+$-many $c$-decreasing maximal chains in $\NC(W,c)$ (\Cref{cor:numbdecreasing}).
\begin{cor}
\label{prop:uniquecdec}
A translated interval $\movetoid{w}{w\cdot c}\subset \NC(W,c)$ contains a unique $c$-decreasing maximal chain. 
Furthermore each $c$-decreasing maximal chain in $\NC(W,c)$ is contained in $\movetoid{w}{w\cdot c}$.
\end{cor}
\begin{proof}
 By \Cref{thm:Dyer} there is a unique $c$-decreasing maximal chain in $[w,wc]$ under the edge labeling of $xy$ by $x^{-1}y=y^{-1}x$. 
 By Theorem~\ref{thm:HasseUnion}, left-multiplication by $w^{-1}$ gives a poset isomorphism between $\movetoid{w}{wc}$ and $[w,wc]$,  which in view of  $(wx)^{-1}wy=x^{-1}y$, induces an edge labeling on $\movetoid{w}{wc}$.
So we conclude there is a unique $c$-decreasing maximal chain in $\movetoid{w}{wc}$. 
Finally, Theorem~\ref{thm:HasseUnion} also implies every maximal chain in $\NC(W,c)$ belongs to at least one $\movetoid{w}{wc}$.
\end{proof}

\section{The Coxeter Flag Variety}
\label{section:CFLc}
In this section we introduce the $c$-Coxeter Richardson varieties and the $c$-Coxeter flag variety.

\subsection{$c$-Coxeter Richardson varieties}
A Bruhat interval $[u,v]$ is said to be \emph{toric} if the Richardson variety $X^v_u$ is a toric variety with respect to $T$. Because Richardson varieties are normal \cite[Proposition 1.22]{speyer2024richardson}, toric Richardson varieties are normal projective toric varieties, and so each such $X^v_u$ is the toric variety $X_\Sigma$ of the complete fan $\Sigma$ arising as the normal fan of the moment polytope.
\begin{prop}
\label{prop:invToricSub}
    Each face of  $P_{[u,v]}$ is a polytope of the form $P_{[u',v']}$ for a subinterval $[u',v']\subset [u,v]$. Furthermore the $T$-invariant subvarieties of $X^v_u$ with $[u,v]$ toric are exactly the Richardson varieties $X^{v'}_{u'}$ with  $[u',v']\subset [u,v]$.
\end{prop}
\begin{proof}
    The first part is \cite[Theorem 7.13]{TW15}. The second part follows as the faces of the moment polytope are in bijection with $T$-invariant subvarieties.
\end{proof}

Now recall from the introduction that a $c$-Coxeter Richardson variety is one of the form $X^{w\cdot c}_{w}$.  

\begin{thm}
\label{thm:CoxeterRichIsToric}
Every $c$-Coxeter Richardson variety $X^{w\cdot c}_w$ is toric.
\end{thm}
\begin{proof}
By \cite[Proposition 7.12]{TW15}, for any Bruhat interval $[u,v]$ one has
\[
\dim P_{[u,v]}\le \ell(v)-\ell(u),
\]
and equality holds if and only if $[u,v]$ is toric.
Setting $u=w$ and $v=wc$ implies that 
\[
    \dim P_{[w,wc]}\le \ell(wc)-\ell(w)=n
\]
Thus it suffices to show that $\dim P_{[w,wc]}\ge n$.

 Write $c=s_1\cdots s_n$ as a reduced product of simple reflections.
Consider the sequence of vertices
\begin{align}\label{eq:seq_of_vertices}
w\cdot\regweight,\quad
ws_1\cdot\regweight,\quad
ws_1s_2\cdot\regweight,\quad \ldots,\quad
ws_1\cdots s_n\cdot\regweight = wc\cdot\regweight    
\end{align}
in $P_{[w,wc]}$.
Consider the vectors $\upsilon_i$ for $1\le i\le n$ obtained by taking consecutive differences:
\[
\upsilon_i \coloneqq ws_1\cdots s_{i-1}\cdot\regweight - ws_1\cdots s_{i}\cdot\regweight
= ws_1\cdots s_{i-1}\bigl(\idem-s_i)\cdot\regweight.
\]
Since $\regweight$ is regular dominant we have $
(\idem-s_i)\cdot\regweight
= ( \regweight, r(s_i)^\vee )\, r(s_i)
$
with $(\regweight, r(s_i)^\vee )>0$.
Thus, 
$\upsilon_i$ is parallel to $
ws_1\cdots s_{i-1}\bigl(r(s_i)\bigr).$
Equivalently, if we translate by $w^{-1}$, the directions are given by:
\[
r(s_1),\quad s_1 r(s_2),\quad s_1s_2 r(s_3),\ \ldots,\ s_1\cdots s_{n-1} r(s_n).
\]

We now check that these $n$ vectors are linearly independent. 
We have for any root $\beta$ that 
\[
s_j(\beta) = \beta - ( \beta, r(s_j)^\vee) \, r(s_j),
\]
which inductively implies that $
s_1\cdots s_{i-1}\bigl(r(s_i)\bigr) $
is a linear combination of $r(s_1),\dots,r(s_i)$ with nonzero coefficient on $r(s_i)$ and no contribution from $r(s_j)$ for $j>i$. 
Hence a routine triangularity argument implies that these vectors are linearly independent.  Left-translating by $w$ does not affect
linear independence, and hence so are the $\upsilon_i$ vectors as well.
Therefore the $n+1$ vertices in~\eqref{eq:seq_of_vertices}
are affinely independent, and we obtain
$\dim P_{[w,wc]} \ge n$, as desired.
\end{proof}

\subsection{The toric complex $\cfl_c$}
\label{sec:complexcfl}

Recall that \emph{$c$-Coxeter flag variety} is the union $$\cfl_c=\bigcup w^{-1}X^{w\cdot c}_w\subset G/B.$$
By \Cref{thm:CoxeterRichIsToric}, each $w^{-1}X^{w\cdot c}_w$ is a $T$-orbit closure in $G/B$ so $\cfl_c$ is a toric complex. 
\begin{thm}\label{thm:FirstCFlFacts} 
The following hold.
    \begin{enumerate}[label=(\arabic*)]
        \item \label{it:7.3.1}
        We have $\cfl_c^T=\NC(W,c)$, and the $T$-invariant curves $\mathbb{P}_{u,v}$ (\Cref{defn:Puv}) contained in $\cfl_c$ are $$\{\mathbb{P}_{u,v}\suchthat u,v\text{ adjacent in }\NC(W,c)\}.$$
        In particular (from the fixed point statement) we have 
        $\cfl_c\subset \pluckervanishing{\NC(W,c)}.$
        \item \label{it:7.3.2}
        Distinct $T$-orbit closures in $\cfl_c$ have distinct fixed point sets $X^T\subset \NC(W,c)$.
        \item \label{it:7.3.3} For $T$-orbit closures $Y$ and $Z$ in $\cfl_c$ we have 
        \[
            Y^T\subset Z^T \Longleftrightarrow Y \subset Z.
        \]
    \end{enumerate}  
\end{thm}
\begin{proof}
\ref{it:7.3.1} follows from \Cref{thm:HasseUnion}.
    For~\ref{it:7.3.2} and ~\ref{it:7.3.3} it suffices to show that $T$-orbit closures in $\cfl_c$ are rigid in the sense of \Cref{thm:PluVan}. By \Cref{prop:invToricSub}, every $T$-orbit closure in $\cfl_c$ is a translated Richardson $w^{-1}X_u^v$ for some $w\in W$ with $\ell(wc)=\ell(w)+\ell(c)$, and so we conclude by \Cref{fact:rigidRich}.
\end{proof}
Later in \Cref{thm:everying_about_cfl} we will improve Theorem~\ref{thm:FirstCFlFacts}\ref{it:7.3.1}, showing the equality $\cfl_c=\pluckervanishing{\NC(W,c)}$.

\subsection{$(W,c)$-polypositroids}
\label{sec:polypos}
We now proceed to identify the faces of the polytopes in $\Complex(\cfl_c)$ from \Cref{sec:Results} as members of the class of $(W,c)$-polypositroids introduced by Lam--Postnikov \cite{LP20}. The results that follow are a substantial generalization of \cite[Theorem 7.6]{BGNST2}.

\begin{defn}[{\cite[Definition 13.1]{LP20}}]
    A Coxeter matroid polytope is called a \emph{$(W,c)$-polypositroid} if it can be defined by inequalities $\xl\cdot \vec{n} \le a$ using vectors $\vec{n}$ with the property that $(\idem-c)\;\vec{n}$ is parallel to a root in $\Phi$.  
\end{defn}
We now show that the moment polytopes arising in this complex are $(W,c)$-polypositroids. 

\begin{thm}[\Cref{thm:polyposSect2}]
\label{thm:polyposSect8}
All faces of polytopes in $\Complex(\cfl_c)$ are $(W,c)$-polypositroids.
\end{thm}
\begin{proof}
Any face of a $(W,c)$-polypositroid is a $(W,c)$-polypositroid, so it suffices to show that the top-dimensional $w^{-1}P_{[w,wc]}$ are $(W,c)$-polypositroids.
By \Cref{prop:invToricSub}, the facets of $w^{-1}P_{[w,wc]}$ are exactly the faces of the form $w^{-1}P_{[w,wv]}$ with $w\le_B wv\lessdot_B wc$ and $w^{-1}P_{[wu,wc]}$ with $w\lessdot_B wu \le_B wc$, so we want to show that the normal vectors $\vec{n}$ to these facets  have the property that $(\idem-c)\vec{n}$ are parallel to roots in $\Phi$.
    
We first consider facets of the form $w^{-1}P_{[w,wv]}$ with $w\le_B wv\lessdot_B wc$. Take any maximal chain 
\[
\mc{C}=w \lessdot_{B} w\tau_1 \lessdot_{B} w\tau_2\tau_1 \lessdot_{B} \ldots \lessdot_{B} w\tau_{n-1}\cdots \tau_1 \lessdot_{B} w\tau_n\cdots \tau_1=wc
\qquad\text{with $\tau_{n-1}\cdots \tau_1 = v$}.
\]
Since $\tau_n\cdots \tau_1$ is a minimal reflection factorization for $c$, the collection of roots $\{r(\tau_1),\dots, r(\tau_n)\}$ is linearly independent.
Translating $\mc{C}$ by $w^{-1}$, the corresponding edges in $w^{-1}P_{[w, wc]}$ are in the linearly independent directions  $r(\tau_1),\ldots,r(\tau_{n-1})$.
Therefore the facet $w^{-1}P_{[w,wv]}$ is spanned by $r(\tau_1),\ldots,r(\tau_{n-1})$. 
The normal vector $\vec{n}$ is therefore fixed by $\tau_{1}, \ldots, \tau_{n-1}$, so  we have 
\[
(\idem-c)\,\vec{n}\;=\;(\idem-\tau_n)\,\vec{n} \;=\; (\vec{n}, r(\tau_{n})^{\vee}) \; r(\tau_{n}),
\]
which is parallel to $r(\tau_n)\in \Phi$.

For facets of the form $w^{-1}P_{[wu,wc]}$ with $w\lessdot_B wu \le_B wc$, we take a maximal chain 
\begin{equation*}
w \lessdot_{B} w\tau_1 \lessdot_{B} w\tau_2\tau_1 \lessdot_{B} \ldots \lessdot_{B} w\tau_{n-1}\cdots \tau_1 \lessdot_{B} w\tau_n\cdots \tau_1=wc
\qquad\text{with $\tau_1 = u$}.
\end{equation*}
Arguing as before we see that the facet $w^{-1}P_{[wu,wc]}$ is spanned by the linearly independent roots given by $r(\tau_2),\ldots,r(\tau_{n})$, and thus the facet normal $\vec{n}$ is fixed by $\tau_2,\dots,\tau_n$.
  
Now we write 
\[
c=\tau\tau_2\cdots \tau_n \text{ where } \tau=(\tau_n\cdots \tau_2)\;\tau_1\;(\tau_n\cdots \tau_2)^{-1}\in \Reflections
.\]
We then have $(\idem-c)\;\vec{n}=(\idem-\tau)\;\vec{n}$, which is parallel to $r(\tau)\in \Phi$.
\end{proof}

\section{Pl\"ucker vanishing}
\label{sec:ClusterCharts}

In this section we characterize the Coxeter flag variety using the combinatorics of noncrossing partitions and clusters. Both \Cref{maintheorem:PlVanish} and \Cref{maintheorem:Paving} are established in \Cref{thm:everying_about_cfl}.  
Recall the map $\mathrm{Clust}^{+}: \NC(W, c) \to \mathrm{Cl}_{c}^{+}$ from Section~\ref{sec:clusternoncrossing}.  
We show in \Cref{fact:invncClosed} that the set $r(\invnc{u})\subseteq r(\inv{u})$ is a closed set of roots, allowing us to make the following definition.

\begin{defn}
For $u\in \NC(W,c)$ we define the \emph{Coxeter Schubert cell} to be
\[
\Xc^u_{\NC}
\coloneqq U_{r(\invnc{u})} uB/B\subset U_{r(\inv{u})}uB/B=BuB/B=\Xc^u
\]
and the \emph{Coxeter Schubert variety} $X^u_{\NC}$ to be the closure of $\Xc^u_{\NC}$.
\end{defn}
\begin{rem}
Recalling \Cref{sec:BruhatDecomp}, there is a $T$-equivariant isomorphism $N_{r(\invnc{u})}\cong \Xc^u_{\NC}$. A similar construction is given by Gelfand--Graev--Postnikov in~\cite{GGP97}, which studies the charts
\[
\bigoplus_{\alpha\in \nctocluster(u)} \mathbb{C}_{-\alpha}\cong u^{-1}\Xc^u_{\NC}\subset U^-B,
\]
where $U^{-} = U_{\Phi^{-}}$ is the unipotent subgroup of the opposite Borel $B^{-}$. These charts in \cite{GGP97} are considered in the coordinate chart compactification of $U^-\cong \mathbb{A}^{\ell(\wnaught)}\subset \mathbb{P}^{\ell(\wnaught)}$, which makes their closures isomorphic to projective spaces. Our closures $X^u_{\NC}$ are not these projective spaces because our charts are compactified under the Pl\"ucker embedding into $\mathbb{P}(V^{\lambda_{reg}})$.
\end{rem}

An \emph{affine paving} of a closed subvariety $X$ is a sequence of closed subvarieties $X_1\subset X_2\subset \cdots \subset X_m=X$ such that $X_k\setminus X_{k-1}$ is isomorphic to an affine space for all $k$.

\begin{thm}[{\Cref{maintheorem:PlVanish}, \Cref{maintheorem:Paving}}]
\label{thm:everying_about_cfl}
We have
\[
\cfl_c=\pluckervanishing{\NC(W,c)}=\bigsqcup_{u\in \NC(W,c)}\Xc^u_{\NC}.
\]
Moreover, for any linear extension $u_1,u_2,\ldots,u_{\cat{W}}$ of the Bruhat order restricted to $\NC(W,c)$,
$$X_i=X^{u_i}\cap \cfl_c=\bigsqcup_{k\le i}\Xc^{u_k}_{\NC}$$
defines an affine paving of $\cfl_{c}$.
\end{thm}
\begin{cor}\label{cor:newguy}
There are exactly $\cat{W}^+$-many distinct  $w^{-1}X^{w\cdot c}_w$ comprising $\cfl_c$, and if  $u\in \movetoid{w}{w\cdot c}$ is the Bruhat-maximum element, then $u\in \NC(W,c)^+$ and $w^{-1}X_w^{w\cdot c}=X^u_{\NC}$.
\end{cor}
\begin{proof}
    The irreducible components of $\pluckervanishing{\NC(W,c)}$ are the closures of the top-dimensional $X^u_{\NC}$, and the Bruhat-maximum element of $X^u_{\NC}$ is $u$. Because $\dim X^u_{\NC}=|\invnc{u}|\le n$ with equality if and only if $u\in \NC(W,c)^+$, the result follows.
\end{proof}
In later sections we will determine the $u\in \NC(W,c)^+$ with $w^{-1}X^{w\cdot c}_w=X^u_{\NC}$.
\begin{rem}
\label{rem:ChapotonFTriangle}
As is well-known, the existence of an affine paving of a variety $X$ tells us that the $i$th Betti number $H^{2i}(X)=H_{2i}(X)$ is the number of pieces isomorphic to $\mathbb{A}^i$. In the case of $\cfl_c$ this is the the number of positive $c$-clusters of size $i$. These numbers are enumerated by the diagonal of the $F$-triangle of Chapoton \cite[\S 2]{Cha04}.
In type $\mathrm{A}$ \cite[\S 4]{Cha04} one obtains the \textit{ballot numbers} occurring in the Catalan triangle \cite[A009766]{OEIS}. 
In type $\mathrm{B}$ \cite[\S 5]{Cha04} one obtains \cite[A059481]{OEIS}.

The $\mathrm{B}_2$ case is instructive. Using the table in Example~\ref{eg:B2examplewithdiagram} we have that the Betti numbers are $1,2,3$. The Betti numbers of $G/B$  on the other hand are $1,2,2,2,1$. 
This implies that we do not have a surjection $H^\bullet(G/B)$ to $H^\bullet(\cfl_c)$ (see also the examples in \Cref{appendix:SO} and \Cref{appendix:C}).
\end{rem}

The proof of Theorem~\ref{thm:everying_about_cfl} is involved and occupies the next two subsections: in Section~\ref{sec:pluckerstuff} we prove several technical results about Pl\"{u}cker coordinates, and in Section~\ref{sec:cBruhatcells} we relate these results to Coxeter Schubert cells and $\cfl_{c}$.  

\subsection{Coordinate subspaces of $\Xc^u$ associated to strongly closed subsets}
\label{sec:pluckerstuff}

Say that a subset $C\subset \Phi^+$ of roots is \emph{strongly closed} if the only roots in the nonnegative cone spanned by $C$ are $C$ itself.  
Every strongly closed subset $C$ is closed in the sense of Section~\ref{sec:closedsets}, but the converse is not true.  

\begin{eg}
In the type $\mathrm{B}_{2}$ root system $\Phi = \{\pm \epsilon_{1} \pm \epsilon_{2} \} \cup \{\pm \epsilon_{1}, \pm \epsilon_{2}\}$, the set $C = \{\epsilon_{1} + \epsilon_{2}, \epsilon_{1} - \epsilon_{2}\}$ is closed but not strongly closed, as
\[
\epsilon_{1} = \frac{1}{2}\left(\epsilon_{1} + \epsilon_{2}\right) + \frac{1}{2} \left( \epsilon_{1} - \epsilon_{2} \right) \in \operatorname{Cone}(C) \cap \Phi.
\]
This property is highly dependent on the root system: the same set is both closed and strongly closed in type $\mathrm{D}_{2}$ ($\Phi = \{\pm \epsilon_{1} \pm \epsilon_{2} \}$), and is neither closed nor strongly closed in type $\mathrm{C}_{2}$ ($\Phi = \{\pm \epsilon_{1} \pm \epsilon_{2} \} \cup \{\pm2 \epsilon_{1}, \pm 2\epsilon_{2}\}$).
\end{eg}

\begin{fact}[\emph{Cluster cone property}]
\label{fact:invncClosed}
For $u\in \NC(W,c)$ the subset $r(\invnc{u})\subset r(\inv{u})$ is strongly closed.
\end{fact}
\begin{proof}
As $r(\invnc{u}) = \{-u^{-1} \alpha \suchthat \alpha \in \mathrm{Clust}^{+}(u)\}$ and the action of $W$ preserves strong closure, it suffices to show that the positive subset of any $c$-clusters are closed.  Recall that the $c$-cluster fan $\mathrm{Cl}_{c}$ is a complete, simplicial fan that includes every almost-positive root as a generating ray.  
Thus, every positive root in each cone of $\mathrm{Cl}_{c}^{+}$ must be one of its generators, for any others would contradict the assumption that the cones form a fan.  
\end{proof}

The aim of this section is to prove the following. 

\begin{prop}
\label{cor:CoordinateSubspace}\label{cor:coordinatesubspace2}
    For $w \in W$ and $C\subset r(\inv{w})$ a strongly closed subset, the isomorphism $N_{w}\cong U_{w}wB/B=\Xc^w$ from Section~\ref{sec:BruhatDecomp} restricts to an isomorphism 
    \[
    N_C\cong U_CwB/B=\Xc^{w} \cap \bigcap_{\substack{\tau\in \inv{w} \\ r(\tau) \notin C}} \{\Pl_{\tau w}=0\}.
    \]
    In particular for $u\in \NC(W,c)$ we have
    $$\Xc_{\NC}^u=\Xc^u\cap \bigcap_{\tau\in \inv{u}\setminus \invnc{u}}\{\Pl_{\tau u}=0\}.$$
\end{prop}
\begin{eg}
    We illustrate this proposition and its proof with an example, using the notation and matrices from type $\mathrm{A}$ as in \Cref{appendix:GL}. Take $c=s_3s_2s_1\in S_4$, and $u=\wnaught=4321$. Then
    $$X^{\wnaught}=\begin{bmatrix}1&a&b&c\\0&1&d&e\\0&0&1&f\\0&0&0&1\end{bmatrix}\wnaught B/B=\begin{bmatrix}c&b&a&1\\e&d&1&0\\f&1&0&0\\1&0&0&0\end{bmatrix}B/B$$
    with $a,b,c,d,e,f$ corresponding to positive roots $\epsilon_{i}-\epsilon_j$ for the transpositions given by $$(i,j)=(1\,2),(1\,3),(1\,4),(2\,3),(2\,4),(3\,4)$$  respectively. The Pl\"ucker coordinates associated to $(i\,j)\wnaught$ are given by

    \medskip
    
    \begin{center}
    \begin{tabular}{ c|c|c|c|c|c|c}
        $(i\,j)$& $(1\,2)$ & $(1\,3)$ & $(1\,4)$ & $(2\,3)$& $(2\,4)$ & $(3\,4)$\\\hline $\Pl_{(i\,j)\wnaught}$&$a$&$b(ad-b)$&$c(cd-be)(bf+ae-c-adf)$&$d$&$e(df-e)$&$f$
    \end{tabular}
    \end{center}

    \medskip
    
    each of which is computed as a product of determinants of submatrices using the first $\ell\in\{1,2,3\}$ columns and rows $(i,j)\wnaught\{1,\ldots,\ell\}$.
    We have $\inv{\wnaught}\setminus \invnc{\wnaught}=\{(1\,2),(2\,4),(3\,4)\}$, so $$\Xc^{\wnaught}_{\NC}=\begin{bmatrix}c&b&0&1\\0&d&1&0\\0&1&0&0\\1&0&0&0\end{bmatrix}B/B,$$
    obtained by setting $a=e=f=0$. On the other hand, the corresponding Pl\"ucker equations indexed by $\inv{\wnaught}\setminus \invnc{\wnaught}$ are $a,e(df-e),f$. 
    
    As in the proof, we will order these equations by removing successive extreme ray generators to ensure---as proved in Lemma~\ref{lem:PluckerPower} below---that all but one variable is eliminated.  
    Note that
    \begin{itemize}
        \item $\epsilon_1-\epsilon_2$ is an extreme ray of $\mathbb{R}_{\ge 0}\Phi^+$,
        \item $\epsilon_3-\epsilon_4$ is an extreme ray of $\mathbb{R}_{\ge 0}(\Phi^+\setminus \{\epsilon_1-\epsilon_2\})$, and
        \item $\epsilon_2-\epsilon_4$ is an extreme ray of $\mathbb{R}_{\ge 0}(\Phi^+\setminus \{\epsilon_1-\epsilon_2,\epsilon_3-\epsilon_4\})$,
    \end{itemize}  
    so if we impose vanishing conditions in the same order, i.e. $(1\,2),(3\,4),(2\,4)$ or $a=0$, $f=0$, $e(df-e)=0$, we see that $e(df-e)=0$ becomes $-e^2 = 0$, so $a=f=e(df-e)=0$ defines the same subspace as $a = f = e = 0$.
\end{eg}
The proof follows the next two technical results.  The first, which we suspect is known to experts in the area, describes the behavior of Pl\"{u}cker coordinates under the action of $W$.

\begin{lem}
\label{lem:PluckerPermute}
Let $\alpha \in \Phi^{+}$ and  $w \in W$.  
There exists a nonzero constant $k \in \CC^{\times}$ such that 
\[
\mathrm{Pl}_{s_{\alpha} w}(v) = k x^{-( \alpha^{\vee}, w\regweight )}
\mathrm{Pl}_{w}(s_{\alpha}(-x) v)
\qquad\text{for all $x \in \CC^{\times}$ and $v \in V^{\regweight}$}.
\]
\end{lem}
\begin{proof}
Recall the fixed extremal weight vectors $\{v_{u\regweight} \suchthat u \in W\}$ from \eqref{eq:weightbasis}, and recall that $s_{\alpha}(-x)$ sends each of these vectors to a scalar multiple of another; we will calculate this scalar for $v_{s_{\alpha} w \regweight}$.  
We begin by applying the Chevalley relations from Section~\ref{sec:Greps}:
\[
s_{\alpha}(-x) v_{s_{\alpha} w \regweight}
\stackrel{\eqref{equation:h_rel}}{=}
h_{\alpha^{\vee}}(-x)s_{\alpha}(-1)^{-1}  v_{s_{\alpha} w \regweight}
\stackrel{\eqref{equation:sdefn}, \eqref{equation:ee_rel1}}{=}
h_{\alpha^{\vee}}(-x)s_{\alpha}(1) v_{s_{\alpha} w \regweight}.
\]
Define  $k \in \CC^{\times}$  by the equation $k s_{\alpha}(1) v_{s_{\alpha} w \regweight} = (-1)^{(\alpha^{\vee}, w \regweight)} v_{w\regweight}$, so that the above is equal to
\begin{align*}
k^{-1}  (-1)^{(\alpha^{\vee}, w \regweight)} h_{\alpha^{\vee}}(-x) v_{w \regweight}
&= k^{-1}   (-1)^{(\alpha^{\vee}, w \regweight)}  (-x)^{( \alpha^{\vee}, w \regweight )} v_{w \regweight} \\
&= k^{-1}  x^{( \alpha^{\vee}, w \regweight )} v_{w \regweight}.
\end{align*}
Thus $\mathrm{Pl}_{w}(s_{\alpha}(-x) v) = k^{-1} x^{( \alpha^{\vee}, w\regweight )} \mathrm{Pl}_{s_{\alpha} w}(v)$, from which the claim follows.
\end{proof}

Before stating the second lemma, we make some general observations about the computation of Pl\"{u}cker coordinates.  
For $w \in W$ and $C = \{\beta_{1}, \ldots, \beta_{|C|}\} \subseteq r(\inv{w})$,  
\[
e_{\beta_1}(x_{\beta_1})\cdots e_{\beta_{|C|}}(x_{\beta_{|C|}}) v_{w\regweight} 
=
\sum_{\vec{a} \in \ZZ_{\ge 0}^{|C|}} x_{\beta_{1}}^{a_{1}} \cdots x_{\beta_{|C|}}^{a_{|C|}}  E_{\beta_{1}}^{a_{1}} \cdots E_{ \beta_{|C|}}^{a_{|C|}} v_{w \regweight}.
\]
With the caveat that all but finitely many summands above are zero,  we observe that the summand for $\vec{a} \in \ZZ_{\ge 0}^{|C|}$ belongs to the $w\regweight + a_{1}\beta_{1} + \cdots + a_{|C|} \beta_{|C|}$ weight space.    
Thus for $u \in W$, 
\begin{equation}
\label{eq:PluckerPrecision}
\Pl_{u}\left( e_{\beta_1}(x_{\beta_1})\cdots e_{\beta_{|C|}}(x_{\beta_{|C|}}) v_{w\regweight} \right)
=
\Pl_{u}\left( 
\sum_{\text{solutions $\vec{a}$}} x_{\beta_{1}}^{a_{1}} \cdots x_{\beta_{|C|}}^{a_{|C|}}  E_{\beta_{1}}^{a_{1}} \cdots E_{ \beta_{|C|}}^{a_{|C|}} v_{w \regweight}
\right),
\end{equation}
where the sum in the right hand side is over all $\vec{a} \in \ZZ_{\ge 0}^{|C|}$ for which 
\[
u \regweight - w \regweight 
=
a_{1}\beta_{1} + \cdots + a_{|C|} \beta_{|C|}.
\]

\begin{lem}
\label{lem:PluckerPower}
Let $w \in W$ and $C\subset r(\inv{w})$.  
For any ordering $\beta_1,\ldots, \beta_{|C|}$ of  $C$ so that $\alpha=\beta_{|C|}$ is an extremal ray generator of $\operatorname{Cone}(C)$, and any collection $\{x_{\beta_{i}} \in \CC \suchthat \beta_{i} \in C\}$, 
\[
\mathrm{Pl}_{s_{\alpha} w} \left(  e_{\beta_1}(x_{\beta_1})\cdots e_{\beta_{|C|}}(x_{\beta_{|C|}}) v_{w\regweight} \right)
=
k x_{\alpha}^{-( \alpha^{\vee}, w\regweight )} 
\]
where $k \in \CC^{\times}$ is the proportionality constant in Lemma~\ref{lem:PluckerPermute} for $w$ and $\alpha$.
\end{lem}

Note that $w^{-1} \alpha$ is a negative root, so the exponent $-( \alpha^{\vee}, w\regweight )$ above is in fact nonnegative.

\begin{proof}
We first reduce to the case wherein $C = \{\alpha\}$ by proving that 
\begin{equation}
\label{eq:PluckerPassoff}
\mathrm{Pl}_{s_{\alpha} w} \left(  e_{\beta_1}(x_{\beta_1})\cdots e_{\beta_{|C|}}(x_{\beta_{|C|}}) v_{w\regweight} \right)
=
\mathrm{Pl}_{s_{\alpha} w} \left(  e_{\alpha}(x_{\alpha}) v_{w\regweight} \right).
\end{equation}
By~\eqref{eq:PluckerPrecision} and the fact that $s_{\alpha}w\regweight = w\regweight - (\alpha^{\vee}, w \regweight) \alpha$, all contributions to the left hand side of~\eqref{eq:PluckerPassoff} must correspond to solutions $\vec{a} \in \ZZ_{\ge 0}^{|C|}$ to the equation
\[
-\big( a_{|C|} + (\alpha^{\vee}, w \regweight)\big) \alpha
= 
a_{1}\beta_{1} + \cdots + a_{|C| - 1} \beta_{|C| - 1}.
\]
By assumption $\operatorname{Cone}(C \setminus \{\alpha\}) \cap \ZZ\alpha = \{\vec{0}\}$, so any possible solution has $a_{|C|} = -(\alpha^{\vee}, w \regweight)$ and $\sum_{i = 1}^{|C|-1} a_{i} \beta_{i} = \vec{0}$.  As each $\beta_{i} \in \Phi^{+}$, the only solution to the latter equation is $a_{1} = \cdots = a_{|C|-1} = 0$; this can be seen for example by taking the inner product of both sides with $\regweight$, as $(\beta_{i}, \regweight) > 0$.  
Therefore the $s_{\alpha}w$-Pl\"{u}cker coordinate comes solely from $x_{\alpha}^{-(\alpha^{\vee}, w \regweight)} E_{\alpha}^{-(\alpha^{\vee}, w \regweight)} v_{w \regweight}$. 
This is also the case for the right hand side of~\eqref{eq:PluckerPassoff}, proving that~\eqref{eq:PluckerPassoff} holds.  

Now assume $C = \{\alpha\}$.  If $x_{\alpha} = 0$, then $e_{\alpha}(x_{\alpha}) v_{w\regweight} =  v_{w\regweight}$ is entirely in the $w\regweight$ weight space and therefore has $s_{\alpha} w$-Pl\"{u}cker coordinate zero.  
Therefore we assume that $x_{\alpha} \neq 0$ and consider $s_{\alpha}(-x_{\alpha})e_{\alpha}(x_{\alpha}) v_{w\regweight}$, with an eye toward applying Lemma~\ref{lem:PluckerPermute}.  

We first claim that $s_{\alpha}(-x_{\alpha}) e_{\alpha}(x_{\alpha}) v_{w\regweight} = e_{\alpha}(-x_{\alpha}) v_{w\regweight}$.
Using the Chevalley relations we have
\[
s_{\alpha}(-x_{\alpha}) e_{\alpha}(x_{\alpha})
\stackrel{\eqref{equation:sdefn}}{=}
e_{\alpha}(-x_{\alpha})e_{-\alpha}(x^{-1}_{\alpha})e_{\alpha}(-x_{\alpha})e_{\alpha}(x_{\alpha})
\stackrel{\eqref{equation:ee_rel1}}{=}
e_{\alpha}(-x_{\alpha})e_{-\alpha}(x^{-1}_{\alpha}).
\]
To prove our claim we  show that $e_{-\alpha}(x^{-1}_{\alpha})$ acts trivially on $v_{w \regweight}$.  
Let $w = s_{\alpha_{1}} \cdots s_{\alpha_{\ell}}$ be a reduced word for $w$.  
As $\regweight$ is regular, there exists a $h \in T$ such that $\dot{w} = s_{\alpha_{1}}(1)s_{\alpha_{2}}(1) \cdots s_{\alpha_{\ell}}(1) h \in N_{G}(T)$ maps $v_{\regweight}$ to $v_{w\regweight}$.  
We therefore have
\[
e_{-\alpha}(x^{-1}_{\alpha}) v_{w\regweight} 
= e_{-\alpha}(x^{-1}_{\alpha}) \dot{w} v_{\regweight} 
\stackrel{\eqref{equation:se_rel}, \eqref{equation:he_rel}}{=} \dot{w} e_{-w^{-1}\alpha}(y) v_{\regweight}
\qquad\text{for some $y \in \CC^{\times}$}. 
\]
Now $v_{\regweight}$ is a highest weight vector, and $-w^{-1} \alpha \in \Phi^{+}$ (as $s_{\alpha} \in \inv{w}$), so $E_{-w^{-1}\alpha} v_{\regweight} = 0$, and therefore $e_{-w^{-1}\alpha}(y)v_{\regweight}  = v_{\regweight}$.  Therefore the above expression is equal to $\dot{w} v_{\lambda} = v_{w\regweight}$.

As a consequence, 
\[
\Pl_w\left(s_{\alpha}(-x_{\alpha})e_{\alpha}(x_{\alpha}) v_{w\regweight}\right) = \mathrm{Pl}_{w} \left(  e_{\alpha}(-x_{\alpha})  v_{w\regweight} \right) = 1.\]  
Lemma~\ref{lem:PluckerPermute} now gives the desired formula for $\mathrm{Pl}_{s_{\alpha} w} \left(  e_{\alpha}(x_{\alpha}) v_{w\regweight} \right)$.
\end{proof}

\begin{proof}[Proof of Proposition~\ref{cor:CoordinateSubspace}]

As $C$ is strongly closed, by greedily removing extreme rays from the cone spanned by $r(\inv{w})$ we can order  $r(\inv{w}) \setminus C$ as $\beta_{|C|+1},\ldots,\beta_{\ell(w)}$ in such a way that, for each $|C| < i \le \ell(w)$, $\beta_{i}$ is an extreme ray of the cone spanned by $r(\inv{w}) \setminus \{\beta_{i+1},\ldots,\beta_{\ell(w)}\}$.  

Now using Lemma~\ref{lem:PluckerPower}, the $s_{\beta_{\ell(w)}}w$-Pl\"{u}cker function  composed with $N_{w} \cong U_{w}wB$ is a power of the $\beta_{\ell(w)}$-coordinate function of $N_{w}$ and as $\regweight$ is regular, this power is not zero.  
Thus the isomorphism descends to $\Xc^{w}  \cap \{\Pl_{\beta_{\ell(w)}}^{\lambda} = 0\}$
and $N_{r(\inv{w}) \setminus \{\beta_{\ell(w)}\}}$.  
Repeating this descent argument for $N_{r(\inv{w}) \setminus D}$ with $D = \{\beta_{|C|+i}, \ldots, \beta_{\ell(w)}\}$ for each $1 \le i < \ell(w)-|C|$, we arrive at the desired isomorphism. 
 \end{proof}

\subsection{Coxeter Schubert cells and $\cfl_c$ via Pl\"ucker vanishing}
\label{sec:cBruhatcells}

We now apply results of the previous section to Coxeter Schubert cells and eventually $\cfl_{c}$.  The proof of Theorem~\ref{thm:everying_about_cfl} is given at the end of the section.  

\begin{prop}
\label{prop:PVInCells}
We have $\pluckervanishing{\NC(W,c)}\subset \bigsqcup_{u\in \NC(W,c)}\Xc^u_{\NC}$.
\end{prop}
\begin{proof}
    Let $x\in \pluckervanishing{\NC(W,c)}$ and take $u \in W$ so that $x\in \Xc^u$.  
    Then $\overline{T\cdot x} \subset \pluckervanishing{\NC(W,c)}$ by \Cref{thm:PluVan}. For a strictly antidominant cocharacter $\psi:\mathbb{C}^*\to T$ (as in the proof of \Cref{fact:rigidRich}) we have $u=\lim_{t\to 0}\psi(t)x\in \overline{T\cdot x}$ and so $u\in \NC(W,c)$. 
    The fact that $x\in \Xc^u_{\NC}$ now follows from \Cref{fact:invncClosed} and \Cref{cor:coordinatesubspace2}.
\end{proof}
\begin{cor}
\label{cor:wwcequalsuNC}
If $u\in \movetoid{w}{w\cdot c}$ is the Bruhat-maximum element then  $w^{-1}X^{w\cdot c}_w=X^{u}_{\NC}$.
\end{cor}
\begin{proof}
We have $u\in w^{-1}X^{w\cdot c}_w\subset X^u$, and so $w^{-1}X_w^{wc}\cap \Xc^u$ is a dense open subset of the irreducible $n$-dimensional variety $w^{-1}X_w^{wc}$. 
Furthermore by \Cref{thm:FirstCFlFacts}(1) and \Cref{cor:CoordinateSubspace} we have $$(w^{-1}X_w^{wc})\cap \Xc^u\subset  \pluckervanishing{\NC(W,c)}\cap \Xc^u\subset \Xc^u_{\NC},$$ and $\Xc^u_{\NC}\cong \mathbb{A}^n$. 
Therefore $w^{-1}X^{w\cdot c}_w\cap \Xc^u$ is dense in $\Xc^u_{\NC}$, and taking closures we obtain 
\[w^{-1}X^{w\cdot c}_w=X^u_{\NC}.\qedhere\]
\end{proof}

For the next result, we make a few observations about descending to a standard parabolic subgroup $W' \subseteq W$.  
First, if $c' \le_B c$ is a Coxeter element of $W'$, then $\NC(W', c') \subseteq \NC(W, c)$.  Second, if $w' \in W$ has $\ell(w'c') = \ell(w') + \ell(c')$ , then $\movetoid{w'}{w'\cdot c'} \subseteq \NC(W’, c’)$ by \Cref{thm:HasseUnion}.

\begin{cor}
\label{cor:w'w'c'equalsuNC}
    Let $W' \subseteq W$ be a standard parabolic subgroup and let $c'$ be a Coxeter element of $W'$ with $c'\le_B c$.
    If $w'\in W$ and $u\in \movetoid{w'}{w'\cdot c'}\subset \NC(W',c')$ is the Bruhat-maximum element, then $(w')^{-1}X_{w'}^{w'\cdot c'}=X^u_{\NC}$.
\end{cor}
\begin{proof}
Let $L\subseteq LB\subseteq G$ be the Levi and parabolic subgroups associated to $W'$, so that $B_{L}\coloneqq B \cap L$ is a Borel subgroup for $L$.  
Then $(w')^{-1}X^{w'c'}_{w'} \subseteq LB/B$, and under $L/B_{L} \cong LB/B$ it must map to $\cfl_{c'}$.  
We can therefore apply the argument in the proof of Corollary~\ref{cor:wwcequalsuNC} above to equate the $\ell(c')$-dimensional varieties \[\wt{(w')^{-1}X^{w'c'}_{w'}} = \wt{X^u_{\NC}},\] where $\wt{\cdot}$ indicates that these varieties are viewed in $L/B_{L}$.  
Transporting back along $L/B_{L}\cong LB/B\subseteq G/B$ identifies $\wt{X^u_{\NC}}$ with $X^u_{\NC}$.
\end{proof}
\begin{lem}
\label{lem:w'c'tow'c}
    If in \Cref{cor:w'w'c'equalsuNC} we have $w' \in W'$, then we have $\ell(w'c)=\ell(w')+\ell(c)$.
\end{lem}
\begin{proof}
    If $y,z$ are in the subgroup of $W$ generated by simple reflections in $\Simple \setminus \{s_j\}$  and $\ell(y)+\ell(z)=\ell(yz)$ then we claim that $\ell(ys_jz)=\ell(yz)+1$. Indeed, $ys_jy^{-1}\not\in \inv{yz}$ as $\inv{yz}$ is in the subgroup of $W$ generated by $\Simple\setminus \{s_j\}$, so $\ell(yz)+1\ge \ell(ys_jz)>\ell(yz)$.
    
    Now, start with a reduced word for $w'c'$ obtained by concatenating reduced words for $w'$ and $c'$ and apply the claim iteratively, inserting the missing letters from $c$ that are not in $c'$.
\end{proof}

\begin{prop}
\label{prop:magiccatalan}
    Let $u\in \NC(W,c)$ and let $c'\le_B c$ be the sub-Coxeter corresponding to a standard parabolic subgroup $W' \subset W$ such that $u\in \NC(W',c')^+$. Then there exists $w'\in W'$ with $\ell(w'c')=\ell(w')+\ell(c')$ with $u\in \movetoid{w'}{w'c'}\subset \NC(W',c')$.
\end{prop}
\begin{proof}
    It suffices to prove the statement for $u\in \NC(W,c)^+$, as we can apply that statement directly to $W'$ and $c'$. Note that by \Cref{prop:uniquecdec} and \Cref{cor:numbdecreasing} we know that there are at least $\cat{W}^+$-many translated $\movetoid{w}{wc}$, so there are at least $\cat{W}^+$ many distinct $w^{-1}X^{w\cdot c}_w$ comprising the irreducible components of $\cfl_c$. On the other hand, by \Cref{thm:FirstCFlFacts}(1) and \Cref{prop:PVInCells} we know that 
$$\cfl_c
\subseteq\pluckervanishing{\NC(W,c)}\subseteq \bigsqcup_{u\in \NC(W,c)}\Xc^u_{\NC}\subseteq \bigcup_{u\in \NC(W,c)} X^u_{\NC},$$ and $\dim \Xc^u_{\NC}=|\invnc{u}|\le n$ with equality if and only if $u\in \NC(W,c)^+$, so there are $\cat{W}^+$-many top-dimensional irreducible components of $\bigcup_{u\in \NC(W,c)} X^u_{\NC}$. We conclude that each $X^u_{\NC}$ must equal at least one $w^{-1}X^{w\cdot c}_w$.
\end{proof}

\begin{proof}[Proof of Theorem~\ref{thm:everying_about_cfl}]
By \Cref{thm:FirstCFlFacts}(1) and \Cref{prop:PVInCells} we know that 
\[
\cfl_c
\subseteq 
\pluckervanishing{\NC(W,c)}
\subseteq 
\bigsqcup_{u\in \NC(W,c)}\Xc^u_{\NC},
\]
so to show these containments are equalities it suffices to show that
$\bigsqcup_{u\in \NC(W,c)}\Xc^u_{\NC}\subseteq \cfl_c$.

For each $u\in \NC(W,c)^+$, there exists $w\in W$ with $\ell(w)+\ell(c)=\ell(wc)$ such that $u$ is the Bruhat-maximum element of $\movetoid{w}{wc}$ (which exists by \Cref{prop:magiccatalan}). Then by \Cref{cor:wwcequalsuNC} we have $$X^u_{\NC}=w^{-1}X_w^{wc}\subset \cfl_c.$$ 
For an arbitrary $u\in \NC(W,c)$ if $c'$ is the associated sub-Coxeter element in the minimal sub-Weyl group  $(W',\Simple')\subset (W,\Simple)$ containing $u$, then $u\in \NC(W',c')^+$ and we again take by \Cref{prop:magiccatalan} a $w'\in W'$ such that $u$ is the Bruhat-maximum element of $\movetoid{w'}{w'c'}\subset W'$. By \Cref{lem:w'c'tow'c} we have $\ell(w')+\ell(c)=\ell(w'c)$ and so by \Cref{cor:w'w'c'equalsuNC} we have 
\[
X^{u}_{\NC}=(w')^{-1}X_{w'}^{w'c'}\subset  (w')^{-1}X_{w'}^{w'c}\subset \cfl_c.
\]

Finally we establish the affine paving.
Recall that $u_1$ through $u_{\cat{W}}$ is an ordering of $\NC(W,c)$ compatible with Bruhat order on $W$.
Because $X_{k+1}\setminus X_k=\Xc^{u_{k+1}}_{\NC}$, which is isomorphic to an affine space, it suffices to show that $X_k$ is closed. 
But this follows as \begin{equation*}X^{u_k}\cap \cfl_{c}=\left(\bigsqcup_{u\le u_k} BuB\right)\cap \bigsqcup_{i=1}^{\cat{W}}\Xc^{u_i}_{\NC}=\bigsqcup_{i=1}^k \Xc^{u_i}_{\NC}=X_k.\qedhere
\end{equation*}
\end{proof}
\begin{rem}
\label{rem:ComplexContractible}
    The affine paving implies that $\Complex(\cfl_{c})$ is contractible. 
    Indeed, let $P_k=\mu(X^{u_k}_{\NC})$. 
    Let $C^{\le k}\subset \Complex(\cfl_{c})$ be the part of the complex associated to $X_k=\bigcup_{i=1}^k \Xc^{u_i}_{\NC}=\bigcup_{i=1}^k X^{u_i}_{\NC}$, i.e. the part of the complex associated to just the polytopes $P_1,\ldots,P_k$. Then $C^{\le k+1}$ is obtained from $C^{\le k}$ by gluing $P_{k+1}$ to $C^{\le i}$ along the subset $\mu(X^{u_{k+1}}_{\NC}\setminus \Xc^{u_{k+1}}_{\NC})=(\partial P_k)'\subset \partial P_k$, the part of $\partial P_{k+1}$ not touching the simple vertex $\mu(u_{k+1})\in P_{k+1}$. Hence we can successively collapse $\Complex(\cfl_{c})$ by collapsing in reverse order each $P_k$ onto $(\partial P_k)'$.
\end{rem}

\section{Cohomology}
\label{section:cohomology}

In this section we describe the torus-equivariant cohomology of $\cfl_{c}$ using the suite of tools known as \emph{GKM theory}. This culminates in the proof of \Cref{maintheorem:indepofc} in \Cref{cor:coho_iso}. 
For technical reasons, we will need to use the \emph{adjoint torus} $T_{ad} = T / Z(G)$ on $G/B$ (\Cref{subsec:adjointgroup}).  
As a consequence, the rings computed in this section depend only on the root system $\Phi$ of $G$.

\subsection{GKM Theory}

In this section we recall specific facts from GKM theory; beginning with Goresky--Kottwitz--MacPherson~\cite{GKM98}, versions of these statements have been proven by several authors under varying conditions.  
In a few cases we have not found statements that match our exact hypothesis, and therefore include proofs for the sake of completeness.

For us, a \emph{GKM variety} is a (possibly reducible) algebraic variety $X$ with the action of a torus $T$ such that
\begin{enumerate}
    \item $X$ has finitely many $T$-fixed points, 
    \item $X$ has finitely many $T$-invariant curves, and 
    \item $X$ embeds $T$-equivariantly into $\mathbb{P}(V)$, the projectivization of a $T$-representation $V$. 
\end{enumerate}

All of our GKM varieties will come from the generalized flag variety $G/B$, equipped with the action of $T_{ad}$.  
As $Z(G)$ acts trivially on $G/B$, the $T_{ad}$-orbits are exactly the same as the $T$-orbits, so that the fixed points are $\{wB \suchthat w \in W\}$ and the $T_{ad}$-invariant curves are still $\mathbb{P}_{u,v}$ for $u=\tau v$ and $\tau\in \Reflections$.
For embedding, we take the generalized Pl\"{u}cker embedding $\Pl: G/B \to V^{\regweight}$ coming from our choice of regular dominant $\regweight\in \Q{T_{ad}} = \ZZ \Phi$.  

Any $T$-invariant subvariety $Y$ of a GKM variety is again a GKM variety, so every $T$-invariant subvariety of $G/B$ is GKM. In particular any $T_{ad}$-invariant (or simply $T$-invariant) subvariety $Y \subseteq G/B$ is also GKM.   

In any GKM variety $X$, each $T$-fixed point corresponds to a line in some weight space $V_{\mu}$,
$\mu \in \Q{T}$, and each $T$-invariant curve $Y$ contains two fixed points coming from distinct weights, $\mu, \nu \in \Q{T}$.  The tangent spaces to $Y$ at these points afford the characters $(\nu-\mu)$ and $-(\nu-\mu)$.  

\begin{defn}
The \emph{GKM graph} of a GKM variety $X$ is the edge-labeled graph $\GKM(X)$ with vertex set $X^T$, edges determined by the $T$-invariant curves, and edge labels ``$\pm \chi$'' where $\pm \chi$ are the characters of the tangent spaces to the two fixed points.
\end{defn}

For $G/B$, identifying $wB$ with $w \in W$ realizes the GKM graph as the Cayley graph $\Cay(W,\Reflections)$, with edges $\{w, \tau w\}$ labelled by $\pm r(\tau)$.  
If $Y \subseteq G/B$ is $T$-invariant, then its GKM graph will be a subgraph of $\Cay(W,\Reflections)$ with the same edge labels; this will be particularly important if $Y$ is toric or if $Y$ is a Pl\"{u}cker vanishing variety.  
In the former case, it is well-known that the GKM graph is exactly the $1$-skeleton of the moment polytope $P_{Y}$, and each edge label is the direction of the corresponding edge in $P_{X}$.  The latter case is described in the following result.  

\begin{prop}
For $\mathcal{A} \subseteq W$, the GKM graph of the Pl\"{u}cker vanishing variety $\pluckervanishing{\mathcal{A}}$ is the induced subgraph of $\Cay(W,\Reflections)$ on $\mathcal{A}$.
\end{prop}
\begin{proof}
By Theorem~\ref{thm:PluVan}, $\pluckervanishing{\mathcal{A}}$ consists of all torus orbit-closures whose $T$-fixed point set lies in the set $\mathcal{A}B = \{wB \suchthat w \in \mathcal{A}\}$.  
The zero-dimensional orbits in this set are $\mathcal{A}B$, and the one-dimensional orbits are the edges $\{w, \tau w\} \subseteq \mathcal{A}$.  
\end{proof}

Under certain assumptions, the cohomology ring of a GKM variety can be computed from its GKM graph.  The $T$-equivariant cohomology ring of a point is given by the polynomial ring
\[
H^{\bullet}_{T}(pt) \coloneqq \operatorname{Sym}^\bullet \big( \Q{T}\big) 
\]
As is standard in algebraic combinatorics, we denote by $t_{\lambda}\in H^{\bullet}_{T}(pt)$ for the element corresponding to $-\lambda \in \Q{T}$, the negative character, and grade this ring so that each nonzero $t_{\lambda}$ has degree $2$.
Let \[
H^{+}_{T}(pt)\coloneqq \text{ the ideal generated by elements of degree $\ge 2$.}  
\]
The adjoint torus $T_{ad}$ has character lattice equal to the root lattice $\ZZ\Phi$, so that in $T_{ad}$-equivariant cohomology we have
\[
H^{\bullet}_{T_{ad}}(pt) = \ZZ[t_{\alpha} \suchthat \alpha \in \Delta].
\]

\begin{defn}
For a GKM graph $\graph$, we define the \emph{graph cohomology ring} to be the $H^{\bullet}_{T}(pt)$-algebra
\[
H^\bullet_T(\graph)\coloneqq \{(f_v)_{v\in V(\graph)}\suchthat t_{\chi(vv')} \text{ divides } f_v-f_{v'}\text{ for all edges $vv'$ in $\graph$}\}\subset (H^\bullet_T(pt))^{\oplus V(\graph)}.
\]
We identify $H^{\bullet}_{T}(pt)$ with its diagonal embedding in $(H^\bullet_T)^{\oplus V(\graph)}$.
\end{defn}

Under certain conditions the graph cohomology ring of the GKM graph computes the actual cohomology ring of a GKM space $X$.  
Say that $X$ has a \emph{good affine paving} if it has a filtration by $T$-invariant subvarieties 
\[
\emptyset=X_0\subset X_1\subset X_2\subset \cdots \subset X_\ell=X
\] 
such that the following holds for each $i \ge 1$:
\begin{enumerate}
\item $X_{i} \setminus X_{i-1}$ is $T$-equivariantly isomorphic to a linear $T$-representation $V_{i}$, and therefore contains a unique $T$-fixed point $w_{i}$; 

\item the representation $V_i$ decomposes into a direct sum of one-dimensional $T$-representations
\[
V_i=\bigoplus_{j\in A_i} V_{i,j} \qquad \text{for }A_i\subseteq \{1,\ldots,i-1\},
\]
such that $\overline{V_{i,j}}=V_{i, j} \sqcup \{w_j\}$; and 

\item The character $f_{i, j}$ of $V_{i, j}$ is nonzero, and is moreover reduced in the sense that $\frac{1}{r}f_{i,j}\not\in \Q{T}$ for any integral $r\ne \pm 1$.

\end{enumerate}

Returning to our example of $X = G/B$, we obtain a good affine paving by fixing a linear extension  of the Bruhat order on $W$, $w_{1}, w_{2}, \ldots, w_{N}$, and taking $X_{i} = X^{w_{1}} \cup X^{w_{2}} \cup \cdots \cup X^{w_{i}}$.  
Then the isomorphism $\Xc^{w_{i}} \cong N_{r(\inv{w_i})}$ in Section~\ref{sec:closedsets} satisfies condition (1) and the $V_{i, j}$ are the root subspaces $\mathfrak{g}_{\alpha}$, which correspond to $x_{r(\tau)}(\CC) wB \subseteq \Xc^{w_{i}}$, whose closure is $\mathbb{P}_{w_i,w_i\tau}$.  
While the character of $\mathfrak{g}_{\alpha}$, which is the root $\alpha \in \Phi$, may not be reduced in $\Q{T}$, it is always reduced in $\Q{T_{ad}}=\mathbb{Z}\Phi$. 
Satisfying condition (3) is the primary reason for us to pass to the adjoint torus.
\begin{fact}
    If $T=T_{ad}$  and all characters $f_{i,j}$ above are roots in $\Phi$, then condition (3) is automatically satisfied. 
\end{fact}

\begin{thm}[{\cite[Theorem 11.3]{BGNST2}}]
\label{thm:GKM}
    For a GKM variety $X$ with a good affine paving, we have
    \begin{enumerate}
        \item $H_\bullet(X)$ has a homology basis $\{[\overline{X_i\setminus X_{i-1}}]\}_{i\in \{1,\ldots,\ell\}}\subset H_\bullet(X)$, 
        \item $H^\bullet_T(X)\cong H^\bullet_T(\graph)$, and 
        \item $H^\bullet_T(X)$ is a free $H^\bullet_T(pt)$-module and $H^\bullet(X)\cong H^\bullet_T(\graph)/(H^+_T(pt))$.
    \end{enumerate}
\end{thm}
\begin{proof}
    All parts were shown in \cite[Theorem 11.3]{BGNST2}, except that in (3) the freeness was a hypothesis rather than a conclusion. However, the freeness follows from (1) as this implies that the Leray--Hirsch spectral sequence degenerates at the $E_2$--page.
\end{proof}
For $\cfl_c$ we will actually be able to combinatorially establish the existence of a free $H^\bullet_{T_{ad}}(pt)$-basis with additional properties, see \Cref{thm:cfl_duality}.

\subsection{The cohomology of $\cfl_c$}

We now apply the machinery developed in the previous section to describe the equivariant cohomology of $\cfl_{c}$, and in particular prove \Cref{maintheorem:indepofc}.  
We continue to use the action of the adjoint torus $T_{ad}$.  
For this reason, the results in this section are independent of the choice of $G$ up to the underlying root system $\Phi$.

\begin{thm}
\label{thm:cfl_coho}
Let $\Phi$ be a root system, $W$ its Weyl group, and $c \in W$ a Coxeter element.
\begin{enumerate}[label = (\roman*)]
\item For any linear extension $u_{1}, u_{2}, \ldots, u_{\cat{W}}$ of the Bruhat order on $\NC(W, c)$, the filtration of $\cfl_{c}$ given by 
\[
X_{i} = \Xc^{u_{1}}_{\NC} \cup \cdots \cup \Xc^{u_{i}}_{\NC}
\]
is a good affine paving with respect to the action of $T_{ad}$.  

\item We have 
\[
H^\bullet_{T_{ad}}(\cfl_c)=H^\bullet_{T_{ad}}(\Cay(W,\Reflections)|_{\NC(W,c)})
\]
and
\[
H^\bullet(\cfl_c)=H^\bullet_{T_{ad}}(\cfl_c)/H^+_{T_{ad}}(pt).
\]
\end{enumerate}
\end{thm}
\begin{proof}
Point (i) follows from \Cref{thm:everying_about_cfl} and the fact that the associated charts for the affine paving are coordinate subspaces of the good affine paving for $G/B$ by Schubert cells. By Theorem~\ref{thm:GKM}, point (i) implies point (ii).
\end{proof}
\begin{cor}[{\Cref{maintheorem:indepofc}}]
\label{cor:coho_iso}
For Coxeter elements $c, c' \in W$, and $u\in W$ such that $c=wc'w^{-1}$, there is a ring isomorphism $\Psi_{c,w}:H^{\bullet}(\cfl_{c}) \cong H^{\bullet}(\cfl_{c'})$. Furthermore these isomorphisms satisfy $\Psi_{c,\idem}=\idem$, and if $w'c'(w')^{-1}=c''\in W$ is another Coxeter element, then $\Psi_{c',w'}\Psi_{c,w}=\Psi_{c,w'w}$.
\end{cor}
\begin{proof}
We define a ring isomorphism on $T_{ad}$-equivariant cohomology rings
\[
\wt{\Psi_{c,w}}:H^{\bullet}_{T_{ad}}(\cfl_{c}) \to H^{\bullet}_{T_{ad}}(\cfl_{c'})
\qquad\text{by}\qquad
(f_{u})_{u \in \NC(W, c)}  \mapsto \big(w\cdot f_{w^{-1}vw} \big)_{v \in \NC(W, c')},
\]
which maps $H^{+}_{T_{ad}}(\cfl_{c})$ to $H^{+}_{T_{ad}}(\cfl_{c'})$. This therefore descends to ring isomorphisms $\Psi_{c,w}$, and the compatibilities are trivially checked on the lifts $\wt{\Psi_{c,w}}$.
\end{proof}

\begin{rem}
The map used to prove Corollary~\ref{cor:coho_iso} is an isomorphism of rings between the equivariant cohomology rings as well, but not an isomorphism of $H_{T_{ad}}^{\bullet}(pt)$-modules.
\end{rem}

\section{Duality bases}
\label{section:DualityBases}

Let $X$ be a GKM variety with a good affine paving $X_{1},  \ldots, X_{N}$. Because $H^\bullet_T(X)$ is a free $H^\bullet_T(pt)$-module, we can look to find distinguished $H^\bullet_T(pt)$-bases.
\begin{defn}
    A \emph{flowup basis} for $H^\bullet_T(X)$ is a subset $\{f^{(1)},f^{(2)},\ldots,f^{(N)}\}\subset H^\bullet_T(X)$ such that
    \begin{equation}
\label{eq:flowup}
\text{$f^{(i)}_{j} = 0$ for all $j < i$}
\qquad\text{and}\qquad
f^{(i)}_{i} = \pm \prod_{\text{edges $w_{i}w_{j}$\text{ and }j < i}} t_{\chi(w_{i}w_{j})}.
\end{equation} 
\end{defn}
A straightforward upper-triangularity argument shows that a flowup basis must be a free $H^\bullet_T(pt)$-basis, justifying the name. In general a flowup basis, if it exists, is not unique. One condition we might hope to impose is that $f^{(i)}_j=0$ whenever $i\ne j$ and $\dim X_i=\dim X_j$ (see for example \cite[Proposition 4.3]{MR2166181}), but because of the strange way that our cells fit together, $H^\bullet_{T_{ad}}(\cfl_c)$ does not have a flowup basis satisfying this extra condition.

We will instead create a basis which is dual with respect to certain degree maps. We work under the assumption in this section that there is a \emph{$T$-equivariant degree map} $\int^T_{\overline{X_i\setminus X_{i-1}}}:H^\bullet_T(\overline{X_i\setminus X_{i-1}})\to H^\bullet_T(pt)$ for each $i$ (see  \Cref{sec:ABloc}).
We say that a collection of elements $f^{(1)}, f^{(2)}, \ldots, f^{(N)}$ in $H^\bullet_T(X)$ is a \emph{duality basis} for this paving if 
\[
\int_{\overline{X_i\setminus X_{i-1}}}^T f^{(j)}=\delta_{i,j}\in H^\bullet_T(pt).
\]

The papers~\cite{BGNST1, BGNST2} define a duality basis with respect to the affine paving in Theorem~\ref{thm:cfl_coho} in the special case that $G = \GL_{n+1}$ and $c = (n+1\,\cdots\,2\,1)$, which we called \emph{double forest polynomials}.  
We now show that analogues of forest polynomials exist for any Coxeter flag variety.  

\begin{thm}
\label{thm:cfl_duality}
Let $W$ be a Weyl group, and $c \in W$ a Coxeter element.  Then there exist unique elements $$\{ \wt{\schub{u}^{\NC}} \suchthat u \in \NC(W, c)\}\subset H^\bullet_{T_{ad}}(\cfl_c)$$
we call \emph{Coxeter Schubert classes} which give a (unique) duality basis with respect to any affine paving of the form described in Theorem~\ref{thm:cfl_coho}. These elements satisfy, but are not uniquely characterized by, the conditions that 
\[
(\wt{\schub{u}^{\NC}})_u=\prod_{\tau\in \invnc{u}} r(\tau)
\]
and for $v\in \NC(W,c)$
\[
\text{$(\wt{\schub{u}^{\NC}})_v=0$ whenever $u\not\le v$ (in the Bruhat order)}.
\]
\end{thm}

The remainder of the section is concerned with the proof of Theorem~\ref{thm:cfl_duality}: Section~\ref{sec:ABloc} establishes general properties of duality bases needed for the proof and Section~\ref{ssec:ForestforCFL} applies these properties to the combinatorial structure of $\cfl_{c}$.  

\subsection{Duality bases via Atiyah--Bott localization}
\label{sec:ABloc}

In this section, let $X$ be a GKM variety with a good affine paving $X_{1}, \ldots, X_{N}$.
If  $\overline{X_i\setminus X_{i-1}}$ is smooth, then $\int_{\overline{X_i\setminus X_{i-1}}}^Tf$ is the $T$-equivariant pushforward to a point. In general, the pushforward is not a priori well-defined. We say that the paving \emph{admits $T$-equivariant degree maps} if for each $i$ we are able to find a fixed $T$-equivariant resolution
$$\pi_i:\wt{X_i\setminus X_{i-1}}\to \overline{X_i\setminus X_{i-1}},$$
by a smooth irreducible GKM variety $\wt{X_i\setminus X_{i-1}}$, which is an isomorphism on the open set $X_i\setminus X_{i-1}$. In this case we define the $T$-equivariant degree
$$\int^T_{\overline{X_i\setminus X_{i-1}}}f\coloneqq \int^T_{\wt{X_i\setminus X_{i-1}}}\pi_{i}^{*}f\in H^\bullet_T(pt).$$

For each class $f \in H^{\bullet}_{T}(X)$ and each $T$-fixed-point $p$, let $f_{p}$ be the pullback of $f$ to $H^{\bullet}_{T}(p)$.
Then $T$-equivariant Atiyah--Bott localization (see e.g.~\cite{AtBo84, BeVe82}) computes the $T$-equivariant degree of $f$ on $Y=\overline{X_i\setminus X_{i-1}}$ by constructing a family of elements $C_{w_{j}}^Y\in \mathrm{Frac}(H^\bullet_T(pt))$, $1 \le j \le i$ for which:
\[
\int^{T}_{Y} f = \sum_{j = 1}^{i} C_{w_{j}}^{Y} f_{w} \in H^{\bullet}_{T}(pt)
\qquad\text{and}\qquad
C_{w_i}^{Y} = \pm \prod_{\text{edges $w_{i}w_{j}$} \text{ and } j < i} t_{\chi(w_{i}w_{j})}^{-1}.
\]
We will primarily be interested in the case where $X$ is a toric complex with a good affine paving. Recalling the well-known facts that
\begin{enumerate}
    \item every projective toric variety admits a toric resolution by a smooth projective toric variety
    \item the associated degree maps are independent of the choice of resolution,
\end{enumerate}
we see that such an $X$ in fact admits $T$-equivariant degree maps in an essentially unique way.
\begin{rem}
    For $\cfl_c$, all of whose $\overline{X_i\setminus X_{i-1}}$ are normal toric varieties associated to Coxeter matroid polytopes, one can always resolve using the toric variety $X_{\operatorname{Perm}_W}$ associated to the $W$-permutahedron (\Cref{subsec:CoxMat}).
\end{rem}
\begin{thm}
\label{thm:dualityflowup}
Suppose that $X$ is a GKM variety with a good affine paving admitting $T$-equivariant degree maps.  
If a duality basis $\{ f^{(1)}, f^{(2)}, \ldots, f^{(N)}\}$ exists for $H^{\bullet}_{T}(X)$, then it is a flowup basis.  
Moreover, if $H^{\bullet}_{T}(X)$ has a flowup basis, then it has a unique duality basis.
\end{thm}

\begin{proof}
Let $F$ be the $N \times N$ matrix with $i, j$ entry $f^{(j)}_{w_{i}}$ and $C$ be the matrix with $j, k$ entry is equal to $C_{w_{j}}^{\overline{X_k\setminus X_{k-1}}}\in \mathrm{Frac}(H^\bullet_T(pt))$ if $w_{j} \in \overline{X_k\setminus X_{k-1}}$ and $0$ otherwise.    
Then
\[
\int^{T}_{\overline{X_{i} \setminus X_{i-1}}}\; f^{(j)} = (CF)_{i, j},
\]
so we have a duality basis if and only if $F = C^{-1}$.  
As $C$ is lower triangular with diagonal entries 
\[
C_{j, j} = \pm \prod_{\text{edges $w_{i}w_{j}$\text{ with }i < j}} t_{\chi(w_{i}w_{j})}^{-1},
\]
we see that $F = C^{-1}$ implies that~\eqref{eq:flowup} holds.  

On the other hand, if~\eqref{eq:flowup} holds then $F$ is lower triangular with diagonal entries $F_{i, i} = 1 / C_{i, i}$.  
In this case $CF$ is unipotent lower triangular, so $(CF)^{-1} = \sum_{k = 0}^{N} (I-CF)^{k}$ is also unipotent lower triangular with entries in $H^{\bullet}_{T}(pt)$.  
Thus the columns of $C^{-1} = F (CF)^{-1}$ determine a duality basis of $H^{\bullet}_{T}(X)$ whose elements are $H^{\bullet}_{T}(pt)$-linear combinations of the $f^{(i)}$.
\end{proof}

For $G/B$ and any linear extension of the Bruhat order, the paving admits $T$-equivariant degree maps via a choice of Bott--Samelson $\mathcal{B}^{\bm{w}}$ resolution for each Schubert cell $X^w$ associated to a choice of reduced word $\bm{w}$ for $w$. 
The unique duality basis corresponds to the Schubert classes $\{\wt{\mathfrak{S}_{w}} \suchthat w \in W\}$ in $H^{\bullet}_{T_{ad}}(\Cay(W,\Reflections)$.  
These are recursively defined by 
\[
\wt{\mathfrak{S}_{\wnaught}} = \big( \delta_{w, \wnaught} \prod_{\alpha \in \Phi^{+}} t_{\alpha} \big)_{w \in W},
\]
where $\delta_{w, \wnaught}$ is the Kronecker delta, and for $w \in W$ and $s_i \in \des{w}$ with associated simple root $\alpha_i$,
\[
\wt{\mathfrak{S}}_{ws_{i}} = \partial_{i} \wt{\mathfrak{S}_{w}}
\qquad\text{where}\qquad
\partial_{i}f=(\frac{f_v-f_{vs_i}}{t_{\alpha_i}})_{v\in W}.
\]
We now describe how to construct duality bases under stronger hypotheses on $X$. 

\begin{thm}
\label{thm:GKMflowupexist}
Let $X$ be a toric complex with good affine paving.  If the characters of each affine chart $X_{i} \setminus X_{i-1}$ can be extended to a basis of $\Q{T}$, then $X$ has a duality basis.
\end{thm}
To prove this theorem, we need the following lemma, which was proved in \cite[Proposition 5.3]{KC24} in a slightly weaker setting over the rational numbers.
\begin{lem}
\label{lem:KCChinese}Under the assumptions of \Cref{thm:GKMflowupexist}, for any $h\in H^\bullet_T(X_{i-1})$ there is a class $h'\in H^\bullet_T(X_i)$ with $h_{w_j}=h'_{w_j}\in H^\bullet_T(pt)$ for all $j\le i-1$.
\end{lem}
\begin{proof}
    Suppose that the vertex $w_{i}$ in $\GKM(X_{i})$ is incident to the edges $w_{j_1}w_i, w_{j_{2}}w_{i},\ldots,w_{j_m}w_i$, with respective edge labels $\lambda_{1}, \lambda_{2}, \ldots, \lambda_{m}$.  
Fix an extension $\lambda_{m+1}, \ldots, \lambda_{n}$ of this set to a basis of $\Q{T}$ and define $t_{j} = t_{\lambda_{j}}$, so that $H^\bullet_T(pt) \cong \ZZ[t_{1}, \ldots, t_{n}]$.

Now we must construct an $h'_{w_{i}} \in  H^\bullet_T(pt)$ which is compatible with each $h_{w_{j_r}}$ under the divisibility condition imposed by each edge $w_{j_r}w_{i}$.  
This amounts to solving the system 
\begin{align*}
h'_{w_{i}} \equiv h_{w_{j_{1}}} \mod{(t_1)}, \qquad
h'_{w_{i}} \equiv h_{w_{j_{2}}} \mod{(t_2)}, \qquad
\ldots\qquad,\qquad
h'_{w_{i}} &\equiv h_{w_{j_{m}}} \mod{(t_m)}.
\end{align*}
It turns out that we have a solution if and only if 
\[
h_{w_{j_k}} \equiv h_{w_{j_{\ell}}} \mod{ (t_{k},t_{\ell})} \qquad \text{for $1\le k < \ell \le m$}.
\]
Indeed, this condition is equivalent to requiring that each monomial $t_{1}^{a_{1}} t_{2}^{a_{2}} \cdots t_{n}^{a_{n}}$ has the same coefficient in all $h_{w_{r}}$ for which $a_{r} = 0$.  
We denote this coefficient by $c_{a_1,\ldots,a_n}$ and for convenience set $c_{a_1,\ldots,a_n} = 0$ when all $a_{i} > 0$. 
We then set
\[
h'_{w_i}=\sum c_{a_1,\ldots,a_n} t_{1}^{a_{1}} t_{2}^{a_{2}} \cdots t_{n}^{a_{n}},
\]
which clearly solves the given system.

We conclude by demonstrating the congruence $h_{w_{j_k}} \equiv h_{w_{j_{\ell}}}$ mod $(t_{k},t_{\ell})$ by considering the GKM conditions arising from the edges of the moment polytope for $\overline{X_{i} \setminus X_{i-1}}$.  
Let $F$ be the $2$-face containing the vertices corresponding to the fixed points $w_{i}$, $w_{j_{k}}$, and $w_{j_{\ell}}$, and let $w_{j_{k}} = v_{1}, v_{2}, \ldots, v_{r} = w_{j_{\ell}}, w_{i}$ be the path in $\GKM(X_{i})$ corresponding to the boundary of $F$.  
As $h$ satisfies the divisibility condition for each edge in this path except the last one, there exist $a_{1}, a_{2}, \ldots, a_{r-1} \in H^\bullet_T(pt)$  for which
\[
h_{w_{j_{k}}} 
= h_{v_{2}} + a_{1} t_{\chi(v_{1}v_{2})}
= h_{v_{3}} + a_{1} t_{\chi(v_{1}v_{2})} + a_{2} t_{\chi(v_{2}v_{3})}
= \cdots
= h_{w_{j_{\ell}}} + \sum_{q = 1}^{r-1} a_{q} t_{\chi(v_{q} v_{q+1})}.
\]
Finally, each $t_{\chi(v_{q} v_{q+1})}$ belongs to $(t_{k},t_{\ell})$: each character $\chi(v_{q} v_{q+1})$ is a real multiple of the corresponding edge $v_{q+1} - v_{q}$, which lies in the real span of $\lambda_{j_k}$ and $\lambda_{j_\ell}$.  
As $\lambda_{j_k}, \lambda_{j_\ell}$ are part of a basis, this means each $\chi(v_{q} v_{q+1})$ is in the integral span of $\{\lambda_{j_k}, \lambda_{j_{\ell}}\}$.  
\end{proof}
 
\begin{proof}[Proof of \Cref{thm:GKMflowupexist}]
We construct a flowup basis, and then conclude by \Cref{thm:dualityflowup}.

If $X = X_{1} = pt$, then the claim is immediate.  
Otherwise, we assume the existence of a flowup basis $\{h^{(1)}, h^{(2)}, \ldots, h^{(i-1)}\}$ for $H^\bullet_T(X_{i-1})$ and construct a flowup basis $\{g^{(1)},\ldots,g^{(i)}\}\subset H^\bullet_T(X_{i})$. For this we define $g^{(i)}$ by setting
\[
g^{(i)}_{w_{j}} = \delta_{i, j} \prod_{\text{edges $w_{j}w_{i}$\text{ and }j < i}} t_{\chi(w_{j}w_{i})},
\]
and to construct $g^{(j)}$ for $j\le i-1$ we apply \Cref{lem:KCChinese} to $h^{(j)}$.
\end{proof}

\subsection{A duality basis for $\cfl_c$}
\label{ssec:ForestforCFL}

We now show \Cref{thm:cfl_duality}. 

\begin{lem}
\label{lem:clusterextend}
For any Coxeter element $c \in W$, every $c$-cluster is a basis for the root lattice $\ZZ \Phi$.  As a consequence, every positive $c$-cluster can be extended to a basis of $\ZZ \Phi$.
\end{lem}

In order to prove the lemma, we recall the definition of a \emph{bipartite Coxeter element}.
Given a bipartition $\Simple_+\sqcup \Simple_-$ of the Coxeter diagram of $\Simple$, we define the corresponding bipartite Coxeter element $c$ by $\displaystyle\prod_{s\in \Simple_+}s\prod_{s\in \Simple_-}s$.

\begin{proof}
If $c$ is a bipartite Coxeter element this is~\cite[Theorem 1.8]{FZ03b}.  
Otherwise, there exists a $w \in W$ such that $c' = wcw^{-1}$ is a bipartite Coxeter element \cite[Proposition 3.16]{Hum90}, and $wuw^{-1} \in \NC(W, c')$ if and only if $u \in \NC(W, c)$.  
Thus for each $u \in \NC(W, c)$, $r(\invnc{wuw^{-1}}) = w \cdot r(\invnc{u})$, so that $w$ maps $\mathrm{Cl}_{c}$ to $\mathrm{Cl}_{c'}$. 
As the reflection action of $w$ is defined over $\ZZ$ and the image of each $c$-cluster under $w$ is a basis for $\ZZ \Phi$, the proof is complete.
\end{proof}

\begin{proof}[Proof of \Cref{thm:cfl_duality}]
The characters of the affine spaces $\nctocluster(u)$ in the affine pavings are $W$-translates of subsets of $c$-clusters, so \Cref{lem:clusterextend} verifies the necessary basis extension property to apply \Cref{thm:GKMflowupexist}. The duality basis remains the same as we vary the linear extension of $\NC(W,c)$ determining the order the pieces of the affine paving appear in. This implies the duality basis is a flowup basis for every linear extension of $\NC(W,c)$, verifying the remaining vanishing conditions.
\end{proof} 

\section{Cohomology in type $A$ and permuted quasisymmetry}
\label{sec:typeAmain}

Fix $G=\GL_{n+1}$ so that $W=S_{n+1}$, and $G/B=\fl{n+1}$ is the complete flag variety parametrizing complete flags of subspaces
$$\fl{n+1}=\{0\subsetneq \mathcal{F}_1\subsetneq \mathcal{F}_2\subsetneq \cdots \subsetneq \mathcal{F}_n\subsetneq \mathbb{C}^{n+1}\suchthat \dim\mathcal{F}_i=i\}.$$ 
We summarize the work of~\cite{BGNST1,BGNST2} on the \emph{quasisymmetric flag variety} $\hhmp_{n+1}=\cfl_{s_n\cdots s_1}\subset \fl{n+1}$, and then prove \Cref{maintheorem:TypeA} in \Cref{cor:permutedqsym}. 
In these references the diagonal torus $T$ of $\GL_{n+1}$ was used instead of the adjoint torus $T_{ad}=T/\mathbb{C}^*\subset \mathrm{PGL}_{n+1}$. 
In type $A$ this is inessential to all cohomology results as the sequence $1\to \mathbb{C}^*\to T\to T_{ad}\to 1$ splits, and hence if $X$ is a variety with a $T_{ad}$ action we have canonical identifications $$H^\bullet_T(X)=H^\bullet_{T_{ad}}(X)\otimes_{H^\bullet_{T_{ad}}(pt)}H^\bullet_T(pt).$$
Moving forward in this section, we work with $T$ rather than $T_{ad}$.

The noncrossing partitions in $\NC_{n+1}\coloneqq \NC(S_{n+1},s_ns_{n-1}\cdots s_1)$ are combinatorially determined by set partitions $\mathcal{P}$ of $\{1,\ldots,n+1\}$ with an additional \emph{noncrossing} property: if $A,B\in \mathcal{P}$ are distinct blocks and $i<j<k<\ell$, then we do not have $i,k\in A$ and $j,\ell\in B$. 
To such a set partition, the associated element of $\NC_{n+1}$ is given by multiplying the backwards cycles on each of the blocks. 
The finest partition corresponds to the identity, the coarsest partition corresponds to $(n+1,n,\ldots,1)=s_ns_{n-1}\cdots s_1$, and the Kreweras order is identified with the refinement order on set partitions restricted to the noncrossing set partitions.

\begin{figure}[!ht]
    \centering
    \includegraphics[width=0.75\linewidth]{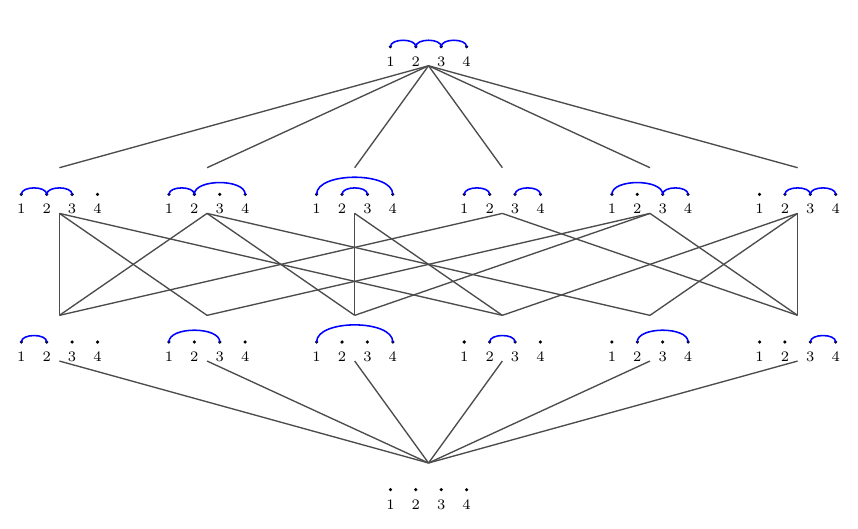}
    \caption{$\NC_4$ drawn as a lattice on noncrossing set partitions}
    \label{fig:kreweras_4}
\end{figure}

The combinatorics of $\hhmp_{n+1}$ centers around the notion of \emph{equivariant quasisymmetry} defined in \cite{BGNST1} for polynomials $f(\xl;\tl)\in \mathbb{Z}[x_1,\ldots,x_{n+1}][t_1,\ldots,t_{n+1}]$.  In particular $\hhmp_{n+1}$ has combinatorially defined Coxeter and equivariant Coxeter Schubert classes given by the forest polynomials and double forest polynomials respectively. 
We refer the reader to \cite{BGNST1,BGNST2} for details about the forest and double forest polynomials.
\begin{defn}
Define the \emph{equivariant Bergeron--Sottile maps} $\rope{i}^-,\rope{i}^+$ by
\begin{align*}\rope{i}^-f=f(x_1,\ldots,x_{i-1},t_i,x_{i},\ldots,x_{n};\tl)
\\
\rope{i}^+f=f(x_1,\ldots,x_{i-1},x_{i},t_i,\ldots,x_{n};\tl)
\end{align*}
    A polynomial $f(\xl;\tl)$ is \emph{equivariantly quasisymmetric} if for $1\le i \le n$ we have $\rope{i}^-f=\rope{i}^+f$, and we denote $\eqsym{n+1}\subset \mathbb{Z}[x_1,\ldots,x_{n+1}][t_1,\ldots,t_{n+1}]$ for the subring of equivariantly quasisymmetric polynomials. 
\end{defn}
Setting $t_i=0$, we have $\rope{i}^-$ and $\rope{i-1}^+$ become identified with the Bergeron--Sottile maps $\rope{i}$ \cite{BS98, NST_a} on the polynomial ring $\mathbb{Z}[x_1,\ldots,x_{n+1}]$ defined by
$$\rope{i}f=f(x_1,\ldots,x_{i-1},0,x_i,\ldots,x_n).$$ 
As shown in \cite{BGNST1} this specialization induces a surjection from $\eqsym{n+1}$ to the ring of \emph{quasisymmetric polynomials} $\qsym{n+1}\subset \mathbb{Z}[x_1,\ldots,x_{n+1}]$ of Gessel \cite{Ges84} and Stanley \cite{StThesis}, defined by either of the equivalent conditions that $\rope{1}f=\cdots=\rope{n+1}f$, or that for any sequence $(a_1,\dots,a_k)$ of positive integers we have equality of the coefficients $[x_{i_1}^{a_1}\cdots x_{i_k}^{a_k}]f=[x_{j_1}^{a_1}\cdots x_{j_k}^{a_k}]f$ for all increasing sequences $1\leq i_1<\cdots<i_k\leq n+1$ and $1\leq j_1<\cdots<j_k\leq n+1$.

Define the ideal $$\eqsymide{n+1}=\langle f(\xl;\tl)-f(\tl;\tl)\suchthat f\in \eqsym{n+1}\rangle\subset \mathbb{Z}[x_1,\ldots,x_{n+1}][t_1,\ldots,t_{n+1}].$$
It was shown in \cite{BGNST2} that the natural map
$H^\bullet_T(\fl{n+1})\to H^\bullet_T(\hhmp_{n+1})
$
is surjective, giving a Borel-type presentation 
\begin{align*}
    \mathbb{Z}[x_1,\ldots,x_{n+1}][t_1,\ldots,t_{n+1}]/\eqsymide{n+1} &\xrightarrow{\cong} H^\bullet_{T}(\cfl_{s_n\cdots s_1})\\
    f(x_{1}, \ldots, x_{n+1}; t_{1}, \ldots, t_{n+1})&\mapsto \big(f(t_{w(1)}, \ldots, t_{w(n+1)}; t_{1}, \ldots, t_{n+1})\big)_{w\in \NC_{n+1}}.
\end{align*}

We now prove \Cref{maintheorem:TypeA}, relating the quasisymmetric coinvariants to the cohomology rings of $\cfl_c$ for other choices of $c$.

\begin{thm}
\label{cor:permutedqsym}
    For $G/B=\fl{n+1}$ and any Coxeter element $c\in S_{n+1}$, the restriction map
$H^\bullet(\fl{n+1})\to H^\bullet(\cfl_c)$ is surjective.  
Moreover, if
$c$ is the cycle $(w(n+1)\,w(n)\,\cdots\,w(1))$ for some $w\in S_{n+1}$ then under this restriction map we can realize $H^\bullet(\cfl_c)$ as a quotient of the presentation in \eqref{eqn:flagcohomology}: 
\[
H^\bullet(\cfl_c)\cong
\mathbb{Z}[x_1,\ldots,x_{n+1}]\Big/
\Big\langle
f(x_{w(1)},\ldots,x_{w(n+1)})-f(0,\ldots,0)\ \Big|\ f\in \qsym{n+1}
\Big\rangle,
\]
the ring of \emph{permuted quasisymmetric coinvariants}.
\end{thm}
\begin{proof}
Because $c=ws_n\cdots s_1w^{-1}$, by \Cref{cor:coho_iso}, we have a map
\[
\wt{\Psi_{c,w}}:H^\bullet_{T}(\cfl_{s_n\cdots s_1})\cong H^\bullet_T(\cfl_c)
\qquad\text{given by}\qquad
(f_v)_{v\in \NC_{n+1}}\mapsto (w\cdot f_{w^{-1}vw})_{v\in \NC_{n+1}}.
\]
This map descends to an isomorphism $\Psi_{c,w}:H^\bullet(\cfl_{s_n\cdots s_1})\cong H^\bullet(\cfl_c)$ after quotienting by the ideal $H^+_T(pt) = (t_1,\ldots,t_{n+1})$.
Let $\eqsymide{n+1,w}$ be the image of $\eqsymide{n+1}$ under the automorphism of  $\mathbb{Z}[x_1,\ldots,x_{n+1}][t_1,\ldots,t_{n+1}]$ given by
\begin{align}
    \label{eq:automorphism_from_w}
    f(x_1,\ldots,x_{n+1};t_1,\ldots,t_{n+1})\mapsto f(x_{w(1)},\ldots,x_{w(n+1)};t_{w(1)},\ldots,t_{w(n+1)}).
\end{align}
Then there is a  commutative square
\begin{center}
\begin{tikzcd}
    \mathbb{Z}[x_1,\ldots,x_{n+1}][t_1,\ldots,t_{n+1}]/\eqsymide{n+1}\ar[r,"\sim"]\ar[d,"\sim"]&\mathbb{Z}[x_1,\ldots,x_{n+1}][t_1,\ldots,t_{n+1}]/\eqsymide{n+1,w}\ar[d,"\sim"]\\H^\bullet_T(\cfl_{s_n\cdots s_1})\ar[r,"\sim"]&H^\bullet_T(\cfl_c),
\end{tikzcd}
\end{center}
where the top map is induced by the automorphism in~\eqref{eq:automorphism_from_w}. 
The result then follows by quotienting out the ideal $H^+_T(pt)=(t_1,\ldots,t_{n+1})$.
\end{proof}
\begin{rem}
\label{rem:flowup}
As in the introduction we remark that the map $\Psi_{c,w}$ does not preserve flowup bases.
While the image of the flowup basis for $H^\bullet_T(\cfl_c)$ certainly maps to a basis, conjugation by $w$ only preserves the absolute order and \textit{not} the Bruhat order used for our flowup condition.
\end{rem}

\section{Preliminaries of $c$-sortability}
\label{sec:csortability}

In this section we define and recall key properties of $c$-sortability.  
Throughout, we use the right weak order $\le_{R}$ defined in Section~\ref{sec:rightorder}.
Additionally we adopt the convention that if $\bm w$ is a word in simple reflections, then $w\in W$ denotes the group element it represents.

Fix a reduced word $\bm c$ for $c$ and let $\bm{c}^{\infty}$ denote the string obtained by repeating this reduced word infinitely.
We will depict each copy of $\bm{c}$ in $\bm{c}^{\infty}$ as separated by vertical bars, so for example if $W=S_5$ and $\bm{c}=s_2s_1s_3s_4$ we have
$$\bm{{c^{\infty}}} = s_{2} \; s_{1} \;s_{3} \; s_{4} \;|\; s_{2} \; s_1 \; s_{3} \; s_{4} \;|\; s_{2} \; s_{1} \; s_3 \; s_4\;|\; \cdots$$
For any subword $\bm{x}$ of $\bm{c}^{\infty}$, we refer to the part of $\bm{x}$ in the $i$th copy of $\bm c$ in $\bm{c}^{\infty}$ as the \emph{$i$th syllable of $\bm{x}$}. 
The \emph{$c$-sorting word} $\bm{x_{c^{\infty}}}$ of an element $x\in W$ is the lexicographically first reduced word for $x$ in $\bm{c}^{\infty}$.

\begin{defn}
    A \emph{$c$-sortable element} $x\in W$ is an element such that each syllable of the $c$-sorting word $\bm{x_{c^{\infty}}}$ contains every letter in the following syllable. 
    We denote the set of $c$-sortable elements in $W$ by \emph{$\Sort(W,c)$}. 
\end{defn}

We illustrate these notions by an example.

\begin{eg}
\label{ex:sorting1}
Let $W = S_{5}$ and $\bm{c}=s_2s_1s_3s_{4} \in S_{5}$.  
We display $\bm c^\infty$ with the letters of the sorting word unstruck (and unused letters struck out in gray):
the permutations $x = 35421$ and $y = 53421$ in one-line notation have $c$-sorting words
\begin{align*}
\bm{x_{c^{\infty}}} &= s_{2} \; s_{1} \;s_{3} \; s_{4} \;|\; s_{2} \; {\color{gray} \stkout{s_{1}}} \; s_{3} \; s_{4}\;|\; s_{2} \; {\color{gray}\stkout{s_{1}} \; \stkout{s_{3}}} \; {\color{gray} \stkout{s_{4}}}\;|\; \cdots, \qquad\text{and} \\
\bm{y_{c^{\infty}}} &= s_{2} \; s_{1} \;s_{3} \; s_{4} \;|\; s_{2} \; {\color{gray} \stkout{s_{1}}} \; s_{3} \; s_{4} \;|\; s_{2} \; s_{1} \; {\color{gray}\stkout{s_{3}}} \; {\color{gray} \stkout{s_{4}}}\;|\; \cdots.
\end{align*}
Then $x$ is $c$-sortable, but $y$ is not, as $\{s_{2}, s_{3}, s_{4}\} \not \supseteq \{s_{1}, s_{2}\}$.
\end{eg}
The following lemma shows that the syllable containment property of a reduced word in $\bm{c^{\infty}}$ implies it is already the $c$-sorting word of a $c$-sortable permutation.

\begin{lem}\label{lem:lexfirstirrelevant}
Let $\bm x$ be a reduced subword of $\bm c^\infty$ representing $x\in W$.
If, syllable by syllable, the set of letters used in $\bm x$ is weakly decreasing under inclusion
(i.e.\ each syllable contains every letter appearing in the following syllable),
then $\bm x$ is the $c$-sorting word of $x$. In particular, $x\in\Sort(W,c)$.
\end{lem}
\begin{proof}
Let $\bm{x}_{c^\infty}$ be the $c$-sorting word of $x$, and assume $\bm x\neq \bm{x}_{c^\infty}$.
Let $p$ be the leftmost position in $\bm c^\infty$ where the chosen subwords differ, so $\bm{x}_{c^\infty}$ uses
$s\coloneqq \bm c^\infty_p$ at position $p$ while $\bm x$ skips it. Writing
\[
\bm{x}_{c^\infty}=\bm y\, s\, \bm z
\qquad\text{and}\qquad
\bm x=\bm y\,\bm w
\]
with $\bm y$ the common chosen prefix before $p$, we claim that $\bm w$ contains no letter $s$.
Indeed, since each simple reflection occurs at most once per syllable of $\bm c^\infty$, skipping $s$ at $p$ means
the syllable of $\bm x$ containing $p$ has no $s$, and by the syllable-containment hypothesis no later syllable can contain $s$.

Let $y,w,z\in W$ be the elements represented by $\bm y,\bm w,\bm z$.
Then $y^{-1}x=w\in W_{\Simple\setminus\{s\}}$, while also $y^{-1}x=sz$ has a reduced word beginning with $s$,
a contradiction. Hence $\bm x=\bm{x}_{c^\infty}$, and $x$ is $c$-sortable by definition.
\end{proof}

We now record some fundamental results due to Reading \cite{Reading2005} that revolve  around the study of $c$-sortability that we shall appeal to in the sequel.

\begin{fact}\leavevmode
\begin{enumerate}
    \item By~\cite[Theorem 1.1]{Reading2005}, there is a unique maximum $c$-sortable element $\sort(x)\in \Sort(W,c)$ below $x$ in the right weak order $\le_R$ called the \emph{downward projection}.
    \item The downward projection determines an equivalence relation on $W$ called \emph{Cambrian equivalence} saying $x\equiv_c y$ if $\sort(x)=\sort(y)$. 
    The proof of~\cite[Theorem 1.1]{Reading2005} shows that the Cambrian equivalence class of $x\in W$ is a $\le_{R}$-interval $[\sort(x),\upsort(x)]$,
    whose maximum $\upsort(x)$ is called the \emph{upward projection}.

    \item  $\upsort(x)$ can be computed (see ~\cite[\S 3]{Reading2005}) using the downward projection $\sort^{(c^{-1})}$ for $c^{-1}$ and the antiautomorphism $x\to x\wnaught$ of the right weak order:
\begin{equation}
\label{eqn:upsort}\upsort(x)=\sort^{(c^{-1})}(x\wnaught)\;\wnaught.\end{equation}
\end{enumerate}
\end{fact}
Figure~\ref{fig:cambrian_classes} depicts the Hasse diagram of the weak order on $S_4$ and highlights the intervals arising from Cambrian equivalence for $c=s_1s_3s_2$. So, for instance, $\{1324,3124,1342,3142,3412\}$ constitutes a single Cambrian class with $1324$ and $3412$ being the bottom and top elements respectively.

\begin{figure}[!ht]
    \centering
    \includegraphics[width=0.6\linewidth]{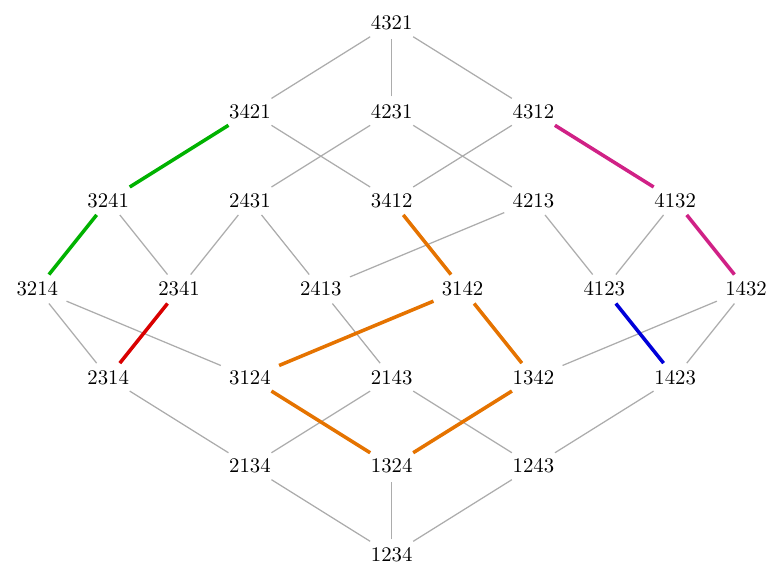}
    \caption{The right weak order on $S_4$ with non-singleton $c$-Cambrian classes emphasized using bolded edges, for $c=s_1s_3s_2$}
    \label{fig:cambrian_classes}
\end{figure}

The longest element $w_\circ\in W$ is $c$-sortable \cite[Corollary 4.4]{Reading2007}, and (recalling the notation of \eqref{eq:reflection_factorization}) the inversion factorization for its $c$-sorting word $\bm{\wnaught}$ determines the \emph{$c$-reflection order} $\le_c$ \cite{ABW07} by
$$\tau_{1}^{\bm{\wnaught}} \prec_c \tau_{2}^{\bm{\wnaught}} \prec_c \cdots \prec_c \tau_{\ell(\wnaught)}^{\bm{\wnaught}}$$
on the set of reflections $\Reflections$. We now describe the connection between $c$-sortability and $c$-reflection orders.

\begin{fact}
\label{fact:reordersort}
    Let $x\in\Sort(W,c)$ with $c$-sorting word $\bm{x}$. Then the sequence $\tau_1^{\bm x},\ldots,\tau_{\ell(x)}^{\bm x}$ can be reordered into a $\le_c$-increasing sequence by repeatedly swapping adjacent pairs of reflections that commute.
\end{fact}
\begin{proof}
By employing a certain bilinear pairing on the roots, Reading and Speyer show \cite[Proposition 3.11]{RS11} that the relative order in which two \textit{noncommuting} reflections appear in the reflection
sequence of a $c$-sorting word is forced (equivalently, any ambiguity in the sequence comes only from commuting adjacent factors).
In particular, whenever $\tau^{\bm x}_p$ and $\tau^{\bm x}_q$ do not commute, their order for $p<q$ already agrees with $\le_c$.
Thus the only way the list $\tau^{\bm x}_1,\ldots,\tau^{\bm x}_{\ell(x)}$ can fail to be $\le_c$-increasing is via commuting
reflections appearing in the ``wrong'' $\le_c$-order.
The claim now follows.
\end{proof}

\subsection{Cover reflections and skip reflections}
\label{sec:coverandskip}
For the remainder of this section we fix $$x\in \Sort(W,c)\text{ with $c$-sorting word }\bm{x}=\bm{x}_{\bm{c}^{\infty}}\subset \bm{c}^{\infty}.$$

We now proceed to describe how the $c$-sortable elements are related to noncrossing partitions. To this end we need the notion of cover reflections. 
\begin{defn}
    The set of \emph{cover reflections} for $x$ is given by
    \[
    \cov(x)=\{xsx^{-1}\suchthat s\in \des{x}\}\subset \inv{x}.
    \]
\end{defn}

Recalling the inversion factorization notation of \eqref{eq:reflection_factorization} we define
 $$\nc_c(x)=\tau^{\bm x}_{i_1}\cdots \tau^{\bm x}_{i_k}\text{ where }\cov(x)=\{\tau^{\bm x}_{i_1},\ldots,\tau^{\bm x}_{i_k}\}\text{ and }i_1<\cdots<i_k.$$
 By \Cref{fact:reordersort} the order in which the cover reflections appear in the factorization of $\nc_c(x)$ is the $c$-reflection order up to swapping adjacent commuting reflections.
 \begin{fact}[{\cite[Theorem 6.1]{Reading2007}}]\label{fact:nc_c}
The map $x \mapsto \nc_c(x)$ is a bijection
\begin{align*}
    \Sort(W,c)\to \NC(W,c)
    .
\end{align*}
Furthermore, this restricts to a bijection
$$\{x\in \Sort(W,c)\suchthat c\le_R x\}\to \NC(W,c)^+.$$
\end{fact}

In the introduction, we write $\nc_{c}(x) = \nc_{c}(\sort(x))$ for any $x \in W$.  
The cover reflections can be recovered from $\nc_{c}(x)$ in the following manner. 

\begin{fact}
[{\cite[Proposition 3.3]{BJV19}}]
\label{fac:bjv_covers_are_generators}
Let $W'$ be the smallest reflection subgroup containing $\nc_c(x)$. With respect to the induced positive system
$\Phi_{W'}^+=\Phi_{W'}\cap \Phi^+$ and simple system $\Delta_{W'}$, one has
$\cov(x)=\{s_\beta:\beta\in\Delta_{W'}\}$, and $\nc_c(x)$ is a Coxeter element of $W'$.
\end{fact}

We therefore have that the product $\tau_{i_1}^{\bm{x}}\cdots \tau_{i_k}^{\bm{x}}=\nc_c(x)$ is a minimal length reflection factorization, and by \cite[Proposition 3.4]{BJV19}  any two minimal length reflection factorizations $t_1\cdots t_k=t_1'\cdots t_k'=\nc_c(x)$ into the cover reflections in some order are related to each other by commuting adjacent transpositions -- we call any such factorization a \emph{cover reflection factorization} of $\nc_c(x)$.

\begin{eg}
\label{ex:covers1}
Continuing Example~\ref{ex:sorting1} with $W = S_{5}$ and $c=s_2s_1s_3s_{4} \in S_{5}$, and $x = 35421 \in \Sort(S_{5}, c)$, the cover reflections of $x$ are $(4\,5)$, $(2\,4)$, and $(1\,2)$.  Moreover,
\[
\nc_{c}(x) = (2\,4)(1\,2)(4\,5) = (2\,4)(4\,5)(1\,2) = 41352.
\]
The minimal reflection subgroup containing $\nc_{c}(x)$ is the subgroup $S_{\{1, 2, 4, 5\}} \subseteq S_{5}$ comprising all permutations which fix $3$, which has simple generating set $(1\,2), (2\,4)$, and $(4\,5)$.
\end{eg}

\begin{defn}\label{def:skip_position}
Say that $p \in \ZZ_{> 0}$ is a \emph{skip position} for $x$ if $\bm{x}$ does not include the $p$th letter $s = \bm{c}^{\infty}_{p}$, and moreover $\bm{x}$ contains all copies of $s$ in $\bm{c}^{\infty}$ before position $p$. We enumerate the skip positions for $x$ in increasing order (i.e.~from left to right) as $$p_{1} < p_{2} < \cdots < p_{n}.$$   
\end{defn}
See Example~\ref{eg:moreskips} where the circled positions denote skip positions. Each skip position $p_{j}$ determines a special reflection as follows.

\begin{defn}\label{def:skip_reflection}
For $1\leq j\leq n$, let $a_{1} a_{2} \cdots a_{r_{j}}$ be the prefix of $\bm{x}$ consisting of all letters to the left of the skip position $p_{j}$, and let $s_{j} = \bm{c}^{\infty}_{p_{j}}$.
We define the $j$th \emph{skip reflection} (or just ``skip'') by
\[
\phi_{j} = a_{1} \cdots a_{r_{j}} s_{j} a_{r_{j}} \cdots a_{1},
\]
and set $\skips{x}\coloneqq\{\phi_1,\ldots,\phi_n\}$
\end{defn}

We will use the following fact repeatedly.

\begin{prop}
\label{prop:skipfill}
For any $1\leq i_1<\cdots <i_k\leq n$, the product $\phi_{i_{1}}\phi_{i_{2}}\cdots \phi_{i_{k}} x$ has a (possibly non-reduced) word obtained from $\bm{x}\subset \bm{c^{\infty}}$ by filling in the  skips associated to $\phi_{i_1},\ldots,\phi_{i_k}$.
\end{prop}
\begin{proof}
Write $\bm x=a_1a_2\cdots a_{\ell(x)}$ and 
take $r_{i_1}, r_{i_2}, \ldots, r_{i_k}$ and $s_{i_{1}}, \ldots, s_{i_{k}}$ as in Definition~\ref{def:skip_reflection}.  Then 
\begin{align*}
\phi_{i_{1}}\phi_{i_{2}}\cdots \phi_{i_{k}} x
&= (a_{1} \cdots a_{r_{i_1}}) s_{i_{1}} (a_{r_{i_1}} \cdots a_{1}) 
\cdots 
(a_{1} \cdots a_{r_{i_k}}) s_{i_{k}} (a_{r_{i_k}} \cdots a_{1}) 
a_{1}a_{2} \cdots a_{\ell(x)} \\
&= a_{1} \cdots a_{r_{i_1}} \; s_{i_{1}} \; a_{r_{i_1}+1} \cdots a_{r_{i_2}}\; s_{i_{2}} \; a_{r_{i_2}+1}\cdots  a_{r_{i_k}} \; s_{i_{k}}  \; a_{r_{i_k}+1}\cdots a_{\ell(x)}
\end{align*}
as desired.
\end{proof}

As a consequence we obtain a minimal reflection factorization for $c$ from the skip reflections.
\begin{cor}
\label{cor:prodc}
We have $\phi_1\cdots \phi_n=c$.
\end{cor}
\begin{proof}
Let $\bm x'$ be the (possibly nonreduced) word for $(\phi_1\cdots\phi_n)x$ given by Proposition~\ref{prop:skipfill},
obtained from $\bm x\subset \bm c^\infty$ by filling all skip positions.

Fix a simple reflection $s\in \Simple$, and let $m_s$ be the number of occurrences of $s$ used by $\bm x$ in $\bm c^\infty$.
Since the syllables of $\bm x$ are weakly decreasing under inclusion, $\bm x$ uses precisely the first $m_s$ occurrences of $s$
in $\bm c^\infty$. 
Hence the skip position for $s$ is its $(m_s+1)$st occurrence, and filling all skips forces $\bm x'$ to use
the first $m_s+1$ occurrences of $s$.

In particular, every $s\in \Simple$ occurs in the first syllable of $\bm x'$, so the first syllable of $\bm x'$ is exactly $\bm c$,
and  the $(i+1)$st syllable of $\bm x'$ coincides with the $i$th syllable of $\bm x$ for all $i\ge 1$.
Therefore $\bm x'=\bm c\,\bm x$ as subwords of $\bm c^\infty$, and hence $(\phi_1\cdots\phi_n)x=cx$.
The claim follows.
\end{proof}

\subsection{Forced and unforced skip reflections}
Following Reading’s Cambrian theory \cite{Reading2007}, the $c$-sorting word of a $c$-sortable element carries combinatorial data via cover reflections that determines its associated noncrossing partition.  Reading--Speyer \cite{RS11} show that this data can be read directly from the pattern of
omitted letters in $\bm c^\infty$: the skips that are \textit{forced} by reducedness are precisely the cover reflections, and hence encode
$\nc_c(x)$, while the remaining skips are optional insertions, motivating the distinction between forced and unforced skips
\cite[Proposition~5.2]{RS11}.

Like before, we write $x\in \Sort(W,c)$ as  $\bm x=a_1a_2\cdots a_{\ell(x)}$.
\begin{defn}
    Say that a skip $\phi_{j}$ is \emph{unforced} if $ a_{1} \cdots a_{r_{j}}s_{j}$ is a reduced expression in $W$, and  \emph{forced} otherwise. 
    Equivalently, letting $y_j=a_1\cdots a_{r_j}$, the skip $\phi_j$ is unforced iff $\ell(y_js_j)=\ell(y_j)+1$,
and forced iff $\ell(y_js_j)=\ell(y_j)-1$.
    
\end{defn}
If we let $\uskips{x}$ and $\fskips{x}$ denote the respective sets of unforced and forced skips, then we have a decomposition
\[
    \skips{x}=\uskips{x}\sqcup \fskips{x}.
\]
We will enumerate elements of these sets according to their skip positions in $\bm{c}^{\infty}$ as 
\begin{align*}
\fskips{x}&=\{\fs_1,\ldots,\fs_k\},\qquad\qquad\text{where $k \coloneq |\fskips{x}|$, and} \\
\uskips{x}&=\{\ufs_1,\ldots,\ufs_{n-k}\}.
\end{align*}
This is the same order used for $\skips{x}$, so $\phi_1,\ldots,\phi_n$ is a shuffle of the lists $\ufs_1,\ldots,\ufs_{n-k}$ and $\fs_1,\ldots,\fs_k$.

\begin{eg}
\label{eg:moreskips}
As in Example~\ref{ex:sorting1} take $c=s_2s_1s_3s_{4} \in S_{5}$ and $x = 35421 \in \Sort(W, c)$.   
The skip positions are the circled letters in the $c$-sorting word
\begin{align*}
\bm{x} &= s_{2} \; s_{1} \;s_{3} \; s_{4} \;|\; 
s_{2} \; \circled{\color{gray} \stkout{s_{1}}} \; s_{3} \; s_{4}\;|\; 
s_{2} \; {\color{gray}\stkout{s_{1}}} \; \circled{\color{gray}\stkout{s_{3}}} \; \circled{\color{gray} \stkout{s_{4}}}\;|\; 
\circled{\color{gray}\stkout{s_{2}}} \; {\color{gray}\stkout{s_{1}} \; \stkout{s_{3}}} \; {\color{gray} \stkout{s_{4}}}\;|\; \cdots
\end{align*}
The corresponding skip reflections are, in order, $(3\, 4) = \ufs_1$, $(2\, 4) = \fs_1$, $(1\, 2) = \fs_2$, and $(4\,5) = \fs_{3}$.
\end{eg}

The agreement of forced skips and cover reflections in Example~\ref{eg:moreskips} is not accidental.

\begin{fact}\label{fact:fskip}
$\cov(x) = \fskips{x}$, and 
$\fs_1\cdots \fs_k=\nc_c(x)$ is a cover reflection factorization.
\end{fact}
\begin{proof}
By \cite[Proposition 5.2]{RS11}, the forced skips are precisely the cover reflections, so
$\cov(x)=\fskips{x}$ as sets.  Let $u\coloneqq \fs_1\cdots \fs_k$, i.e. the product in skip-position order.

By definition, $\nc_c(x)$ is the product of the cover reflections in the order they appear in the reflection sequence
of the $c$-sorting word (equivalently, in $c$-reflection order up to commuting adjacent commuting reflections).
Moreover, \cite[Remark 3.2(2)]{BJV19} implies that for any two noncommuting reflections in $\cov(x)$,
their relative order agrees in any cover reflection factorization. 
Hence the words $u$ and $\nc_c(x)$ differ only by swapping
adjacent commuting reflections, and in particular $u=\nc_c(x)$.

Finally, by \cite[Proposition 3.4]{BJV19} any minimal reflection factorization of $\nc_c(x)$ into the cover reflections
is unique up to swapping adjacent commuting factors, so $\fs_1\cdots \fs_k$ is a cover reflection factorization.
\end{proof}

\begin{prop}
\label{prop:insertunforced}
The following hold.
\begin{enumerate}[label=(\arabic*)]
    \item\label{it1:9.17} We have $x \lessdot_{B} \ufs_{i} x$ for all $1 \le i \le n-k$.
    \item \label{it2:9.17}$\ufs_{i} x$ is $c$-sortable, and $\ufs_{1}, \ldots, \ufs_{i-1}$ are unforced skips of $\ufs_{i} x$.
\end{enumerate}  
\end{prop}
\begin{proof}
    Fix $i$. 
    By Proposition~\ref{prop:skipfill}, the product $\ufs_i x$ admits an expression $\bm x'$ obtained from the
$c$-sorting word $\bm x\subset \bm c^\infty$ by filling the skip corresponding to $\ufs_i$ (i.e.\ inserting the skipped letter
at its skip position). 
In particular, $\bm x'$ has length $\ell(x)+1$, so $\ell(\ufs_i x)\le \ell(x)+1$.
On the other hand, \cite[Lemma 5.7]{RS11} gives $\ell(\ufs_i x)>\ell(x)$, hence $\ell(\ufs_i x)=\ell(x)+1$.
Therefore $\bm x'$ is reduced and $x\lessdot_B \ufs_i x$, proving~\ref{it1:9.17}.

Since $\bm x'$ is a reduced subword of $\bm c^\infty$ and is obtained from $\bm x$ by adding a single letter in one syllable,
it still satisfies the syllable-containment condition. 
Thus Lemma~\ref{lem:lexfirstirrelevant} implies that $\bm x'$ is the
$c$-sorting word of $\ufs_i x$, so $\ufs_i x$ is $c$-sortable.

Finally, for any $j<i$, the skip position of $\ufs_j$ occurs before that of $\ufs_i$, so the prefix of $\bm x'$ up to that position
agrees with the corresponding prefix of $\bm x$. 
Hence the reduced/nonreduced status of the word $a_1\cdots a_{r_j}s_j$
is unchanged, and $\ufs_j$ remains an unforced skip for $\ufs_i x$. 
This proves~\ref{it2:9.17}.
\end{proof}

\begin{prop}
\label{prop:skipscommute}
Let $\phi_{i}, \phi_{j} \in \skips{x}$ with $i < j$.  If $\phi_{i}$ is forced and $\phi_{j}$ is unforced, then $\phi_{i} \phi_{j} = \phi_{j} \phi_{i}$.
\end{prop}
\begin{proof}
We may truncate $\bm x$ immediately after the skip position of $\phi_j$; this does not change $\phi_i$ or $\phi_j$
since both are defined using prefixes before that position. In particular, we may assume the skip position of $\phi_j$
occurs after all letters of $\bm x$, so $\phi_j = x s_j x^{-1}$ and hence $\phi_j x = x s_j$.

Since $\phi_i$ is forced, it is a cover reflection by Fact~\ref{fact:fskip}, so $
    \phi_i=x s x^{-1}$  for some $s\in\des{x}$.
Now consider the element $x s_j = \phi_j x$, which is $c$-sortable by \Cref{prop:insertunforced}.
Appending $s_j$ does not change the prefix data for the earlier skip position $i$,
so $\phi_i$ is still a forced skip for $x s_j$, hence again a cover reflection of $x s_j$.
Thus 
$\phi_i = (x s_j)\, s'\, (x s_j)^{-1}$  for some  $s'\in\des{x s_j}$.

Equating the two expressions for $\phi_i$ gives
\[
x s x^{-1} = x s_j s' s_j x^{-1},
\quad\text{so}\quad
s = s_j s' s_j.
\]
Since $s$ and $s'$ are simple reflections, the conjugate $s_j s' s_j$ can be simple only if $s_j$ commutes with $s'$,
in which case $s_j s' s_j=s'$. Therefore $s=s'$ and $s$ commutes with $s_j$.
Conjugating by $x$ shows $\phi_i$ commutes with $\phi_j$.
\end{proof}

Applying these commutation relations to the identity in \Cref{cor:prodc} gives us the following alternate factorizations of $c$ into skip reflections.
\begin{cor}
\label{cor:ufsthenfs} $\ufs_1\cdots \ufs_{n-k}\fs_1\cdots \fs_k=c$.
\end{cor}

\section{Characterizing equivalent translated intervals}
\label{section:CharacterizeBruhatIntervals}
We now determine when $\movetoid{w'}{w'\cdot c}=\movetoid{w''}{w''\cdot c}$ using Cambrian equivalence by identifying certain distinguished decreasing chains in the $c$-reflection order. The main theorem of this section was also independently shown by Defant--Sherman-Bennett--Williams~\cite[Theorem~1.9 and Corollary~1.10]{DSBW25}. To transition between the condition $\ell(wc)=\ell(w)+\ell(c)$ and $c$-sortable combinatorics, we need the following transformation.
\begin{defn}
    For each $w \in W$, define $w_{op}\coloneqq w^{-1}\wnaught.$
\end{defn}
Note that $w\mapsto w_{op}$ is a bijection, with inverse $x\mapsto \wnaught x^{-1}$.
\begin{prop}
\label{prop:numbcambclasses}\leavevmode
\begin{enumerate}[label=(\arabic*)]
    \item \label{it1:10.2} $\ell(wc)=\ell(w)+\ell(c)$ is equivalent to $c\le_R w_{op}$.
    \item \label{it2:10.2} $\{x\suchthat c\le_R x\}$ is the union of $\cat{W}^+$-many Cambrian equivalence classes.
\end{enumerate}
\end{prop}
\begin{proof}
    \ref{it1:10.2} follows because $c\le_R w_{op}$ is equivalent to $\ell(c)+\ell(c^{-1}w^{-1}\wnaught)=\ell(w^{-1}\wnaught)$, and we can apply the identity $\ell(y^{-1}\wnaught)=\ell(\wnaught)-\ell(y)$ to both sides for $y=wc$ and $y=w$ respectively. 
    For~\ref{it2:10.2}, by \Cref{fact:nc_c} and the fact that $c\le_R x$ implies $c\le_R \sort(x)$ by definition of $\sort(x)$, we have that \[
    c\le_R x\Longleftrightarrow \nc_c(\sort(x))\in \NC(W,c)^+.\] 
    Because $|\NC(W,c)^+|=\cat{W}^+$  we conclude.
\end{proof}

\begin{thm}
\label{thm:Equivwwc}
    For $w', w''\in \{w\suchthat \ell(wc)=\ell(w)+\ell(c)\}$, or equivalently with $(w')_{op},(w'')_{op}\in \{x\suchthat c\le_R x\}$, the following are equivalent.
    \begin{enumerate}
        \item $(w')_{op}\equiv_c (w'')_{op}$
        \item $\movetoid{w'}{w'c}=\movetoid{w''}{w''c}$
        \item $\movetoid{w'}{w'c}$ and $\movetoid{w''}{w''c}$ have the same $c$-decreasing chain.
    \end{enumerate}
    In particular, there are $\cat{W}^+$-many distinct translated Bruhat intervals $\movetoid{w}{wc}$ in $\NC(W, c)$.
\end{thm}

\begin{proof}
By \Cref{cor:shapeequiv} and \Cref{cor:newguy} there are exactly $\cat{W}^+$-many translates $\movetoid{w}{wc}$, and by \Cref{prop:uniquecdec} each of these contains a unique $c$-decreasing chain. By \Cref{prop:numbcambclasses} there are exactly $\cat{W}^+$-many $c$-decreasing chains in $\NC(W,c)$, which establishes the equivalence between (2) and (3). By Corollary~\ref{cor:numbdecreasing} the number of $c$-Cambrian equivalence classes in $\{x\suchthat c\le_R x\}$ is also equal to $\cat{W}^+$, so to show (1) and (3) are equivalent it therefore suffices to show the following.

\smallskip
\noindent\textbf{Goal.} For $w$ with $w_{op}\in\{x\suchthat c\le_R x\}$, the $c$-decreasing (reflection-labeled) maximal chain in $\movetoid{w}{wc}$ depends only on the $c$-Cambrian class of $w_{op}$.

\smallskip
Fix such a $w$, and set
\[
x\coloneqq \sort(w_{op}).
\]
Then $x$ depends only on the $c$-Cambrian class of $w_{op}$.  
Moreover $c\le_R x$ and, by \Cref{lem:lexfirstirrelevant}, deleting the first syllable $c$ from the $c$-sorting word $\bm{x}$ produces the $c$-sorting word $\bm{c^{-1}x}$ for $c^{-1}x$, so in particular $c^{-1}x\in \Sort(W,c)$.

Let
\[
    \skips{c^{-1}x}=\{\phi_1,\ldots,\phi_n\}.
\]
Define reflections $\psi_1,\ldots,\psi_n$ by
\[
\psi_i\coloneqq \phi_1\cdots \phi_{i-1}\,\phi_i\,\phi_{i-1}\cdots \phi_1\qquad (1\le i\le n).
\]
A $c$-decreasing maximal chain is determined by its set of labels, so it is enough to prove that the unique $c$-decreasing maximal chain in $\movetoid{w}{wc}$ is labeled by $\psi_1,\ldots,\psi_n$ (in $c$-reflection order).  The remainder of the proof is organized in three steps.

\smallskip
\noindent\textbf{Step 1: a distinguished Bruhat chain in $[c^{-1}x,x]$.}
Set
\[
x_i\coloneqq \psi_i\psi_{i+1}\cdots \psi_n\,x \qquad (1\le i\le n),
\qquad\text{and}\qquad x_{n+1}:=x.
\]
Then $x_1=c^{-1}x$ and $x_{i+1}=\psi_i x_i$ for $1\le i\le n$.  We claim that
\[
x_1
\lessdot_{B}
x_2
\lessdot_{B}
\cdots
\lessdot_{B}
x_{n+1}
\]
is a maximal Bruhat chain in the interval $[c^{-1}x,x]$.

To verify the covering relations, write the $c$-sorting word $\bm{c^{-1}x}$ (as a subword of $\bm c^\infty$) and fix $1\le i\le n$.  Let $\bm{a_i}\bm{b_i}$ denote the portions of $\bm{c^{-1}x}$ occurring before and after the $i$th skip position, and let the letter at that skip position be $\bm{s_i}$ so that, on the level of group elements,
\[
\phi_i\,c^{-1}x=a_i s_i b_i.
\]
Let $\bm{a_i'}$ be the word obtained from $\bm{a_i}$ by filling in the first $i-1$ skips of $\bm{c^{-1}x}$.  Then
\[
\phi_1\cdots \phi_{i-1}\,c^{-1}x \;=\; a_i' b_i,
\qquad\text{hence}\qquad
x_i=\psi_i\cdots \psi_n x=\phi_1\cdots\phi_{i-1}c^{-1}x=a_i'b_i.
\]

We show by induction on $i$ that $\bm{a_i'}\bm{b_i}$ is a reduced $c$-sorting word for a $c$-sortable element (namely $x_i$).  
For $i=1$ this is exactly $\bm{c^{-1}x}$.  Assuming the claim for $\bm{a_i'}\bm{b_i}$, the $i$th skip position of $c^{-1}x$ is also a skip position for $\bm{a_i'}\bm{b_i}$, and inserting the letter $\bm{s_i}$ fills this position:
\[
\bm{a_{i+1}'}\bm{b_i}=\bm{a_i'}\bm{s_i}\bm{b_i}.
\]
This insertion is unforced, because $\bm{a_i'}\bm{s_i}$ is a prefix of $\bm{x}$ and hence reduced.  
Therefore \Cref{prop:insertunforced} implies that $\bm{a_{i+1}'}\bm{b_i}$ remains reduced.  Finally, \Cref{lem:lexfirstirrelevant} implies that $\bm{a_{i+1}'}\bm{b_{i+1}}$ is again a $c$-sorting word of a $c$-sortable element.  
This completes the induction and establishes that each step $x_i\lessdot_B x_{i+1}$ is a Bruhat cover.

\smallskip
\noindent\textbf{Step 2: lifting the chain to $[c^{-1}w_{op},w_{op}]$.}
Since $x\le_R w_{op}$, let $z\coloneqq x^{-1}w_{op}$ and choose a reduced word $\bm{z}$ for $z$ so that $\bm{x}\bm{z}$ is a reduced word for $w_{op}$.   
For each $i$ set
\[
y_i\coloneqq \psi_i\psi_{i+1}\cdots \psi_n\,w_{op}\qquad (1\le i\le n),
\qquad\text{and}\qquad y_{n+1}\coloneqq w_{op}.
\]
Then $y_1=c^{-1}w_{op}$ and $y_{i+1}=\psi_i y_i$.

We claim that
\[
y_1
\lessdot_{B}
y_2
\lessdot_{B}
\cdots
\lessdot_{B}
y_{n+1}
\]
is a maximal Bruhat chain in $[c^{-1}w_{op},w_{op}]$.  Using the notation from Step~1, we have
\[
y_i=\psi_i\cdots \psi_n w_{op}= (a_i'b_i)\,z,
\]
so it suffices to prove that $\bm{a_i'}\bm{b_i}\bm{z}$ is reduced for each $i$.  Proceed by induction on $i$, with the base case $i=1$ given by the reduced word $\bm{c^{-1}x}\bm{z}$ for $c^{-1}w_{op}$.

Assume $\bm{a_i'}\bm{b_i}\bm{z}$ is reduced.  Since $y_{i+1}=\psi_i y_i$ and $\psi_i a_i'=a_i's_i$, the word
\[
\bm{a_{i+1}'}\bm{b_i}\bm{z}=\bm{a_i'}\bm{s_i}\bm{b_i}\bm{z}
\]
represents $y_{i+1}$ and has one more letter than $\bm{a_i'}\bm{b_i}\bm{z}$.  Thus it is enough to show that $\psi_i\notin \inv{a_i'b_iz}$ (equivalently, that left-multiplication by $\psi_i$ increases length by $1$).
Using
\[
\inv{a_i'b_iz}=\inv{a_i'} \,\cup\, (a_i')\inv{b_iz}(a_i')^{-1},
\]
and the fact that $\psi_i\notin \inv{a_i'}$ (since $\psi_i a_i'=a_i's_i$ and $\bm{a_i'}\bm{s_i}$ is reduced from Step~1), it remains to show that $\psi_i\notin (a_i')\inv{b_iz}(a_i')^{-1}$, i.e.\ that $s_i\notin \inv{b_iz}$.  As before, by \Cref{lem:lexfirstirrelevant}, we have $\sort(c^{-1}w_{op})=c^{-1}\sort(w_{op})=c^{-1}x$. If $s_i\in \inv{b_iz}$, then one can delete the letter of $\bm{b_i}\bm{z}$ corresponding to this inversion and prepend $\bm{s_i}$, contradicting the $\le_R$-maximality of the $c$-sorting word $\bm{c^{-1}x}=\bm{a_i}\bm{b_i}$ for $c^{-1}x=\sort(c^{-1}w_{op})$.  Hence $s_i\notin \inv{b_iz}$, completing the induction and proving the Bruhat covering claims.

\smallskip
\noindent\textbf{Step 3: transporting to the interval $\movetoid{w}{wc}$.}
Applying the map $v\mapsto \wnaught v^{-1}$ to the chain in Step~2 yields a maximal chain
\[
w 
\lessdot_{B}
w \psi_{n}
\lessdot_{B}
w \psi_{n}\psi_{n-1}
\lessdot_{B}
\cdots
\lessdot_{B}
w \psi_{n}\psi_{n-1}\cdots \psi_{1},
\]
whose reflection labels are, in order, $\psi_n,\ldots,\psi_1$.

These same reflections appear (in reverse order) in the reflection sequence associated to the $c$-sorting word of $w_{op}$, corresponding to the final occurrence of each simple reflection.  Therefore, by \Cref{fact:reordersort}, we can repeatedly swap adjacent commuting reflections to obtain a reordering
\[
\psi_1',\psi_2',\ldots,\psi_n'
\]
that is in reverse $c$-reflection order.  Since only commuting reflections are swapped, the chain
\[
w 
\lessdot_{B}
w \psi_{n}'
\lessdot_{B}
w \psi_{n}'\psi_{n-1}'
\lessdot_{B}
\cdots
\lessdot_{B}
w \psi_{n}'\psi_{n-1}'\cdots \psi_{1}'
\]
remains inside the same Bruhat interval $[w,wc]$, hence is the unique $c$-decreasing maximal chain in $\movetoid{w}{wc}$ by \Cref{prop:uniquecdec}.  In particular, its label set is $\{\psi_1,\ldots,\psi_n\}$, which depends only on $x=\sort(w_{op})$, and hence only on the $c$-Cambrian class of $w_{op}$.

This proves the Goal.  
Combining this with the counting discussion at the start of the proof yields the equivalences (1)--(3), and the final assertion follows.
\end{proof}

\begin{eg}
\label{ex:B4-fullchains}
Work in type $B_4$ in one-line notation $w(1)\,w(2)\,w(3)\,w(4)$, writing $\overline{i}\coloneqq -i$.
Let $s_0,s_1,s_2,s_3$ be the standard simple reflections, where $s_0$ negates the first entry.
Fix ${\bm c}=s_0s_1s_2s_3$. 

Say we take
$w_{op}=\overline{3}\;\overline{4}\;2\;\overline{1}$, so that $c\le_R w_{op}$.
Its $c$-sorting word $\bm{w_{op}}$ is $
s_0\,s_1\,s_2\,s_3\;\;s_1\,s_2\;\;s_0\,s_1\;\;s_0\,s_1$.
The downward Cambrian projection is $
x=\sort(w_{op})=4\;3\;2\;\overline{1},$
whose $c$-sorting word $\bm x$ is given by
$\bm{x}=s_0\,s_1\,s_2\,s_3\;\;s_1\,s_2\,s_1$.
Deleting the initial syllable $\bm c$ from $\bm{x}$ gives $c^{-1}x=3\;2\;1\;4$ with $c$-sorting word $\bm{c^{-1}x}=s_1s_2s_1.$
The skip positions of $c^{-1}x$ are $\{1,4,7,10\}$ with simple reflections at these positions being $s_0$, $s_3$, $s_2$, and $s_1$ respectively.
We thus get 
\[
    \skips{c^{-1}x}=\{\phi_1,\dots,\phi_4\}=\{s_0,s_1s_2s_3s_2s_1,s_1,s_2\}=\{\overline{1}\;2\;3\;4, 4\;2\;3\;1, 2\;1\;3\;4, 1\;3\;2\;4 \}.
\]
The conjugated reflections 
$\psi_i=\phi_1\cdots\phi_{i-1}\,\phi_i\,\phi_{i-1}\cdots\phi_1$ are then computed to equal
\[
\psi_1=\overline{1}\;2\;3\;4,\qquad
\psi_2=\overline{4}\;2\;3\;\overline{1},\qquad
\psi_3=1\;4\;3\;2,\qquad
\psi_4=1\;2\;4\;3.
\]
The Bruhat chain in $[c^{-1}x,x]$ constructed in Step~1 in the proof is given by
\[
3\;2\;1\;4
\xrightarrow{\psi_1}
3\;2\;\overline{1}\;4
\xrightarrow{\psi_2}
3\;2\;4\;\overline{1}
\xrightarrow{\psi_3}
3\;4\;2\;\overline{1}
\xrightarrow{\psi_4}
4\;3\;2\;\overline{1}.
\]
These reflections then give the following Bruhat chain in $[c^{-1}w_{op},w_{op}]$ as per Step~2 in the proof.
\[
\overline{2}\;\overline{3}\;1\;4
\xrightarrow{\psi_1}
\overline{2}\;\overline{3}\;\overline{1}\;4
\xrightarrow{\psi_2}
\overline{2}\;\overline{3}\;4\;\overline{1}
\xrightarrow{\psi_3}
\overline{4}\;\overline{3}\;2\;\overline{1}\xrightarrow{\psi_4}
\overline{3}\;\overline{4}\;2\;\overline{1}.
\]
We now transport to the translated interval $\movetoid{w}{wc}$ as per Step~3 in the proof.
We have $w_\circ=\overline{1}\;\overline{2}\;\overline{3}\;\overline{4}$ and so
$w\coloneqq w_\circ w_{op}^{-1}=4\;\overline{3}\;1\;2$.
Then $wc=\overline{3}\;1\;2\;\overline{4}$.
and applying the inverse of the $op$-map to the previous chain gives the maximal chain in $[w,wc]$:
\[
w=4\;\overline{3}\;1\;2
\xrightarrow{\psi_4}
4\;\overline{3}\;2\;1
\xrightarrow{\psi_3}
4\;1\;2\;\overline{3}
\xrightarrow{\psi_2}
3\;1\;2\;\overline{4}
\xrightarrow{\psi_1}
\overline{3}\;1\;2\;\overline{4}=wc.
\]
Thus the labels in the translated interval $\movetoid{w}{wc}$ are exactly $\psi_1,\psi_2,\psi_3,\psi_4$ (appearing along the chain as $\psi_4,\psi_3,\psi_2,\psi_1$), as in the proof.
We reconcile this with the $c$-reflection order for our choice of $c$. 
The $c$-sorting word for $\wnaught$ is $\bm c^4$ and in the resulting reflection order, the $\psi_1$ through $\psi_4$ appear in positions $1$, $4$, $15$ and $16$ respectively. 
Thus the chain constructed last is the $c$-decreasing maximal chain in $[w,wc]$.
\end{eg}

\begin{cor}
\label{cor:intervals_cprime}
Let $W' \subseteq W$ be a standard parabolic subgroup and let $c'\le_B c$ be the corresponding Coxeter element.  
If $w \in W$ has $\sort(w_{op}) \in W'$, then denoting $x\in W$ for the element with $x_{op}=\sort(w_{op})$, we have $\movetoid{w}{wc'} = \movetoid{x}{xc'}$.
\end{cor}
\begin{proof}
We first recall the well-known fact~\cite[\S 1.10]{Hum75} that $W'$ has a complete set of right coset representatives $M$ such that for every $m \in M$ and every $v \in W'm$, we have $v = v' \cdot m$ for $v' = vm^{-1} \in W'$.  
If $u \in W'$ has $u \le_{R} v$, then applying the above to $u^{-1}v$ shows that $u \le_{R} v'$.  
Furthermore, if $v \lessdot_{B} \tau v$ is a Bruhat cover in $W'$, then $vm \lessdot_{B} \tau vm$ is a Bruhat cover in $W$.  

Now factor $w_{op} = w'_{op} \cdot m$.    
As $c'$-sortablility and $c$-sortability coincide in $W'$, the $\le_{R}$-preserving property implies that $x_{op}$ is the maximal $c'$-sortable below $w'$. Thus Theorem~\ref{thm:Equivwwc} implies that   $[(c')^{-1}x_{op}, x_{op}]$ and $[(c')^{-1}w_{op}', w_{op}']$ are isomorphic as edge-labeled posets with each $v \lessdot_{B} \tau v$ labeled by $\tau$.  
By the cover-preserving property, $[(c')^{-1}w_{op}', w_{op}']$ is isomorphic in the same way to $[(c')^{-1}w_{op}, w_{op}]$.
Transporting intervals as in Step 3 of the proof above, we have $\movetoid{x}{x c'} = \movetoid{w}{wc'}$. 
\end{proof}

\section{Clusters, noncrossing inversions, and Bruhat maximal elements}
\label{sec:Bruhatmax}

We now give a novel characterization of Reading's map from $c$-sortable elements to noncrossing partitions; see~\Cref{fact:nc_c}.  
The following result implies~ \Cref{maintheorem:BruhatMax}, as we have established in \Cref{cor:newguy} that $w^{-1}X^{w\cdot c}_w=X^u_{\NC}$ for $u\in w^{-1}[w,w\cdot c]$ the Bruhat-maximum element, and in Corollaries~\ref{cor:w'w'c'equalsuNC} and~\ref{cor:intervals_cprime} that all cases are equivalent to considering noncrossing partitions which are fully supported in an appropriate standard parabolic subgroup.  

\begin{thm}
\label{thm:BruhatMaxIsNC}
    The Bruhat-maximal element of $\movetoid{w}{w\cdot c}$ is $u=\nc_c(\sort(w_{op}))$, so the Bruhat maximal elements of $w^{-1}[w,w\cdot c]$ and $(w')^{-1}[w',w'\cdot c]$ are equal if and only if $w_{op} \equiv_{c} w_{op}'$.
    
    In particular, the $n$ elements adjacent to $u$ in $\movetoid{w}{w\cdot c}\subset \NC(W,c)$ are $\{\tau u\suchthat \tau\in \invnc{u}\}.$
\end{thm}

We illustrate Theorem~\ref{thm:BruhatMaxIsNC}  before laying the groundwork for the proof.

\begin{eg}
    We take $W=S_5$ and $c=s_2s_1s_3s_4$ as in Example~\ref{eg:moreskips}. For $w=12534$ it is seen that $\ell(wc)=\ell(w)+\ell(c)$. One checks that $w_{op}=35421$, which is $c$-sortable and thus satisfies $\sort(w_{op})=w_{op}$. Using the forced skips computed in Example~\ref{eg:moreskips} together with Fact~\ref{fact:fskip} we see that $u=\nc_c(\sort(w_{op}))=41352$. A direct computation tells us that 
    \[
\movetoid{w}{wc}
=
\left\{
\begin{array}{cccc}
12345,&12354,&13245,&13254,\\
13425,&13452,&14325,&14352,\\
21345,&21354,&31245,&31254,\\
31425,&31452,&41325,&41352
\end{array}
\right\},
\]
and that $u=41352$ is indeed the Bruhat-maximal element.  
Finally, Theorem~\ref{thm:BruhatMaxIsNC} says that the elements of $\movetoid{w}{wc}$
adjacent to $u$ are precisely the four permutations $\tau u$ with $\tau\in \invnc{u}$, namely
\[
(2\,4)u=21354,\qquad
(1\,4)u=14352,\qquad
(2\,5)u=41325,\qquad
(3\,4)u=31452.
\]
We remark that in the Bruhat interval $[w,wc]$, the element $wu$ equals $31542$, and among the four $w\tau u$ for $\tau\in \invnc{u}$, one covers $wu$ while the remaining are covered by $wu$.
\end{eg}

Biane and Josuat-Verg\`es describe the right noncrossing inversion set $\invncR{u}$ used to define $\nctocluster$ explicitly in terms of cover reflection factorizations.  
For $u \in \NC(W, c)$, note that $u^{-1}c \in \NC(W,c)$ as well. Indeed, if $c = \tau_{1} \cdots \tau_{n}$ with $u = \tau_{1} \cdots \tau_{k}$, then $u^{-1}c = \tau_{k+1} \cdots \tau_{n}$, and we have a minimal factorization $c = \tau_{k+1} \cdots \tau_{n} \sigma_{1} \cdots \sigma_{k}$ where $\sigma_{i} = (u^{-1}c)^{-1} \tau_{i} (u^{-1}c)$.

\begin{fact} [\cite{BJV19}]
\label{lem:BianeInversion}
    Let $u\in \NC(W,c)^+$.  Fix cover reflection factorizations
\[
u=t_1\cdots t_k
\qquad\text{and}\qquad
u^{-1}c=t_{k+1}\cdots t_n.
\]
    The set of right noncrossing inversions $\invncR{u}=\{u^{-1}\tau u\suchthat \tau\in \invnc{u}\}$ of $u$ are given by
    \begin{enumerate}
        \item the reflections $p_1,\ldots,p_k$ such that
        $$t_1\cdots t_k p_i=t_1\cdots \wh{t_i}\cdots t_k,$$
        \item the reflections $q_1,\ldots,q_{n-k}$ such that
        $$q_it_{k+1}\cdots t_n=t_{k+1}\cdots \wh{t_{k+i}}\cdots t_n.$$
    \end{enumerate}
\end{fact}
Of all cover reflection factorizations of $u$ and $u^{-1}c$ from \Cref{lem:BianeInversion}, there are two which are most useful for us. They will be determined by the forced and unforced skips of the $c$-sortable $x\in \Sort(W,c)$ with $c\le_R x$ (\Cref{fact:nc_c}) for which $\nc_c(x)=u$.

\begin{lem}
\label{lem:covstop}
For $w\in W$, the element $\upsort(w_{op})w_{\circ}$ is $c^{-1}$-sortable, and
\[
\cov\big(\upsort(w_{op})w_{\circ}\big)
=
\uskips{\sort(w_{op})},
\]
i.e.\ its set of cover reflections is the unforced skip set of $\sort(w_{op})$.
\end{lem}

\begin{proof}
By \eqref{eqn:upsort} we have
\[
\upsort^{(c)}(w^{-1})=\sort^{(c^{-1})}(w^{-1}w_{\circ})\,w_{\circ},
\]
so in particular $\upsort(w_{op})w_{\circ}\in \Sort(W,c^{-1})$.

Let $x\coloneqq \sort(w_{op})$, and let $\mathcal I$ be the $c$-Cambrian class of $w_{op}$ (equivalently of $x$).
Then $x$ is the unique $c$-sortable element in $\mathcal I$.
Write $\{\fs_1,\ldots,\fs_k\}$ and $\{\ufs_1,\ldots,\ufs_{n-k}\}$ for the forced and unforced skips of $x$.
Let $D=\RR_{\ge 0}\Lambda^{+}$ be the dominant chamber.
By \cite[Theorem~6.3]{RS11} we have
\[
\bigcup_{z\in \mathcal{I}}zD
=
\left( \bigcap_{i=1}^k\{\langle v, r(\fs_i)\rangle\le 0 \} \right)
\cap
\left(\bigcap_{i=1}^{n-k} \{\langle v,r(\ufs_i)\rangle \ge 0\}\right).
\]
Thus the data of forced versus unforced skips for $x$ are encoded by which side of each of these
$n$ facet hyperplanes the cone lies on.

Now consider the image class $\mathcal I w_{\circ}$ under the anti-automorphism $z\mapsto zw_{\circ}$.
Since $w_{\circ}D=-D$, we have
\[
\bigcup_{z\in \mathcal{I}w_{\circ}}zD
=
\bigcup_{z\in \mathcal{I}}zw_{\circ}D
=
-\bigcup_{z\in \mathcal{I}}zD.
\]
Therefore the defining inequalities for the cone $\bigcup_{z\in \mathcal{I}w_{\circ}}zD$ are obtained from
those for $\bigcup_{z\in \mathcal{I}}zD$ by reversing all inequality signs. Equivalently, in the class
$\mathcal I w_{\circ}$ the roles of ``forced'' and ``unforced'' skips are interchanged.
Thus for the unique $c^{-1}$-sortable element in the class
$\mathcal I w_{\circ}$, namely $\upsort(w_{op})w_{\circ}$, we have
\[
\cov\big(\upsort(w_{op})w_{\circ}\big)=\{\ufs_1,\ldots,\ufs_{n-k}\},
\] as claimed.
\end{proof}

\begin{prop}
\label{prop:specificcoverfactorizations}
    If $u\in \NC(W,c)^+$ and $c\le_R x\in \Sort(W,c)$ is such that $\nc_c(x)=u$ (\Cref{fact:nc_c}), then we have
    \begin{enumerate}[label=(\arabic*)]
        \item \label{it1:11.5} a cover reflection factorization  $u=t_1\cdots t_k$ is given by $t_i=\fs_i$,
        \item \label{it2:11.5} a cover reflection factorization $u^{-1}c=t_{k+1}\cdots t_n$ is given by $t_{k+i}=c^{-1}\ufs_ic$.
    \end{enumerate} 
    Here $\fs_1,\dots,\fs_k$ and $\ufs_1,\dots,\ufs_{n-k}$ are the forced and unforced skips of $x$ respectively.
\end{prop}
\begin{proof}
Part~\ref{it1:11.5} has been shown in \Cref{fact:fskip}. 
For~\ref{it2:11.5}, note that  we have $\ufs_1\cdots \ufs_{n-k}\fs_1\cdots \fs_{k}=c$ by \Cref{cor:ufsthenfs}.
Rewriting this as
\[
\fs_1\cdots \fs_k\,(c^{-1}\ufs_1c)\cdots(c^{-1}\ufs_{n-k}c)=c
\]
and using $u=\fs_1\cdots \fs_k$ yields
\[
u^{-1}c=(c^{-1}\ufs_1c)\cdots(c^{-1}\ufs_{n-k}c).
\]
Thus it remains to show that the reflections $c^{-1}\ufs_i c$ are the set of cover reflections associated to the noncrossing element $u^{-1}c$.

By \Cref{lem:covstop} applied for $w_{op}$ equal to $x$, the reflections $\ufs_1,\ldots,\ufs_{n-k}$ are the cover reflections of
\[
y\coloneqq \upsort(x)\wnaught\in \Sort(W,c^{-1}).
\] 
    We know that $c\le_R x\le_R  \upsort(x)$ so $\ell(c^{-1})+\ell(y)=\ell(c^{-1}y)$, and so $c^{-1}y\in \Sort(W,c^{-1})$, with $c^{-1}$-sorting word obtained by appending $c^{-1}$ onto the front of the $c^{-1}$-sorting word for $y$. 
    This shows that $c^{-1}\ufs_1c,\ldots, c^{-1}\ufs_{n-k}c$ are still skip reflections of $c^{-1}y$, and we claim that they are a subset of the cover reflections. 
    Indeed,
    $$(c^{-1}\ufs_i c) (c^{-1}y)=c^{-1}\ufs_iy,$$
    which has smaller length than $\ell(c^{-1}y)=n+\ell(y)$. This is because $\ell(\ufs_iy)<\ell(y)$ due to $\ufs_i$ being a cover reflection of $y$.

    Finally, since the cover reflections of a noncrossing element are the simple generators relative to the induced simple system $\Delta_{W'}$ for the minimal reflection subgroup $W'$
containing that noncrossing element (see Fact~\ref{fac:bjv_covers_are_generators}), we see that $c^{-1}y$ is the Coxeter element for the reflection subgroup $W'$ generated by $\cov(y)$, and so $u^{-1}c$ is the Coxeter element associated to the standard parabolic subgroup $W''$ of $W'$ with simple generators the reflections $c^{-1}\ufs_1c,\ldots,c^{-1}\ufs_{n-k}c$. Therefore $W''\subset W$ is the minimal reflection subgroup containing the $c$-noncrossing element $u^{-1}c$, establishing that the $c^{-1}\ufs_ic$ are exactly the cover reflections of $u^{-1}c$.
\end{proof}

\begin{eg}
\label{eg:S8-nontrivial}
Let $W=S_8$ with $
c \;=\; s_2 s_5 s_1 s_3 s_6 s_7 s_4$.
Set $x \coloneqq c^2$. 
We have that $c \le_R x$ and that $x$ is $c$--sortable.
Let $\leq_c$ denote the reflection order coming from the $c$--sorting word of $\wnaught$.
The cover reflections, ordered by $\leq_c$, are the transpositions $
(5\,8),\ (3\,4),\ (2\,5),\ (1\,7).$
Therefore
\[
u \;=\; \nc_c(x)
\;=\;
(5\,8)(3\,4)(2\,5)(1\,7)
\;=\;
78432615.
\]
In particular $u\in \NC(W,c)^+$.

The Kreweras complement of $u$ is $
u^{-1}c = 47365128$, and
the unique $c$-sortable $x'\in\Sort(W,c)$ such that $
\nc_c(x') = u^{-1}c$ is $
x' = 36417258$.
Its cover reflections are $
\cov(x')=\{(1\,4),(2\,7),(4\,6)\}.$
Ordered by the same reflection order $\leq_c$, we have
$
(1\,4)\ \leq_c\ (2\,7)\ \leq_c\ (4\,6),
$
and therefore
\[
u^{-1}c
=
(1\,4)(2\,7)(4\,6),
\]
which one can check agrees with  the conjugated unforced skips of $x$.
\end{eg}

\subsection{The Bruhat maximal element}
\label{sec:bruhatmax}

The proof of Theorem~\ref{thm:BruhatMaxIsNC} follows the next two lemmas.

\begin{lem}
\label{lem:BooleanTrick}
Fix $w \in W$ and let $d = s_{1} \cdots s_{k}$ be the product of $k$ distinct simple reflections.  
\begin{enumerate}[label=(\arabic*)]
    \item \label{it1:11.6} If $\ell(wd) = \ell(w)+k$ and $w [\idem,d]\subset [w, wd]$ then $w[\idem,d]=[w, wd]$.
    \item \label{it2:11.6} If $\ell(wd) = \ell(w)-k$ and $w [\idem,d]\subset [wd, w]$ then $w[\idem,d]=[wd, w]$.
\end{enumerate}
\end{lem}
\begin{proof}
\ref{it2:11.6} is the special case of~\ref{it1:11.6} with $w$ replaced with $ws_k\cdots s_1$. 
The hypothesis  $w[\id,d]\subseteq [w,wd]$ in~\ref{it1:11.6} is equivalent to $[\id,d]\subseteq w^{-1}[w,wd]$.
Interpreting these as $T$-fixed point sets of Richardson varieties and applying the rigidity statement
(\Cref{fact:rigidRich}) gives an inclusion
\[
X_{\id}^{d}\subseteq w^{-1}X_{w}^{wd}.
\]
Both Richardson varieties are irreducible of dimension $\ell(d)=k$, hence they are equal.
Taking $T$-fixed points yields the claim.
\end{proof}
\begin{lem}
\label{lem:covinanyorder}
    For $w \in W$ and distinct $s_1,\ldots,s_k\in \des{w}$, multiplication by $w$ gives an anti-isomorphism from the interval $[\idem,s_{1}\cdots s_{k}]$ to $[ws_{1}\cdots s_{k}, w]$.
\end{lem}
\begin{proof}
By \cite[Corollary 2.8(i)]{BjBr05} we have that if $s_i$ and $s_j$ are distinct simple reflections so that $ws_i<_B w$ and $ws_j<_B w$, then $ws_is_j<_B ws_i$ and $ws_is_j<_B ws_j$.
Applying this repeatedly implies that for all $\{i_1<i_2<\cdots <i_j\}\subseteq [k]$, we have $ws_1s_2\cdots s_k <_B ws_{i_1}\cdots s_{i_j}<_B w$.
Since by the subword property the interval $[\idem,s_1\cdots s_k]$ contains all elements of the form $s_{i_1}\cdots s_{i_j}$ where $1\leq i_1<\cdots <i_j\leq k$, the claim now follows from \Cref{lem:BooleanTrick}.
\end{proof}

\begin{proof}[Proof of Theorem~\ref{thm:BruhatMaxIsNC}]
By Theorem~\ref{thm:Equivwwc} we may assume that $w_{op} = w^{-1}\wnaught$ is $c$-sortable, so that $w_{op} = \sort(w_{op})$.  
Take $w'$ so that $(w')_{op} = \upsort(w_{op})$.  
Then by \Cref{thm:Equivwwc} we have
$$
\movetoid{w}{wc} = \movetoid{w'}{w'c}.
$$
We will prove the following three claims.  
Let $p_1,\ldots, p_k,q_1,\ldots,q_{n-k}$ be as in \Cref{lem:BianeInversion}, associated to the specific cover reflection factorizations from \Cref{prop:specificcoverfactorizations}. Then
\begin{enumerate}
    \item $u\in \movetoid{w}{wc}=\movetoid{w'}{w'c}$
    \item The elements covered by $w u$ in $[w ,w u]$ are $w up_1,\ldots,wu p_k$. 
    \item The covers of $w'u$ in $[w'u,w'c]$ are the elements $w'uq_1,\ldots, w'uq_{n-k}$. 
\end{enumerate}
Suppose that (1)-(3) are established. 
Then by Theorem~\ref{thm:Equivwwc}, the adjacent elements of $\movetoid{w}{wc}$ which are adjacent to $u\in \movetoid{w}{wc}$ are 
\[
\{up_1,\ldots,up_{k},uq_1,\ldots,uq_{n-k}\} =\{\tau u\suchthat \tau\in \invnc{u}\}.
\]
In particular, each element of $\movetoid{w}{wc}$ which is adjacent to $u$ precedes it in the Bruhat order.  
This characterizes the unique Bruhat-maximum element in any Coxeter matroid, and $\movetoid{w}{wc}$ is a Coxeter matroid, so $u$ must be this maximum, completing the proof. 

We now show (1). Let $s_1,\ldots,s_k\in \des{w_{op}}$ be the descents associated to the forced skips $\fs_1,\ldots,\fs_k$ of $w_{op}$, i.e. with $\fs_i w_{op}=w_{op} s_i$. Then as $u=\fs_1\cdots \fs_k$, by \Cref{lem:covinanyorder} applied to $w_{op}$ and the descents $s_k,\ldots,s_1$ we have 
\begin{equation}
\label{eqn:uinversexx}
w_{op}[\idem,s_k\cdots s_1]=[w_{op}s_k\cdots s_1,w_{op}]=[u^{-1} w_{op},w_{op}],
\end{equation}
and hence $u^{-1}w_{op}\le w_{op}$. 
As $c\le_R w_{op}$, we have that $c^{-1}w_{op}$ is $c$-sortable and the conjugates $c^{-1}\ufs_1c,\ldots,c^{-1}\ufs_{n-k}c$ are an ordered subset of the unforced skips of $c^{-1}w_{op}$.
We can therefore apply \Cref{prop:insertunforced} repeatedly followed by \Cref{cor:ufsthenfs} to get
\begin{align*}
c^{-1}w_{op}\lessdot_B (c^{-1}\ufs_{n-k}c)(c^{-1}w_{op})\lessdot_B  \cdots \lessdot_B \prod_{i=1}^{n-k}(c^{-1}\ufs_ic)(c^{-1}w_{op})=u^{-1}w_{op}.
\end{align*}
This shows that $c^{-1}w_{op}\le u^{-1}w_{op}\le w_{op}$, which after inverting and multiplying by $\wnaught$ on the left shows $w\le wu \le wc$ as desired.

We record for future use that, because we have now established (1), we have the sequence
\begin{equation}
\label{eqn:seqofequiv}
[c^{-1}w_{op}',u^{-1}w_{op}'] 
\xrightarrow{op}
[w'u, w'c]
\xrightarrow{\text{shape}}
[wu, wc]
\xrightarrow{op}
[c^{-1} w_{op}, u^{-1} w_{op}]
\end{equation}
between length $n-k$ Bruhat intervals, where the first and last are anti-equivalences and the second is an equivalence.

Now we show (2). 
By Proposition~\ref{prop:insertunforced}, the elements covering $w_{op} s_{k}\cdots s_{1}$ in \eqref{eqn:uinversexx} are given by $w_{op}s_{k} \cdots \hat{s_{i}} \cdots s_{1}$, where $\wh{s_i}$ denotes omission, so the elements covering $u^{-1}w_{op}$ in \eqref{eqn:uinversexx} are 
\[
w_{op}s_k\cdots \wh{s_i}\cdots s_1=\fs_k\cdots \wh{\fs_i}\cdots \fs_1 w_{op}=p_iu^{-1}w_{op}.
\]
Inverting and multiplying by $\wnaught$ on the left then shows that the elements covered by $wu$ in this interval are exactly the $wup_i$ for $i=1,\ldots,k$.

Finally, we show (3).  
The strategy is similar to (2), but with added complications to account for the reversal in order.  
Lemma~\ref{lem:covstop} states that the cover reflections of $w_{op}'\wnaught$ are  $\ufs_1,\ldots,\ufs_{n-k}$, the unforced skips of $w_{op}$. 
Let $s_1',\ldots,s_{n-k}'$ be the associated descents of $w_{op}'\wnaught$, so that $(w_{op}'\wnaught)s_i'=\ufs_i(w_{op}'\wnaught)$. 
Noting that $\ufs_1\cdots \ufs_{n-k}=cu^{-1}$ by \Cref{fact:fskip} and \Cref{cor:ufsthenfs}, we thus have by \Cref{lem:covinanyorder} that
\[
w_{op}'\wnaught[\idem,s_1'\cdots s_{n-k}']=[w_{op}'\wnaught s_1'\cdots s_{n-k}',w_{op}'\wnaught]=[(cu^{-1})w_{op}'\wnaught,w_{op}'\wnaught].
\]
Writing $s_i''=\wnaught s_i'\wnaught$, we have $s_1'',\ldots,s_i''$ are the ascents of $w_{op}'$ and  $w_{op}'s_i''=\ufs_iw_{op}'$.  Therefore
$$
w_{op}'[\idem,s_1''\cdots s_{n-k}'']=[w_{op}',w_{op}'s_1''\cdots s_{n-k}'']=[w_{op}',(cu^{-1})w_{op}'],
$$
providing an analogue of \eqref{eqn:uinversexx} which we will use to prove (3).

We claim that
$c^{-1}w_{op}'[\idem,s_1''\cdots s_{n-k}'']= [c^{-1}w_{op}',u^{-1}w_{op}']$ (note the right hand interval is the leftmost length $n-k$ Bruhat interval from \eqref{eqn:seqofequiv}). 
This will follow from Lemma~\ref{lem:BooleanTrick} if we can show that $c^{-1}w_{op}'[\idem,s_1''\cdots s_{n-k}'']\subset [c^{-1}w_{op}',u^{-1}w_{op}']$, or equivalently
$$
\left\{\prod_{j=1}^{\ell}(c^{-1}\ufs_{i_j}c)(c^{-1}w_{op}')\suchthat 1\le i_1<i_2<\cdots <i_\ell \le n-k \right\}\subset [c^{-1}w_{op}',u^{-1}w_{op}']
$$
Applying the sequence of equivalences and anti-equivalence in \eqref{eqn:seqofequiv}
to this condition, it suffices to show that each $z = \prod_{j=1}^{\ell}(c^{-1}\ufs_{i_j}c)(c^{-1}w_{op})$ is contained in $[c^{-1} w_{op}, u^{-1} w_{op}]$.  
To see this, note that by applying \Cref{prop:insertunforced} repeatedly, a reduced word for $z$ is obtained by filling in the unforced skips of $c^{-1}w_{op}$ corresponding to each $c^{-1}\ufs_{i_j}c$, so we see that $c^{-1}w_{op}$ is naturally a subword of $z$, and $u^{-1}w_{op}=\prod_{i=1}^{n-k}(c^{-1}\ufs_i c)(c^{-1}w_{op})$ has a reduced word which contains the reduced word for $z$.

Finally, by \Cref{lem:covinanyorder} the covers of $u^{-1}w_{op}'$ in $[c^{-1}w_{op}',u^{-1}w_{op}']$ are the elements 
\begin{align*}
c^{-1}w_{op}'s_1''\cdots \wh{s_i''}\cdots s_{n-k}''
&=(c^{-1}\ufs_{1}c)\cdots \wh{(c^{-1}\ufs_ic)}\cdots (c^{-1}\ufs_{n-k}c)(c^{-1}w_{op}) \\
&= (q_{i} u^{-1}c)(c^{-1}w_{op})\\
&=q_iu^{-1}w_{op}'.
\end{align*}
Inverting and multiplying by $\wnaught$ on the left shows that the covers of $w'u$ in $[w'u,w'c]$ are given by the elements $w'uq_1,\ldots, w'uq_{n-k}$ as desired.
\end{proof}

\appendix

\section{Type $A$ Examples}
\label{appendix:GL}
\subsection{The type $\mathrm{A}$ flag variety}
Our choice of $G,B,B^-,T$ for type $\mathrm{A}$ that we use in this section is $G=\GL_{n+1}$, $B,B^-$ are upper and lower triangular matrices respectively, and $T=B\cap B^-$ are the diagonal matrices.  The complete flag variety $\GL_{n+1}/B$ parametrizes complete flags of subspaces
$\{0\subsetneq V_1\subsetneq\cdots \subsetneq V_n\subsetneq \mathbb{C}^{n+1}\suchthat \dim V_i=i\}$. A coset $MB$ is determined by $M$ up to invertible forward column operations. To recover the flag from $M$ we take $V_i$ to be the span of the first $i$ columns.

We denote $\epsilon_i\in \Q{T}$ for the $i$th standard character written additively so that $\Q{T}=\bigoplus_{i=1}^{n+1} \mathbb{Z}\epsilon_i $. The Weyl group is the group $S_{n+1}=\langle s_1,\ldots,s_n\rangle$ of permutations on $\{1,\ldots,n+1\}$, where $s_i=(i,i+1)$. The root system is $\Phi=\{\epsilon_i-\epsilon_j\suchthat i\ne j\}$, and the positive roots are $\Phi^+=\{\epsilon_i-\epsilon_j\suchthat i<j\}$. 
The weight lattice is $\mathbb{Z}^{n+1}$ and the root lattice is $\{(x_1,\dots,x_{n+1})\suchthat \sum x_i=0\}\subset \mathbb{Z}^{n+1}$. 

The fundamental dominant weights are $\omega_i=\epsilon_1+\cdots+\epsilon_i$, with associated Pl\"ucker functions \[
\Pl_{w}^{\omega_i}(M)=\det M_{w(1),\ldots,w(i)},
\] 
where $M_A$ is the submatrix of $M$ with rows chosen from $A$ and columns chosen from $1,\ldots,|A|$. 
Since each $\omega_i$ lies in the character lattice $ \operatorname{Char}(T)$, we can take the regular dominant weight $\regweight=\sum \omega_i=(n,n-1,\ldots,0)$. 
The associated Pl\"ucker functions for $\regweight$ are given by $$\Pl_w(M)=\prod_{i=1}^{n}\det M_{w(1),\ldots,w(i)}.$$
\subsection{Classical presentation of Schubert cells}

Each permutation in $w\in S_{n+1}$ corresponds to the permutation matrix in $\GL_{n+1}$ with $1$'s in the entries $(w(i),i)$ and $0$ elsewhere. 
We can represent $BwB$ as a matrix with entries in $\{0,1,\ast\}$, where $\ast\in \mathbb{C}$ is a free entry (an omitted entry is by default zero). 
To do so, we take the permutation matrix associated to $w$, and replace each $0$ which has no $1$ either to the left or above it with a $\ast$.

The 6 charts $\Xc^u=BwB=U_wwB$ for $w\in S_3$ are given by
    $$\underbrace{\begin{bmatrix}1&&\\&1&\\&&1\end{bmatrix}}_{\Xc^{123}}, \underbrace{\begin{bmatrix}\ast&1&\\1&&\\&&1\end{bmatrix}}_{\Xc^{213}}, \underbrace{\begin{bmatrix}1&&\\&\ast&1\\&1&\end{bmatrix}}_{\Xc^{132}},\underbrace{\begin{bmatrix}\ast&\ast&1\\1&&\\&1&\end{bmatrix}}_{\Xc^{231}},\underbrace{\begin{bmatrix}\ast&1&\\\ast&&1\\1&&\end{bmatrix}}_{\Xc^{312}},\underbrace{\begin{bmatrix}\ast&\ast&1\\\ast&1&\\1&&\end{bmatrix}}_{\Xc^{321}}.$$
As an example, if we write $\Xc^{321}$ as
$\begin{bmatrix}c&a&1\\b&1&0\\1&0&0\end{bmatrix}$, then the Pl\"ucker functions in this chart are
$$(\Pl_{123},\Pl_{213},\Pl_{132},\Pl_{231},\Pl_{312},\Pl_{321})=(c(c-ab),b(c-ab),-ac,-b,-a,-1).$$
We see that these do not induce an injection on $\Xc^{321}$ -- for example if $c=\pm 1, a=0, b=0$.
\subsection{Examples of type $A$ Coxeter flag varieties}
For $c=s_2s_1=132$ we have
$$\underbrace{\begin{bmatrix}1&&\\&1&\\&&1\end{bmatrix}}_{\Xc_{\NC}^{123}}, \underbrace{\begin{bmatrix}\ast&1&\\1&&\\&&1\end{bmatrix}}_{\Xc_{\NC}^{213}}, \underbrace{\begin{bmatrix}1&&\\&\ast&1\\&1&\end{bmatrix}}_{\Xc_{\NC}^{132}},\qquad\qquad,\underbrace{\begin{bmatrix}\ast&1&\\\ast&&1\\1&&\end{bmatrix}}_{\Xc_{\NC}^{312}},\underbrace{\begin{bmatrix}\ast&\ast&1\\0&1&\\1&&\end{bmatrix}}_{\Xc_{\NC}^{321}}.$$
For $c=s_1s_2=231$ we have
    $$\underbrace{\begin{bmatrix}1&&\\&1&\\&&1\end{bmatrix}}_{\Xc_{\NC}^{123}}, \underbrace{\begin{bmatrix}\ast&1&\\1&&\\&&1\end{bmatrix}}_{\Xc_{\NC}^{213}}, \underbrace{\begin{bmatrix}1&&\\&\ast&1\\&1&\end{bmatrix}}_{\Xc_{\NC}^{132}},\underbrace{\begin{bmatrix}\ast&\ast&1\\1&&\\&1&\end{bmatrix}}_{\Xc_{\NC}^{231}},\qquad\qquad,\underbrace{\begin{bmatrix}\ast&0&1\\\ast&1&\\1&&\end{bmatrix}}_{\Xc_{\NC}^{321}}.$$

    For $c=s_3s_2s_1=4123$ we have
    $$\underbrace{\begin{bmatrix}1&&&\\&1&&\\&&1&\\&&&1\end{bmatrix}}_{\Xc_{\NC}^{1234}}, \underbrace{\begin{bmatrix}\ast&1&&\\1&&&\\&&1&\\&&&1\end{bmatrix}}_{\Xc_{\NC}^{2134}},\underbrace{\begin{bmatrix}1&&&\\&\ast&1&\\&1&&\\&&&1\end{bmatrix}}_{\Xc_{\NC}^{1324}},\underbrace{\begin{bmatrix}1&&&\\&1&&\\&&\ast&1\\&&1&\end{bmatrix}}_{\Xc_{\NC}^{1243}}$$
    $$\underbrace{\begin{bmatrix}\ast&1&&\\\ast&&1&\\1&&&\\&&&1\end{bmatrix}}_{\Xc_{\NC}^{3124}},\underbrace{\begin{bmatrix}1&&&\\&\ast&1&\\&\ast&&1\\&1&&\end{bmatrix}}_{\Xc_{\NC}^{1423}},\underbrace{\begin{bmatrix}\ast&1&&\\1&&&\\&&\ast&1\\&&1&\end{bmatrix}}_{\Xc_{\NC}^{2143}},\underbrace{\begin{bmatrix}\ast&\ast&1&\\0&1&&\\1&&&\\&&&1\end{bmatrix}}_{\Xc_{\NC}^{3214}},\underbrace{\begin{bmatrix}1&&&\\&\ast&\ast&1\\&0&1&\\&1&&\end{bmatrix}}_{\Xc_{\NC}^{1432}}$$
    $$\underbrace{\begin{bmatrix}\ast&1&&\\\ast&&1&\\\ast&&&1\\1&&&\end{bmatrix}}_{\Xc_{\NC}^{4123}},\underbrace{\begin{bmatrix}\ast&\ast&1&\\0&1&&\\\ast&&&1\\1&&&\end{bmatrix}}_{\Xc_{\NC}^{4213}},\underbrace{\begin{bmatrix}\ast&1&&\\\ast&&\ast&1\\0&&1&\\1&&&\end{bmatrix}}_{\Xc_{\NC}^{4132}},\underbrace{\begin{bmatrix}\ast&\ast&\ast&1\\0&1&&\\0&&1&\\1&&&\end{bmatrix}}_{\Xc_{\NC}^{4231}},\underbrace{\begin{bmatrix}\ast&\ast&0&1\\0&\ast&1&\\0&1&&\\1&&&\end{bmatrix}}_{\Xc_{\NC}^{4321}}$$
    
    For $c=s_2s_1s_3=3142$ we have
    $$\underbrace{\begin{bmatrix}1&&&\\&1&&\\&&1&\\&&&1\end{bmatrix}}_{\Xc_{\NC}^{1234}}, \underbrace{\begin{bmatrix}\ast&1&&\\1&&&\\&&1&\\&&&1\end{bmatrix}}_{\Xc_{\NC}^{2134}},\underbrace{\begin{bmatrix}1&&&\\&\ast&1&\\&1&&\\&&&1\end{bmatrix}}_{\Xc_{\NC}^{1324}},\underbrace{\begin{bmatrix}1&&&\\&1&&\\&&\ast&1\\&&1&\end{bmatrix}}_{\Xc_{\NC}^{1243}}$$
    $$\underbrace{\begin{bmatrix}\ast&1&&\\\ast&&1&\\1&&&\\&&&1\end{bmatrix}}_{\Xc_{\NC}^{3124}},\underbrace{\begin{bmatrix}1&&&\\&\ast&0&1\\&\ast&1&\\&1&&\end{bmatrix}}_{\Xc_{\NC}^{1432}},\underbrace{\begin{bmatrix}\ast&1&&\\1&&&\\&&\ast&1\\&&1&\end{bmatrix}}_{\Xc_{\NC}^{2143}},\underbrace{\begin{bmatrix}\ast&\ast&1&\\0&1&&\\1&&&\\&&&1\end{bmatrix}}_{\Xc_{\NC}^{3214}},\underbrace{\begin{bmatrix}1&&&\\&\ast&\ast&1\\&1&&\\&&1&\end{bmatrix}}_{\Xc_{\NC}^{1342}}$$
    $$\underbrace{\begin{bmatrix}\ast&1&&\\\ast&&0&1\\\ast&&1&\\1&&&\end{bmatrix}}_{\Xc_{\NC}^{4132}},\underbrace{\begin{bmatrix}\ast&\ast&0&1\\0&1&&\\\ast&&1&\\1&&&\end{bmatrix}}_{\Xc_{\NC}^{4231}},\underbrace{\begin{bmatrix}\ast&1&&\\\ast&&\ast&1\\1&&&\\&&1&\end{bmatrix}}_{\Xc_{\NC}^{3142}},\underbrace{\begin{bmatrix}\ast&\ast&\ast&1\\0&1&&\\1&&&\\&&1&\end{bmatrix}}_{\Xc_{\NC}^{3241}},\underbrace{\begin{bmatrix}\ast&\ast&1&\\0&\ast&&1\\1&&&\\&1&&\end{bmatrix}}_{\Xc_{\NC}^{3412}}$$

\section{Type $B$ Examples}
\label{appendix:SO}
\subsection{Type $\mathrm{B}$ flag variety}
We write the coordinates on $\mathbb{C}^{2n+1}$ in order as $x_{\bar{n}},\ldots,x_0,\ldots,x_n$, and fix the symmetric form $x_0^2+\sum x_ix_{\bar{i}}$. In type $\mathrm{B}$ we take $G=SO(2n+1)$, $B,B^-\subset G$ the subsets of upper and lower triangular matrices, and $T=B\cap B^-\subset G$ are the diagonal matrices.
Recall that a subspace $V\subset \mathbb{C}^{2n+1}$ is isotropic if $\langle v,w\rangle=0$ for all $v,w\in V$. For Type $\mathrm{B}$ the flag variety is $$SO(2n+1)/B=\{0\subsetneq V_1\subsetneq \cdots \subsetneq V_{n}\subsetneq \mathbb{C}^{2n+1}\suchthat V_i\text{ is isotropic and }\dim V_i=i\}.$$

For $i\in \{-n,\ldots,n\}$ we denote $\epsilon_i$ for the \emph{negative} of the standard character associated to $i$, and $\Q{T}=\bigoplus_{i=-n}^n \mathbb{Z}\epsilon_i/\langle \epsilon_{-i}+\epsilon_i\rangle $. The root system is $\Phi=\{\pm \epsilon_{j}\pm \epsilon_{i}\suchthat 1\le i<j\le n\}\cup \{\pm \epsilon_{i}\suchthat 1\le i \le n\}$, and the simple roots are $\alpha_{0} = \epsilon_{1}$ and $\alpha_{i} = -\epsilon_{i} + \epsilon_{i+1}$ for $i \in \{1, \ldots, n-1\}$.  
The associated Weyl group is $H_n$, the permutations of $\{\pm1,\ldots, \pm n\}$ with $w(-i)=-w(i)$. These are called the \emph{signed permutations}, and $W$ is called the \emph{hyperoctahedral group}. As is standard we write $\bar{i}$ for $-i$, and we write the generators $s_0=(1,\bar{1})$ and $s_i=(i,i+1)(\bar{i},\overline{i+1})$. Every permutation can be written in one-line notation as $w=w(1)w(2)\cdots w(n)$.

The rows and columns of the matrices in $SO(2n+1)$ are labeled in order by $\overline{n},\ldots,0,\ldots,n$. An element of $SO(2n+1)/B$ is determined by its restriction to the first $n$ columns (the ones labeled $\bar{n},\ldots,\bar{1}$), and a $(2n+1)\times n$ matrix represents an element of $SO(2n+1)$ if it has rank $n$, and for column vectors $[a_{\bar{n}},\ldots,a_n]^T$ and $[b_{\bar{n}},\ldots,b_n]^T$ we have $a_0b_0+\sum_{i=1}^n(a_{i}b_{\bar{i}}+a_{\bar{i}}b_i)=0$. Two matrices are equivalent if one can be obtained from the other by forward column operations, and such a matrix is determined by the flag of isotropic subspaces where $V_i$ is the span of the first $i$ columns. 

The fundamental dominant weights are 
\[
\omega_{0} = \frac{1}{2} \sum_{i = 1}^{n} \epsilon_{i}
\qquad\text{and, for $1 \le i \le n-1$, }\qquad
\omega_{i}=\epsilon_{i+1} + \cdots + \epsilon_{n}.
\]
Each of $2\omega_0,\omega_{1}, \ldots, \omega_{n-1}$ lie in the character lattice $\Q{T}$, with associated Pl\"{u}cker coordinates
\begin{equation}
\label{eq:BFundPlucker}
\det M_{\overline{w(n)},\ldots,\overline{w(i+1)}},
\end{equation}
where $M_{A}$ is the submatrix consisting of the first $|A|$ columns (indexed by $\overline{n},\ldots,\overline{n-|A|+1}$) and the rows indexed by $A$.  
Therefore if we take the regular dominant weight $\lambda_{reg}=2\omega_{0} + \omega_{1} + \cdots + \omega_{n-1}$, the associated Pl\"ucker function is
\begin{equation}
\label{eq:BRegPlucker}
\Pl_w=\prod_{i=1}^n \det M_{\overline{w(n)},\ldots,\overline{w(i+1)}}.
\end{equation}

\subsection{Classical presentation of Schubert cells}
We learned the following presentation from \cite{A18}.
The chart $BwB$ can be written as an $(2n+1)\times n$ matrix containing $0,1,\ast,\otimes$, where $0$ is often omitted from the notation, $\ast\in \mathbb{C}$ is a free variable, and $\otimes$ are variables which are the unique polynomials in the $\ast$ that ensure isotropy between the column they are in and previous columns. Concretely, we put $1$ in the entry $(w(i),i)$ for $i\in \{-n,\ldots,-1\}$, and then we put $\ast$ in all entries which are not below and not to the right of a $1$, and then finally convert all $\ast$ to $\otimes$ which are \emph{weakly} to the right of $(\overline{w(i)},i)$ for any $i\in \{-n,\ldots,-1\}$. For $n=2$ the $8$ charts $M(w)=BwB$ for $w\in H_2$, with rows labeled by $\bar{2},\bar{1},1,2$ and columns labeled by $\bar{2},\bar{1}$ are given by

$$\underbrace{\begin{bmatrix}1&\\&1\\&\\&\\&\end{bmatrix}}_{\Xc^{12}},\underbrace{\begin{bmatrix}1&\\&\otimes\\&\ast\\&1\\&\end{bmatrix}}_{\Xc^{\bar{1}2}},\underbrace{\begin{bmatrix}\ast&1\\1&\\&\\&\\&\end{bmatrix}}_{\Xc^{21}},\underbrace{\begin{bmatrix}\ast&1\\\otimes&\\\ast&\\1&\\&\end{bmatrix}}_{\Xc^{2\bar{1}}},\underbrace{\begin{bmatrix}\ast&\otimes\\1&\\&\ast\\&\otimes\\&1\end{bmatrix}}_{\Xc^{\bar{2}1}},\underbrace{\begin{bmatrix}\otimes&\otimes\\\ast&1\\\ast&\\\ast&\\1&\end{bmatrix}}_{\Xc^{1\bar{2}}},\underbrace{\begin{bmatrix}\ast&\otimes\\\otimes&\otimes\\\ast&\ast\\1&\\&1\end{bmatrix}}_{\Xc^{\bar{2}\bar{1}}},\underbrace{\begin{bmatrix}\otimes&\otimes\\\ast&\otimes\\\ast&\ast\\\ast&1\\1&\end{bmatrix}}_{\Xc^{\bar{1}\bar{2}}}.$$
\subsection{Examples of type $B$ Coxeter flag varieties}
For the Coxeter $s_1s_0=\bar{2}1$ we have the charts
$$\underbrace{\begin{bmatrix}1&\\&1\\&\\&\\&\end{bmatrix}}_{\Xc_{\NC}^{12}},\underbrace{\begin{bmatrix}1&\\&\otimes\\&\ast\\&1\\&\end{bmatrix}}_{\Xc_{\NC}^{\bar{1}2}},\underbrace{\begin{bmatrix}\ast&1\\1&\\&\\&\\&\end{bmatrix}}_{\Xc_{\NC}^{21}},\quad\quad,\underbrace{\begin{bmatrix}\ast&\otimes\\1&\\&\ast\\&\otimes\\&1\end{bmatrix}}_{\Xc_{\NC}^{\bar{2}1}},\underbrace{\begin{bmatrix}\otimes&\otimes\\\ast&1\\\ast&\\0&\\1&\end{bmatrix}}_{\Xc_{\NC}^{1\bar{2}}},\underbrace{\begin{bmatrix}\ast&\otimes\\\otimes&\otimes\\\ast&0\\1&\\&1\end{bmatrix}}_{\Xc_{\NC}^{\bar{2}\bar{1}}},\quad\quad.$$
For the Coxeter $s_2s_1s_0=\bar{3}12$ we have the $20$ charts
$$\underbrace{\begin{bmatrix}1&&\\&1&\\&&1\\&&\\&&\\&&\\&&\end{bmatrix}}_{\Xc_{\NC}^{123}},\underbrace{\begin{bmatrix}1&&\\&\ast&1\\&1&\\&&\\&&\\&&\\&&\end{bmatrix}}_{\Xc_{\NC}^{213}},\underbrace{\begin{bmatrix}\ast&1&\\1&&\\&&1\\&&\\&&\\&&\\&&\end{bmatrix}}_{\Xc_{\NC}^{132}},\underbrace{\begin{bmatrix}\ast&0&1\\\ast&1&\\1&&\\&&\\&&\\&&\\&&\end{bmatrix}}_{\Xc_{\NC}^{321}},\underbrace{\begin{bmatrix}1&&\\&1&\\&&\otimes\\&&\ast\\&&1\\&&\\&&\end{bmatrix}}_{\Xc_{\NC}^{\bar{1}23}},\underbrace{\begin{bmatrix}1&&\\&\otimes&\otimes\\&\ast&1\\&\ast&\\&0&\\&1&\\&&\end{bmatrix}}_{\Xc_{\NC}^{1\bar{2}3}},\underbrace{\begin{bmatrix}\otimes&&\\\ast&1&\\\ast&&1\\\ast&&\\0&&\\0&&\\1&&\end{bmatrix}}_{\Xc_{\NC}^{12\bar{3}}},\underbrace{\begin{bmatrix}1&&\\&\ast&\otimes\\&\otimes&\otimes\\&\ast&0\\&1&\\&&1\\&&\end{bmatrix}}_{\Xc_{\NC}^{\bar{2}\bar{1}3}}$$
$$\underbrace{\begin{bmatrix}\ast&\otimes&\otimes\\\otimes&\otimes&\otimes\\0&\ast&1\\0&0&\\0&0&\\1&&\\&1&\end{bmatrix}}_{\Xc_{\NC}^{1\bar{3}\bar{2}}},\underbrace{\begin{bmatrix}\ast&0&\otimes\\\ast&1&\\\otimes&&\otimes\\\ast&&0\\1&&\\&&\otimes\\&&1\end{bmatrix}}_{\Xc_{\NC}^{\bar{3}2\bar{1}}},\underbrace{\begin{bmatrix}\ast&0&\otimes\\\otimes&\otimes&\otimes\\0&1&\\\ast&&0\\\ast&&\otimes\\1&&\\&&1\end{bmatrix}}_{\Xc_{\NC}^{\bar{3}1\bar{2}}},\underbrace{\begin{bmatrix}\ast&\ast&1\\1&&\\&1&\\&&\\&&\\&&\\&&\end{bmatrix}}_{\Xc_{\NC}^{312}},\underbrace{\begin{bmatrix}\ast&\ast&\otimes\\1&&\\&\otimes&\otimes\\&\ast&0\\&1&\\&&\otimes\\&&1\end{bmatrix}}_{\Xc_{\NC}^{\bar{3}\bar{1}2}},\underbrace{\begin{bmatrix}\ast&1&\\1&&\\&&\otimes\\&&\ast\\&&1\\&&\\&&\end{bmatrix}}_{\Xc_{\NC}^{\bar{1}32}},\underbrace{\begin{bmatrix}\ast&0&\otimes\\\ast&\otimes&\otimes\\\otimes&\otimes&\otimes\\0&\ast&0\\1&&\\&1&\\&&1\end{bmatrix}}_{\Xc_{\NC}^{\bar{3}\bar{2}\bar{1}}},
$$
$$
\underbrace{\begin{bmatrix}\otimes&\otimes&\otimes\\\ast&\ast&1\\\ast&1&\\\ast&&\\0&&\\0&&\\1&&\end{bmatrix}}_{\Xc_{\NC}^{21\bar{3}}}, 
\underbrace{\begin{bmatrix}\ast&\otimes&\otimes\\1&&\\&\ast&1\\&\ast&\\&0&\\&\otimes&\\&1&\end{bmatrix}}_{\Xc_{\NC}^{1\bar{3}2}},
\underbrace{\begin{bmatrix}\ast&0&\otimes\\\ast&1&\\1&&\\&&\ast\\&&\otimes\\&&\otimes\\&&1\end{bmatrix}}_{\Xc_{\NC}^{\bar{3}21}},
\underbrace{\begin{bmatrix}1&&\\&\ast&\otimes\\&1&\\&&\ast\\&&\otimes\\&&1\\&&\end{bmatrix}}_{\Xc_{\NC}^{\bar{2}13}},
\underbrace{\begin{bmatrix}\ast&\ast&\otimes\\1&&\\&1&\\&&\ast\\&&\otimes\\&&\otimes\\&&1\end{bmatrix}}_{\Xc_{\NC}^{\bar{3}12}}.
$$

\section{Type $C$ Examples}
\label{appendix:C}
\subsection{The type $\mathrm{C}$ flag variety}
We write the coordinates on $\mathbb{C}^{2n}$ in order as $x_{\overline{n}},\ldots,x_{\overline{1}},x_1,\ldots,x_n$, and fix the symplectic form $\sum dx_i\wedge dx_{\bar{i}}$. In type $\mathrm{C}$ we take $G=Sp(2n)$, $B,B^-\subset G$ the subsets of upper and lower triangular matrices, and $T=B\cap B^-\subset G$ are the diagonal matrices.

Recall that a subspace $V\subset \mathbb{C}^{2n}$ is isotropic if $\langle v,w\rangle=0$ for all $v,w\in V$. For type $\mathrm{C}$ the flag variety is $$Sp(2n)/B=\{0\subsetneq V_1\subsetneq \cdots \subsetneq V_n\subsetneq \mathbb{C}^{2n}\suchthat V_i\text{ is isotropic and }\dim V_i=i\}.$$
For $i\in \{-n,\ldots,n\}$ we denote $\epsilon_i$ for the \emph{negative} of the standard character associated to $i$, and $\Q{T}=\bigoplus_{i=-n}^n \mathbb{Z}\epsilon_i/(\epsilon_{-i}+\epsilon_i)$. The root system is $\Phi=\{\pm \epsilon_i-\pm \epsilon_j\suchthat 1\le i<j\le n\}\cup \{\pm 2\epsilon_i\suchthat 1\le i \le n\}$, and the simple roots are $\alpha_{0} = 2\epsilon_{0}$ and $\alpha_{i} = -\epsilon_{i} + \epsilon_{i+1}$ for $i \in \{2, \ldots, n-1\}$.
The Weyl group is $H_n$ like for type $\mathrm{B}$.

The rows and columns of the matrices in $Sp_{2n}$ are labeled in order $\bar{n},\ldots,\bar{1},1,\ldots,n$. An element of $Sp_{2n}/B$ is determined by its restriction to the first $n$ columns, and a $2n\times n$ matrix represents an element of $Sp_{2n}$ if it is rank $n$, and for column vectors $[a_{\bar{n}},\ldots,a_{\bar{1}},a_1,\ldots,a_n]^T$ and $[b_{\bar{n}},\ldots,b_{\bar{1}},b_1,\ldots,b_n]^T$ we have $\sum_{j=1}^n a_jb_{\bar{j}}=\sum_{j=1}^na_{\bar{j}}b_j$.  Two matrices are equivalent if one can be obtained from the other by forward column operations, and such a matrix is determined by the flag of isotropic subspaces where $V_i$ is the span of the first $i$ columns. 
The fundamental weights are given by $\omega_{i} = \epsilon_{i+1} + \cdots + \epsilon_{n}$ for each $0 \le i \le n-1$, so the fundamental Pl\"ucker coordinates are given by the determinantal expressions in~\eqref{eq:BFundPlucker} for all $i$, and with $\regweight = \sum_{i=0}^{n-1} \omega_{i}$, the $\regweight$-Pl\"{u}cker coordinate is given by~\eqref{eq:BRegPlucker}.

\subsection{Classical presentation of Schubert cells}We learned the following presentation from \cite{A18}.
The chart $BwB$ can be written as an $(2n)\times n$ matrix containing $0,1,\ast,\otimes$, where $0$ is often omitted from the notation, $\ast\in \mathbb{C}$ is a free variable, and $\otimes$ are variables which are the unique polynomials in the $\ast$ that ensure isotropy between the columns. Concretely, we put $1$ in the entry $(w(i),i)$ for $i\in \{-n,\ldots,-1\}$, and then we put $\ast$ in all entries which are not below and not to the right of a $1$, and then finally convert all $\ast$ to $\otimes$ which are \emph{strictly} to the right of $(\overline{w(i)},i)$ for any $i\in \{-n,\ldots,-1\}$.

    For $n=2$ the $8$ charts $M(w)=BwB$ for $w\in H_2$, with rows labeled by $\bar{2},\bar{1},1,2$ and columns labeled by $\bar{2},\bar{1}$ are given by
    $$\underbrace{\begin{bmatrix}1&\\&1\\&\\&\end{bmatrix}}_{\Xc^{12}},\underbrace{\begin{bmatrix}1&\\&\ast\\&1\\&\end{bmatrix}}_{\Xc^{\bar{1}2}},\underbrace{\begin{bmatrix}\ast&1\\1&\\&\\&\end{bmatrix}}_{\Xc^{21}},\underbrace{\begin{bmatrix}\ast&1\\\ast&\\1&\\&\end{bmatrix}}_{\Xc^{2\bar{1}}},\underbrace{\begin{bmatrix}\ast&\ast\\1&\\&\otimes\\&1\end{bmatrix}}_{\Xc^{\bar{2}1}},\underbrace{\begin{bmatrix}\ast&\otimes\\\ast&1\\\ast&\\1&\end{bmatrix}}_{\Xc^{1\bar{2}}},\underbrace{\begin{bmatrix}\ast&\ast\\\ast&\otimes\\1&\\&1\end{bmatrix}}_{\Xc^{\bar{2}\bar{1}}},\underbrace{\begin{bmatrix}\ast&\ast\\\ast&\otimes\\\ast&1\\1&\end{bmatrix}}_{\Xc^{\bar{1}\bar{2}}}.$$
    For the Coxeter $s_1s_0=\bar{2}1$ we have the charts
    $$\underbrace{\begin{bmatrix}1&\\&1\\&\\&\end{bmatrix}}_{\Xc_{\NC}^{12}},\underbrace{\begin{bmatrix}1&\\&\ast\\&1\\&\end{bmatrix}}_{\Xc_{\NC}^{\bar{1}2}},\underbrace{\begin{bmatrix}\ast&1\\1&\\&\\&\end{bmatrix}}_{\Xc_{\NC}^{21}},\quad\quad,\underbrace{\begin{bmatrix}\ast&\ast\\1&\\&\otimes\\&1\end{bmatrix}}_{\Xc_{\NC}^{\bar{2}1}},\underbrace{\begin{bmatrix}\ast&\otimes\\\ast&1\\0&\\1&\end{bmatrix}}_{\Xc_{\NC}^{1\bar{2}}},\underbrace{\begin{bmatrix}\ast&\\\ast&\otimes\\1&\\&1\end{bmatrix}}_{\Xc_{\NC}^{\bar{2}\bar{1}}},\quad\quad.$$
For the Coxeter $s_2s_1s_0$ we have the $20$ charts
$$\underbrace{\begin{bmatrix}1&&\\&1&\\&&1\\&&\\&&\\&&\end{bmatrix}}_{\Xc_{\NC}^{123}},\underbrace{\begin{bmatrix}1&&\\&\ast&1\\&1&\\&&\\&&\\&&\end{bmatrix}}_{\Xc_{\NC}^{213}},\underbrace{\begin{bmatrix}\ast&1&\\1&&\\&&1\\&&\\&&\\&&\end{bmatrix}}_{\Xc_{\NC}^{132}},\underbrace{\begin{bmatrix}\ast&0&1\\\ast&1&\\1&&\\&&\\&&\\&&\end{bmatrix}}_{\Xc_{\NC}^{321}},\underbrace{\begin{bmatrix}1&&\\&1&\\&&\ast\\&&1\\&&\\&&\end{bmatrix}}_{\Xc_{\NC}^{\bar{1}23}},
\underbrace{\begin{bmatrix}1&&\\&\ast&\otimes\\&\ast&1\\&0&\\&1&\\&&\end{bmatrix}}_{\Xc_{\NC}^{1\bar{2}3}},
\underbrace{\begin{bmatrix}\ast&&\\\ast&1&\\\ast&&1\\0&&\\0&&\\1&&\end{bmatrix}}_{\Xc_{\NC}^{12\bar{3}}},
$$
$$
\underbrace{\begin{bmatrix}1&&\\&\ast&0\\&\ast&\otimes\\&1&\\&&1\\&&\end{bmatrix}}_{\Xc_{\NC}^{\bar{2}\bar{1}3}},
\underbrace{\begin{bmatrix}\ast&0&\otimes\\0&\otimes&\otimes\\0&\ast&1\\0&0&\\1&&\\&1&\end{bmatrix}}_{\Xc_{\NC}^{1\bar{3}\bar{2}}},
\underbrace{\begin{bmatrix}\ast&0&0\\\ast&1&\\\ast&&\otimes\\1&&\\&&\otimes\\&&1\end{bmatrix}}_{\Xc_{\NC}^{\bar{3}2\bar{1}}},
\underbrace{\begin{bmatrix}\ast&0&0\\\ast&\otimes&\otimes\\0&1&\\\ast&&\otimes\\1&&\\&&1\end{bmatrix}}_{\Xc_{\NC}^{\bar{3}1\bar{2}}},
\underbrace{\begin{bmatrix}\ast&\ast&1\\1&&\\&1&\\&&\\&&\\&&\end{bmatrix}}_{\Xc_{\NC}^{312}},
\underbrace{\begin{bmatrix}\ast&\ast&0\\1&&\\&\ast&\otimes\\&1&\\&&\otimes\\&&1\end{bmatrix}}_{\Xc_{\NC}^{\bar{3}\bar{1}2}},
$$
$$
\underbrace{\begin{bmatrix}\ast&1&\\1&&\\&&\ast\\&&1\\&&\\&&\end{bmatrix}}_{\Xc_{\NC}^{\bar{1}32}},\underbrace{\begin{bmatrix}\ast&0&0\\\ast&\ast&\otimes\\0&\otimes&\otimes\\1&&\\&1&\\&&1\end{bmatrix}}_{\Xc_{\NC}^{\bar{3}\bar{2}\bar{1}}},
\underbrace{\begin{bmatrix}\ast&\otimes&\otimes\\\ast&\ast&1\\\ast&1&\\0&&\\0&&\\1&&\end{bmatrix}}_{\Xc_{\NC}^{21\bar{3}}},
\underbrace{\begin{bmatrix}\ast&\ast&\otimes\\1&&\\&\ast&1\\&0&\\&\otimes&\\&1&\end{bmatrix}}_{\Xc_{\NC}^{1\bar{3}2}},\underbrace{\begin{bmatrix}\ast&0&\ast\\\ast&1&\\1&&\\&&\otimes\\&&\otimes\\&&1\end{bmatrix}}_{\Xc_{\NC}^{\bar{3}21}},\underbrace{\begin{bmatrix}1&&\\&\ast&\ast\\&1&\\&&\otimes\\&&1\\&&\end{bmatrix}}_{\Xc_{\NC}^{\bar{2}13}},\underbrace{\begin{bmatrix}\ast&\ast&\ast\\1&&\\&1&\\&&\otimes\\&&\otimes\\&&1\end{bmatrix}}_{\Xc_{\NC}^{\bar{3}12}}.
$$

\section{Type $D$ Examples}
\label{appendix:D}
\subsection{The type $\mathrm{D}$ flag variety}
We write the coordinates on $\mathbb{C}^{2n}$ in order as $x_{\bar{n}},\ldots,x_{\bar{1}},x_1,\ldots,x_n$, and fix the symmetric form $\sum x_ix_{\bar{i}}$. In type $\mathrm{D}$ we take $G=SO(2n)$, $B,B^-$ the subsets of upper and lower triangular matrices, and $T=B\cap B^-\subset G$ are the diagonal matrices.

Recall that a subspace $V\subset \mathbb{C}^{2n}$ is isotropic if $\langle v,w\rangle=0$ for all $v,w\in V$. For type $\mathrm{D}$ the flag variety is
$$SO(2n)/B=\{0\subsetneq V_1\subsetneq \cdots \subsetneq V_n\subsetneq \mathbb{C}^{2n}\suchthat V_i\text{ is isotropic and }\dim V_i=i\}.$$
For $i\in \{-n,\ldots,n\}$ we denote $\epsilon_i$ for the \emph{negative} of the standard character associated to $i$, and $\Q{T}$ is the index $2$ sublattice in the type $\mathrm{B}_{n}$ consisting of elements with even coordinate sum lattice. The root system is $\Phi=\{\pm \epsilon_i-\pm \epsilon_j\suchthat 1\le i<j\le n\}$, and the simple roots are $\alpha_{\hat{1}} = \epsilon_{1} + \epsilon_{2}$ and $\alpha_{i} = -\epsilon_{i} + \epsilon_{i+1}$ for $i \in \{1, \ldots, n-1\}$.  
The associated Weyl group is $H_n^{even}=\langle s_{\hat{1}},s_1,\ldots,s_{n-1}\rangle$, where $s_{\hat{1}}=s_0s_1s_0$.  This is the set of permutations $w \in H_{n}$ with $\#\{w(1),\ldots,w(n)\}\cap \{1,\ldots,n\}$ even.

The rows and columns of the matrices in $SO(2n)$ are labeled in order by $\overline{n},\ldots,\overline{1},1,\ldots,n$. An element of $SO(2n)/B$ is determined by its restriction to the first $n$ columns (the ones labeled $\bar{n},\ldots,\bar{1}$), and a $2n\times n$ matrix represents an element of $SO(2n)$ if it has rank $n$, and for column vectors $[a_{\bar{n}},\ldots,a_n]^T$ and $[b_{\bar{n}},\ldots,b_n]^T$ we have $\sum_{i=1}^n(a_{i}b_{\bar{i}}+a_{\bar{i}}b_i)=0$. Two matrices are equivalent if one can be obtained from the other by forward column operations, and such a matrix is determined by the flag of isotropic subspaces where $V_i$ is the span of the first $i$ columns.

The fundamental dominant weights are 
\[
\begin{array}{rl}
\omega_{\hat{1}} &\hspace{-0.9em}= \frac{1}{2}(\epsilon_{1} + \epsilon_{2} + \cdots + \epsilon_{n}),\\
\omega_{1} &\hspace{-0.9em}= \frac{1}{2}(-\epsilon_{1} + \epsilon_{2} + \cdots + \epsilon_{n}), \\
\end{array}
\qquad\text{and, for $2 \le i \le n-1$,}\qquad
\omega_i= \epsilon_{i+1} + \cdots + \epsilon_{n}
\]
Each of $\omega_1+\omega_{\hat{1}},\omega_{2}, \ldots, \omega_{n-1}$ lie in the character lattice $\Q{T}$, with associated  Pl\"ucker coordinates given by the determinantal expressions in~\eqref{eq:BFundPlucker} for $i = 1, 2, \ldots, n-1$, and with $\regweight = (\omega_1+\omega_{\hat{1}})+\sum_{i=2}^{n-1} \omega_{i}$, the $\regweight$-Pl\"{u}cker coordinate is given by
\[
\Pl_w=\prod_{i=1}^{n-1} \det M_{\overline{w(n)},\ldots,\overline{w(i+1)}}.
\]

\subsection{Classical presentation of Schubert cells}
We adapt the presentations in types $B$ and $C$ we learned from \cite{A18} to type $D$.
The chart $BwB$ can be written as an $2n\times n$ matrix containing $0,1,\ast,\otimes$, where $0$ is often omitted from the notation, $\ast\in \mathbb{C}$ is a free variable, and $\otimes$ are variables which are the unique polynomials in the $\ast$ that ensure isotropy between the column they are in and previous columns. Concretely, we put $1$ in the entry $(w(i),i)$ for $i\in \{-n,\ldots,-1\}$, and then we put $\ast$ in all entries which are not below and not to the right of a $1$, and then finally convert all $\ast$ to $\otimes$ which are \emph{weakly} to the right of $(\overline{w(i)},i)$ for any $i\in \{-n,\ldots,-1\}$. For $n=2$ the $4$ charts $M(w)=BwB$ for $w\in H_2^{even}$, with rows labeled by $\bar{2},\bar{1},0,1,2$ and columns labeled by $\bar{2},\bar{1}$ are given by

$$\underbrace{\begin{bmatrix}1&\\&1\\&\\&\end{bmatrix}}_{\Xc_{\NC}^{12}},\underbrace{\begin{bmatrix}\ast&1\\1&\\&\\&\end{bmatrix}}_{\Xc_{\NC}^{21}},\underbrace{\begin{bmatrix}\ast&\otimes\\\otimes&\otimes\\1&\\&1\end{bmatrix}}_{\Xc_{\NC}^{\bar{2}\bar{1}}},\underbrace{\begin{bmatrix}\otimes&\otimes\\\ast&\otimes\\\ast&1\\1&\end{bmatrix}}_{\Xc_{\NC}^{\bar{1}\bar{2}}}.$$

For $c=s_2s_1s_{\wh{1}}=\bar{1}\bar{3}2$ we have 
$$\underbrace{\begin{bmatrix}1&&\\&1&\\&&1\\&&\\&&\\&&\end{bmatrix}}_{\Xc_{\NC}^{123}},\underbrace{\begin{bmatrix}1&&\\&\ast&1\\&1&\\&&\\&&\\&&\end{bmatrix}}_{\Xc_{\NC}^{213}},\underbrace{\begin{bmatrix}\ast&1&\\1&&\\&&1\\&&\\&&\\&&\end{bmatrix}}_{\Xc_{\NC}^{132}},\underbrace{\begin{bmatrix}\ast&0&1\\\ast&1&\\1&&\\&&\\&&\\&&\end{bmatrix}}_{\Xc_{\NC}^{321}},\underbrace{\begin{bmatrix}1&&\\&\ast&\otimes\\&\otimes&\otimes\\&1&\\&&1\\&&\end{bmatrix}}_{\Xc_{\NC}^{\bar{2}\bar{1}3}},\underbrace{\begin{bmatrix}\ast&\otimes&\otimes\\\otimes&\otimes&\otimes\\\ast&0&1\\\ast&0&\\1&&\\&1&\end{bmatrix}}_{\Xc_{\NC}^{1\bar{3}\bar{2}}},\underbrace{\begin{bmatrix}\ast&0&\otimes\\\ast&1&\\\otimes&&\otimes\\1&&\\&&\ast\\&&1\end{bmatrix}}_{\Xc_{\NC}^{\bar{3}2\bar{1}}}$$
$$\underbrace{\begin{bmatrix}\ast&\ast&\otimes\\1&&\\&\otimes&\otimes\\&1&\\&&\otimes\\&&1\end{bmatrix}}_{\Xc_{\NC}^{\bar{3}\bar{1}2}},\underbrace{\begin{bmatrix}\otimes&\otimes&\otimes\\\ast&1&\\\ast&&\otimes\\\ast&&1\\0&&\\1&&\end{bmatrix}}_{\Xc_{\NC}^{\bar{1}2\bar{3}}},\underbrace{\begin{bmatrix}\ast&\otimes&\otimes\\\ast&\ast&1\\\otimes&\otimes&\\1&&\\&0&\\&1&\end{bmatrix}}_{\Xc_{\NC}^{2\bar{3}\bar{1}}},\underbrace{\begin{bmatrix}\ast&\ast&1\\1&&\\&1&\\&&\\&&\\&&\end{bmatrix}}_{\Xc_{\NC}^{312}},\underbrace{\begin{bmatrix}1&&\\&\otimes&\otimes\\&\ast&\otimes\\&\ast&1\\&1&\\&&\end{bmatrix}}_{\Xc_{\NC}^{\bar{1}\bar{2}3}},\underbrace{\begin{bmatrix}\ast&\otimes&\otimes\\\ast&\ast&\otimes\\1&&\\&\otimes&\otimes\\&0&1\\&1&\end{bmatrix}}_{\Xc_{\NC}^{\bar{2}\bar{3}1}},\underbrace{\begin{bmatrix}\ast&\otimes&\otimes\\1&&\\&\ast&\otimes\\&\ast&1\\&\otimes&\\&1&\end{bmatrix}}_{\Xc_{\NC}^{\bar{1}\bar{3}2}}$$

\bibliographystyle{hplain}
\bibliography{main.bib}
\end{document}